\documentclass[11pt]{amsart}
\usepackage{amsthm}
\usepackage{amssymb}
\usepackage{mathrsfs}
\usepackage{tikz}
\usepackage{fullpage}
\usepackage{color}
\usepackage{enumitem}
\usepackage{bbm}
\usepackage{bm}
\usepackage{verbatim}
\usepackage{todonotes}
\usepackage[
	colorlinks,
	linkcolor={black},
	citecolor={black},
	urlcolor={black}
]{hyperref}
\usepackage{cite}

\author[J. Fillman]{Jake Fillman}
\email{\href{mailto:fillman@tamu.edu}{fillman@tamu.edu}}
\address{J. Fillman: Department of Mathematics, Texas A\&M University, College Station, TX  77843-3368}

\author[L. Li]{Long Li}
\email{\href{mailto:ll106@rice.edu}{longli@rice.edu}}
\address{L. Li: Department of Mathematics, Rice University, Houston, TX 77005, USA}

\author[M. Luki\'{c}]{Milivoje Luki\'{c}}
\email{\href{mailto:milivoje.lukic@emory.edu}{milivoje.lukic@emory.edu}}
\address{M. Luki\'c: Emory University, Atlanta, GA 30322, USA}

\author[Q. Zhou]{Qi Zhou}
\email{\href{qizhou@nankai.edu.cn}{qizhou@nankai.edu.cn}}
\address{Q. Zhou:
Chern Institute of Mathematics and LPMC, Nankai University, Tianjin 300071, China
}

\newcommand{\bbC}{{\mathbb{C}}}
\newcommand{\bbD}{{\mathbb{D}}}
\newcommand{\bbN}{{\mathbb{N}}}
\newcommand{\bbR}{{\mathbb{R}}}
\newcommand{\bbT}{{\mathbb{T}}}
\newcommand{\bbZ}{{\mathbb{Z}}}
\newcommand{\bbP}{{\mathbb{P}}}

\newcommand{\sfE}{{\mathsf{E}}}

\newcommand{\rmd}{{\mathrm{d}}}

\newcommand{\SU}{{\mathrm{SU}}}
\newcommand{\SL}{{\mathrm{SL}}}
\newcommand{\imaginary}{\operatorname{Im}}
\newcommand{\real}{\operatorname{Re}}
\renewcommand{\Im}{\imaginary}
\renewcommand{\Re}{\real}
\newcommand{\DC}{\operatorname{DC}}
\newcommand{\su}{\mathrm{su}}
\newcommand{\essell}{\mathrm{sl}}
\newcommand{\stab}{\mathcal{E}^+}
\newcommand{\unstab}{\mathcal{E}^-}
\newcommand{\bistab}{\mathcal{E}^\pm}

\newcommand{\transfermat}{{\mathbf{T}}}

  \DeclareMathOperator\bohrmean{Av}

\allowdisplaybreaks

\newtheorem{theorem}{Theorem}[section]
\newtheorem{prop}[theorem]{Proposition}
\newtheorem{coro}[theorem]{Corollary}
\newtheorem{lemma}[theorem]{Lemma}

\newtheorem{ex}[theorem]{Example}
\newtheorem{claim}{Claim}

\theoremstyle{definition}
\newtheorem{definition}[theorem]{Definition}
\newtheorem{remark}[theorem]{Remark}

\definecolor{purple}{rgb}{.5,0,1}
\definecolor{orange}{rgb}{1,.5,0}
\definecolor{green}{rgb}{0,.5,0}

\newcommand{\ac}{{\rm ac}}
\DeclareMathOperator{\sgn}{sgn}

\numberwithin{equation}{section}
\newcommand{\ol}{\overline}

\title[Cubic Defocusing NLS]{Almost Periodic Solutions of the Cubic Defocusing Nonlinear Schr\"odinger Equation}

\begin{document}

\begin{abstract}
This paper addresses the Cauchy problem for the cubic defocusing nonlinear Schr\"odinger equation (NLS) with almost periodic initial data.  We prove that for small analytic quasiperiodic initial data satisfying Diophantine frequency conditions, the Cauchy problem admits a solution that is almost periodic in both space and time, and that this solution is unique among solutions locally bounded in a suitable sense. This establishes a conjecture by P. Deift in the literature of defocusing NLS for this class of initial data. The analysis combines direct and inverse spectral theory.

In the inverse spectral theory part, we prove existence, almost periodicity, and uniqueness for solutions with initial data whose associated Dirac operator has purely a.c.\ spectrum that is not too thin. This resolves novel challenges presented by the NLS hierarchy, such as an additional degree of freedom and an additional commuting flow.

In the direct spectral theory part, for Dirac operators with small analytic quasiperiodic potentials with Diophantine frequency conditions, we prove pure a.c.\ spectrum, exponentially decaying spectral gaps, and spectral thickness conditions (homogeneity and Craig-type conditions).
\end{abstract}

\maketitle

\setcounter{tocdepth}{1}

\tableofcontents

\hypersetup{
	linkcolor={black!30!blue},
	citecolor={red},
	urlcolor={black!30!blue}
}

\section{Introduction} \label{sec:intro}

In this paper, we study the Cauchy problem for the cubic defocusing nonlinear Schr\"odinger equation (NLS) 
\begin{equation}\label{eq:nls}
i u_t
=-\partial_x^2u+2|u|^2 u, \quad u(x,0)=\varphi(x).
\end{equation} 
The NLS is a fundamental mathematical model in physics with a wide range of applications. For example, 
it describes the propagation of light in nonlinear media, such as optical fibers \cite{Agrawal2013book}
and governs Bose-Einstein condensates \cite {PitaevskiiStringari2003book, KFC2008book} and other systems where particle interactions are significant \cite{AblowitzSegur1981}.
It can also model surface and internal water waves in fluid dynamics \cite{Mei199}.  

The well-posedness of nonlinear Schr\"odinger equations has been the focus of the study of nonlinear PDEs for a long time. 
Tsutsumi \cite{Tsutsumi1987} proved global well-posedness in $L^2(\bbR)$, and Bourgain \cite{Bourgain1999JAMS} proved global well-posedness in $H^1(\bbR^n)$ for $n=3,4$  in the radially symmetric case. 
Colliander--Keel--Staffilani--Takaoka--Tao \cite{CKSTT2004CPAM} established global existence and scattering for the defocusing cubic nonlinear
Schr\"odinger equation in $H^s(\bbR^3)$ for $s>\frac{4}{5}$, while
Griffiths--Killip--Vi\c{s}an \cite{GriffithsKillipVisan2024ForumPi} proves (for both focusing and defocusing) sharp global well-posedness in $H^s(\bbR)$ for $s>-\frac{1}{2}$. Dodson--Soffer--Spencer \cite{DSS2021JMP} established global well-posedness for the one-dimensional NLS with sufficiently smooth initial data lying in $L^p(\bbR)$ for any $2<p<\infty$, extending their previous work on local existence \cite{DSS2020JSP}.  Bourgain \cite{Bourgain93GAFA1, Bourgain93GAFA2} proved local and global well-posedness in $H^s(\bbR^p/\bbZ^p)$ for various optimal choice of $s$ depending on the dimension and the power of the nonlinearity. Applications of Bourgain's select Strichartz estimates can give global well-posedness in $L^2(\bbT)$.  In \cite{OhWang2020JDE}, Oh-Wang 
 proved the global well-posedness for the cubic NLS in $M^{2,p}(\bbR),p<\infty$ and for the normalized NLS in $\mathcal{F}L^p(\bbT), p<\infty$; compare also \cite{OhWang2021JAM} for an approach of normal form.  Well-posedness for some special initial data can be found in \cite{Christ2007, GH08JMA,Tsugawa12}, and Damanik--Li--Xu \cite{DLX2024preprint1, DLX2024preprint2} have proved certain local existence and uniqueness statements.

The one-dimensional defocusing cubic NLS is integrable: its Lax pair was found by Zakharov--Shabat \cite{ZS1972,GrebertKappeler2014}, shortly after the discovery of the Lax pair representation of the KdV equation. These discoveries started the field of integrable PDEs and motivated new questions about the long time behavior of solutions, even for non-decaying, non-periodic initial data. For almost periodic initial data, there is no general well-posedness theory, since the conserved quantities are spatial averages which don't control the solution (see \cite{DeiftOpenProblem2008}); nonetheless, integrability provides a powerful way to study existence, uniqueness, and long time behavior of solutions for almost periodic initial data. One central question has been whether almost periodicity of the initial data leads to almost periodicity in time. This question was popularized as Deift's conjecture \cite{DeiftOpenProblem2008,DeiftOpenProblem2017} and heavily studied in the setting of the KdV equation \cite{BDGL,EVY19,DLVY,ChapoutoKillipVisan} and for the Toda lattice \cite{VinnikovYuditskii,LYZZ}. 

In this work, we study the NLS with almost periodic initial data. We will first state a special case of our work: a positive solution of Deift's conjecture for the NLS with small quasiperiodic initial data. We will then describe separately the inverse and direct spectral theory analyses that combine to give this result, and the obstacles and novelty in each part.

\subsection{A special case: small quasiperiodic initial data}
Consider a quasiperiodic analytic function $\varphi$ of the form
\begin{equation} \label{eq:varphiQPform}
\varphi(x) = \widetilde \varphi(\omega x), \quad x \in \bbR,
\end{equation}
where $\bbT^d := \bbR^d/\bbZ^d$ denotes the $d$-dimensional torus, $\widetilde{\varphi}\in C^\omega(\bbT^d,\bbC)$ is analytic, and the frequency  $\omega\in\bbR^d$ satisfies a suitable {\it Diophantine condition}, that is, there exist constants $\kappa>0$ and $\tau>d$ such that
\begin{equation}\label{eq.diophantine}
\inf_{j\in\bbZ}|\langle n,\omega\rangle-j|\geq\frac{\kappa}{|n|^\tau}, ~n\in\bbZ^d\setminus\{0\},
\end{equation}
where $|n|=|n_1|+|n_2|+\cdots+|n_d|$. Let $\DC_d(\kappa,\tau)$ be the set of all such vectors.

For small enough initial data in this class, we prove existence of a global solution almost periodic in space and time, and its uniqueness among locally bounded solutions:

\begin{theorem}\label{thm:main1}
Let $\omega \in \DC_d(\kappa,\tau)$. For every $h>0$ there exists $\epsilon_0 = \epsilon_0(d,\kappa,\tau,h)>0$ such that, if 
\[
\Vert \widetilde{\varphi} \Vert_{h}:=\sup_{|\Im z|<h}|\widetilde{\varphi}(z)|<\epsilon_0,
\]
 then the initial value problem \eqref{eq:nls}, \eqref{eq:varphiQPform} has a unique solution  that is almost periodic in both time and space. To be more specific, there exist a continuous map $\mathcal{M}:\bbT^\infty\times \bbT\to \bbC$ and frequencies $\eta,\eta^{(1)}\in\bbR^\infty, \vartheta_0,\vartheta_1\in\bbR$ and  constant vectors $\zeta\in\bbT^\infty,\zeta'\in\bbT$ such that 
$u(x,t):=\mathcal{M}(\zeta+x\eta+t\eta^{(1)},\zeta'+\vartheta_0 x+\vartheta_1 t)$ solves \eqref{eq:nls}. Moreover, for any other solution $v(x,t)$ of the NLS on $t\in [0,T]$ with $v(x,0)=\varphi(x)$ satisfying the local boundedness condition $v, \partial_x v \in L^\infty(\bbR \times [0,T])$,
we have $v=u$.
\end{theorem}

\subsection{Inverse spectral theory and integrability}
The study of integrable systems has a rich history dating back to 19th century when Jacobi, Abel, Weierstrass and in particular Riemann formulated the remarkable theory of algebraic and Abelian functions. 
Connections between this theory to mathematical physics  were made by the Neumann's system \cite{Neumann1859, Neumann1865} and Kowalevski top \cite{Kowalevski89a, Kowalevski89b}.
Drach, Burchnall and Chaundy, and Baker's excellent work \cite{Baker1928} between 1919 and 1928 was fortunately discovered in 1970's. Seeking for periodic solutions saw great breakthroughs since the debut of inverse scattering method by Gardner--Greene--Kruskal--Miura \cite{GGKM1974} and the discovery of the role of finite gap periodic potentials by Novikov \cite{Novikov1974}. The modern ideas of integrating completely integrable systems stem from Akhiezer's work \cite{Akhiezer1964,Akhiezer1971,Akhiezer1978}. 
In 1975, Its--Matveev \cite{ItsMatveev1975,ItsMatveev1975a} and Dubrovin \cite{Dubrovin1975a,Dubrovin1975b} rediscovered Akhiezer's idea and generalized it to the solution of the Jacobi inversion problem. From a spectral theoretic perspective, this theory describes precisely the reflectionless operators on finite gap spectra; see also \cite{GH03}. Moreover, the theory provides solutions of integrable equations with reflectionless finite gap initial data.

 Finite gap approximation was used by Mckean--van Moerbeke \cite{McKeanMoerbeke1975}, McKean--Trubowitz \cite{McKeanTrubowitz1976}, Levitan \cite{Levitan1987}, Boutet de Monvel--Egorova \cite{ME97}, Sodin--Yuditskii \cite{SY1994}, Gesztesy--Yuditskii \cite{GesztesyYuditskii2006} to construct periodic and almost periodic operators with infinite gap spectra. Sodin--Yuditskii \cite{SY1995} developed a construction of reflectionless Jacobi matrices on bounded spectra which are regular Widom sets with the DCT property. This construction is  based on character-automorphic Hardy spaces \cite{Widom1969,Widom1971Acta, Widom1971Annals,Pommerenke1976,JonesMarshall1985,Hasumi1983}; an important consequence of the construction is almost periodicity of the operators. For adaptations of this construction to other classes of operators see Peherstorfer--Yuditskii \cite{PeherstorferYuditskii06},  Eichinger--VandenBoom--Yuditskii \cite{EVY19}, and Bessonov--Luki\'c--Yuditskii \cite{BLY2}.

Key ingredients for inverse scattering for reflectionless operators and corresponding solutions were introduced by Craig \cite{Craig89} and Rybkin \cite{Rybkin08} in the Schr\"odinger/KdV setting.  In 2018, Binder--Damanik--Goldstein--Luki\'c \cite{BDGL} gave a positive answer to Deift's problem  under suitable assumptions on the spectrum. The construction of almost periodic solutions was further generalized, and extended to the KdV hierarchy, by  Eichinger--VandenBoom--Yuditskii \cite{EVY19}. The same questions were studied for the Toda lattice by Vinnikov--Yuditskii \cite{VinnikovYuditskii} and Leguil--You--Zhao--Zhou \cite{LYZZ}. In 2021, Damanik--Luki\'c--Volberg--Yuditskii \cite{DLVY} proposed a program of constructing counterexamples for Deift's problem. In 2024, an example with almost periodic initial data that loses almost periodicity in space was constructed by Chapouto--Killip--Vi\c{s}an \cite{ChapoutoKillipVisan}. The example they constructed develops a discontinuity that breaks the almost periodicity. 
Another perspective to the construction of NLS flows motivated by integrability was proposed by Kotani \cite{Kotani2018JMPAG,Kotani2023PKU,Kotani2024book}.

In this work we address the Deift conjecture by showing the global existence, uniqueness and almost periodicity in both space and time for the solutions of the cubic defocusing NLS for a broad class of almost periodic initial data with sufficiently thick spectrum, in a sense made precise below.

The Lax pair representation of the NLS uses Dirac operators in the Zakharov--Shabat gauage \cite{ZS1972}:
\begin{equation}\label{eq:diracOperLambda}
	\Lambda_\varphi =  \begin{bmatrix}
	i & 0 \\ 0 & -i
\end{bmatrix}	 \frac{{\rmd}}{{\rmd}x} + \begin{bmatrix} 0 & \varphi(x) \\ \overline{\varphi(x)} & 0 \end{bmatrix}.
\end{equation}
Up to a pointwise unitary conjugation corresponding to the Cayley transform, \eqref{eq:diracOperLambda} is equivalent to the classical form of Dirac operators \cite{LevitanSargsjan, GrebertKappeler2014, ClarkGesztesy} given by
\[
L_\varphi =  \begin{bmatrix}
	0 & -1 \\ 1 & 0
\end{bmatrix}	 \frac{{\rmd}}{{\rmd}x} - \begin{bmatrix} \Re \varphi(x) & \Im \varphi(x) \\  \Im \varphi(x) &  -\Re \varphi(x) \end{bmatrix}.
\]
Due to this simple connection, the two forms $\Lambda_\varphi$, $L_\varphi$ can be used almost interchangeably. If $\varphi\in L^\infty(\bbR)$, the operator $\Lambda_\varphi$ is an unbounded self-adjoint operator; denote its spectrum and absolutely continuous spectrum by $\Sigma_\varphi$ and $\Sigma_{\varphi,\ac}$, respectively. The connected components of $\bbR \setminus \Sigma_\varphi$ are called spectral gaps. Labelling them by a countable index $j$, we can write
\[
\Sigma_\varphi=\bbR\setminus \cup_{j} G_j
\]
where $G_j=(a_j,b_j)$ are open intervals and denote gap lengths and gap distances by
\[
\gamma_j = \lvert G_j\rvert=b_j-a_j,  \quad \eta_{jk} = \mathrm{dist}( G_j, G_k).
\]
Define also
\begin{equation}\label{eq:Cj}C_j=\sup_{\lambda\in G_j}\left(\prod_{G_k<G_j}\frac{a_k-\lambda}{b_k-\lambda}\prod_{G_j<G_k}\frac{b_k-\lambda}{a_k-\lambda}\right)^{\frac{1}{2}},\end{equation}
where we write $G_k<G_j$ to mean $G_k$ lies to the left of $G_j$ in $\bbR$. Under a set of Craig-type conditions on the spectrum, we prove global existence, uniqueness and almost periodicity:
\begin{theorem}\label{thm:main}
Assume that $\varphi$ is uniformly almost periodic and that the associated Dirac operator $\Lambda_\varphi$ has purely absolutely continuous spectrum that obeys the Craig-type conditions 
\begin{equation}\label{eq:craigCond1}
\sum_{k}(1+\eta_{k0})C_k\gamma_k^{1/2}<\infty,
\end{equation}
\begin{equation}\label{eq:craigCond2}
\sup_j\sum_{k\neq j}C_j^3C_k^2(1+\eta_{0j}^2)\frac{\gamma_j^{1/2}\gamma_k^{1/2}}{\eta_{jk}}<\infty,
\end{equation}
and there exists $\delta>0$ such that
\begin{equation}\label{eq:craigCond3}
\sup_j\sup_{k\neq j}\frac{\gamma_j^{1/2}\gamma_k^{1/2}}{\gamma_j^{\delta}\eta_{jk}}<\infty.
\end{equation}
Then the Cauchy problem \eqref{eq:nls}  has a unique solution that is almost periodic in both time and space. To be more specific,
there exist a continuous map $\mathcal{M}:\bbT^\infty\times \bbT\to \bbC$, frequencies $\eta,\eta^{(1)}\in\bbR^\infty, \vartheta_0,\vartheta_1\in\bbR$, and  constant vectors $\zeta\in\bbT^\infty,\zeta'\in\bbT$ such that 
$u(x,t):= \ \mathcal{M}(\zeta+x\eta+t\eta^{(1)},\zeta'+\vartheta_0 x+\vartheta_1 t)$ solves \eqref{eq:nls}. 
Moreover, for any other solution $v(x,t)$ of the NLS on $t\in [0,T]$ with $v(x,0)=\varphi(x)$ satisfying the local boundedness condition $v, \partial_x v \in L^\infty(\bbR \times [0,T])$,
we have $v=u$.
\end{theorem}

\begin{remark}
Our analysis differs in key aspects from past results, and addresses new obstacles and important qualitative differences between the KdV and NLS hierarchies which were not discussed before in this context:
\begin{enumerate}
\item For reflectionless finite-gap Dirac operators, as in the Schr\"odinger case, in case of $g$ open gaps, there is an isospectral torus of dimension $g$ \cite{BBEIM1994,GH03,GrebertKappeler2014}. However, there is an intrinsic difference. For KdV, all \(g\) tori in the product torus arise from the character group of the associated domain obtained by the removal of the spectrum from the extended complex plane. For NLS, \(g-1\) of these tori arise from the character group and one of them comes from the ``character group" of sheet involution. The isospectral torus of NLS can be realized as \(\pi_1(\Omega)^*\times\bbT\) instead of \(\pi_1(\Omega)^*\), where \(\Omega\) is the domain mentioned above. Our discussion will show that this extra different coordinate indeed has its origin in the construction of solutions of the NLS for finite gap potentials. In the case of infinite gap potentials, bridging the physical meaning of this coordinate in the expression of the solution with the abstract algebro-geometrical formulation is the one of the major challenges in this work.


\item There is also a corresponding rotation flow commuting with the NLS flow, which corresponds to a trivial symmetry of NLS given by $u \mapsto e^{i\beta} u$, $\beta \in \bbR$. This rotation flow has no analog in the KdV hierarchy.
We couple this flow in a novel way with trace formulas to recover both the real and imaginary parts of the solution.

\item When deriving the space and time evolution of Dirichlet eigenvalues, the behavior in gap interiors is a matter of implicit differentiation \cite{Craig89, GrebertKappeler2014}, but the behavior at gap edges must be addressed separately. In the Schr\"odinger case, a non-pausing property of the translation flow is used; in the Dirac case, the translation flow does not have that property, but the rotation flow can be used instead. This again illustrates the difference between hierarchies which has not been addressed before. Without this, uniqueness of solutions cannot be established.

\item Whereas periodic operators have a sequence of open gaps which can only accumulate at infinity (compare \cite{GrebertKappeler2014}), almost periodic operators often have Cantor spectrum. The accumulation of gaps greatly complicates their inverse spectral theory, and has to be controlled by suitable conditions. The Craig-type conditions \eqref{eq:craigCond1}--\eqref{eq:craigCond3} are standard in the analysis of the Dubrovin vector fields \cite{Craig89,ME97,PeherstorferYuditskii06, BDGL}. They seem to be complicated but are less restrictive than expected. For Schrödinger and Dirac operators, small quasiperiodic analytic potentials with Diophantine frequencies can be proved to have spectrum satisfying these conditions. The purpose of these conditions are two fold. First, they imply that the spectrum is a homogeneous subset. This is crucial for the analysis in the Hardy spaces of multiplicative functions on infinitely connected Riemann surfaces \cite{SY97,SY1995}. Second, they are sufficient for the Dubrovin vector fields to be globally Lipschitz, see Proposition \ref{prop:LipVecField} in our literature and \cite{Craig89,BDGL} in the literature of KdV.
\end{enumerate}
\end{remark}

A construction of almost periodic solutions for the NLS was presented in \cite{ME97}.
This construction relied on a solution of the generalized Jacobi inversion problem in the form of Concini--Johnson \cite{ConciniJohnson1987ETDS} but involving both space and time variables that was not fully carried out. In addition, different order trace formulas were used for the real and imaginary parts, and the real part was not expressed directly; only the Ricatti-type expression $\pm \partial_x \Re u_n + (\Re u_n)^2$ was expressed in terms of Dirichlet eigenvalues, which is insufficient to conclude almost periodicity of the real part of the solution.

From the perspective of inverse spectral theory, the most important object is the generalized Abel map of Sodin--Yuditskii \cite{SY97}. We emphasize that our proof does not only show almost periodicity in an abstract form, it also shows that the isospectral torus is geometrically an infinite product of tori on which the image of the NLS and translation flow under the generalized Abel map is a linear flow. 

\begin{theorem}\label{thm:main3} Under the assumptions of Theorem \ref{thm:main}, 
the Abel map $\mathcal{A}=(\mathcal{A}_c,\mathcal{A}_r)$ in \eqref{eq:Abel1} and \eqref{eq:rotFlowAbelMap} maps the  NLS and translation flow $u(x_0,t_0)\mapsto u(x_0+x,t_0+t)$ into a linear flow.
\end{theorem}

\subsection{Direct spectral theory and quantitative almost reducibility}
Gap labeling theory \cite{JM82,Johnson1986JDE} assigns to each gap the value of the rotation number, which lies in the frequency module. In particular, for a quasiperiodic Dirac operator with potential \eqref{eq:varphiQPform}, the rotation number is of the form $ \langle k, \omega\rangle$, $k\in \bbZ^d$. In the direct spectral analysis in this paper, we will use $k \in \bbZ^d$ as the index for the gap having rotation number $\langle k, \omega\rangle$.

From the perspective of direct spectral theory, the two main ingredients needed for our work are 
absolutely continuous spectrum and suitable estimates on gap size for the Dirac operator. Precise gap estimates, in turn, imply thickness properties of the spectrum, such as Craig-type conditions and homogeneity (recalled in Definition~\ref{def:homogeneous}). 
We prove the following analogous result of \cite[Theorem B part (1)]{DamanikGoldstein2014}:

\begin{theorem}\label{thm:main2}
Under the conditions of Theorem~\ref{thm:main1}, the spectral type of $\Lambda_\varphi$ is purely absolutely continuous. For all $k\in\bbZ^d, r\in (0,h)$,
\[
\gamma_{k}\leq \epsilon_0^{\frac{1}{2}}e^{-2 \pi r |k|},
\] 
and for all $k' \neq k$,
\[\mathrm{dist}(G_k,G_{k'})\geq\frac{4^\tau\kappa^2}{c^2|k'-k|^{2\tau}}.
\]
Moreover, the spectrum is homogeneous in the sense of Carleson, and satisfies the Craig-type conditions \eqref{eq:craigCond1}, \eqref{eq:craigCond2}, \eqref{eq:craigCond3}.
\end{theorem}

The study of absolutely continuous spectrum has a long history, dating back to the groundbreaking work of Dinaburg--Sinai \cite{DinaburgSinai}, Eliasson \cite{eliasson} and more recently Avila--Jitomirskaya \cite{aj1} and Avila \cite{Avila1}. 
A closely related notion is {\it reflectionless} operators, that is, those for which two half-line Weyl functions satisfy a pseudocontinuation relation a.e.\ on the spectrum. Almost periodic Schr\"odinger operators are known to be reflectionless on the absolutely continuous part of the spectrum \cite{Kotani1984, Kotani1997, Remling2007}. 
The counterpart for Dirac operators can be found in \cite{BLY1}. On the other hand,  the gap sizes can be viewed as a certain kind of Fourier coefficients for the potential; see P\"oschel \cite{Poschel2011},  Damanik--Goldstein \cite{DamanikGoldstein2014}, Damanik--Goldstein--Luki\'c \cite{DGL2017TAMS,DGL2017Invent}, Binder--Damanik--Goldstein--Luki\'c \cite{BDGL2025} and the references therein. 
In particular, analytic almost periodic potentials should exhibit exponential decaying gaps at least in the regime of small analytic quasiperiodic sampling functions. 
Theorem~\ref{thm:main2} confirms this general belief for one-dimensional Dirac operators. 

The direct spectral theory part of this paper establishes fundamental spectral results for quasiperiodic Dirac operators by extending techniques previously developed for Schr\"odinger operators. The primary difficulty lies in overcoming the distinct symmetries inherent in Dirac operators compared to the Schr\"odinger framework. Our approach leverages dynamical systems methods, specifically \textit{quantitative almost reducibility} (detailed in Section \ref{qal}), which has driven significant advances in quasiperiodic Schr\"odinger theory~\cite{Avila1,aj1,avila2017sharp,avila2016dry,ds,eliasson,GYZ}. 

While almost reducibility has been well-established~\cite{eliasson,HY2012,LYZZ}, we prove new results for Dirac operators. 
We also derive some foundational results which we expect to be of interest  independent of our main results, such as the Jitomirskaya--Last inequality \cite{JL1999Acta,JL2000CMP} for Dirac operators. This inequality has many important applications on spectral dimensions\cite{JZ2022,GDO2022} and absolutely continuous spectrum \cite{Avila1,Cai2022AHP,DLZ22}.

We generalize Avila's  approach~\cite{Avila1} for Schr\"odinger operators to the setting of Dirac operators. By combining quantitative subordinacy theory via the Jitomirskaya--Last inequality we obtained with quantitative almost reducibility, and imposing Diophantine frequency constraints to prevent resonance accumulation, we establish purely absolutely continuous spectrum.
     
The notion of homogeneous sets stem from Carleson \cite{Carleson1983} and plays a fundamental role in harmonic analysis and inverse spectral theory. In particular, homogeneity of $\sfE\subset \bbR$ implies the interpolating property of the zero set of the Blaschke product of the domain $\Omega=\bbC\setminus\sfE$ \cite{JonesMarshall1985}. Based on this observation, \cite{SY97} proved that homogeneity is a sufficient condition for the Widom condition and  the Direct Cauchy Theorem. The homogeneity of the spectrum of one dimensional discrete Schr\"odinger operator in the positive Lyapunov exponent region was obtained by Damanik--Goldstein--Schlag--Voda \cite{DGSV15}. For continuous Schr\"odinger operators with small potentials, see Damanik--Goldstein--Luki\'c \cite{DGL2016JST}. Similar results for discrete Schr\"odinger operators with finitely differentiable quasiperiodic potentials was obtained by Cai--Wang \cite{CaiWang2021JFA} and with $C^2$ cosine type potentials by Ge--Wang--Xu \cite{GXW25}. We extend the Moser-P\"oschel argument \cite{MoserPoschel194CMH} to Dirac operators by observing monotonicity with respect to the spectral parameter. This yields {spectral thickness} through Craig-type conditions and homogeneity in the sense of Carleson.

\subsection{Further discussion}
It is very useful to study the spectral  problem from the direct and inverse pserpectives simultaneously as mentioned in \cite[Page 695]{VolbergYuditskii2014}. 
From this aspect, perhaps the most instructive example is the solution of the Kotani--Last conjecture (KLC), which asks  {\it whether the presence of a.c.\ spectrum for an ergodic family implies almost periodicity}. 
This was disproved independently by Avila \cite{Avila2015JAMS} for discrete and continuum Schr\"odinger operators through an induction scheme based on slow deformations of periodic potentials (direct spectral theory) and Volberg--Yuditskii \cite{VolbergYuditskii2014} through a dichotomy based on the Direct Cauchy Theorem property (DCT); additional counterexamples can be found in  \cite{DamanikYuditskii2016Adv, YouZhou2015IMRN} and references therein. To be more specific, let $\mathcal{J}(\sfE)$ be the set of reflectionless Jacobi matrices, $\Omega=\hat{\bbC}\setminus \sfE$ be of Widom type, where $\hat{\bbC}=\bbC\cup\{\infty\}$ is the Riemann sphere, every $J\in\mathcal{J}(\sfE)$ has purely absolutely continuous spectrum $\sfE$ and the gap labeling frequencies are rationally independent. 
Volberg and Yuditskii proved the following: \emph{
If $\Omega$ obeys DCT, then every $J\in\mathcal{J}(\sfE)$ is almost periodic; if $\Omega$ fails DCT, then none of $J\in\mathcal{J}(\sfE)$ is almost periodic.
}
The method of \cite{VolbergYuditskii2014}   based on Sodin--Yuditskii \cite{SY1994,SY1995,SY97} is completely different from Avila's approach. 
This method was further implemented for continuum Schr\"odinger operators (and CMV matrices) by Damanik--Yuditskii \cite{DamanikYuditskii2016Adv}. 
We are not disproving KLC for Dirac operators in this work but provide an additional positive example, namely, we prove the purely absolutely continuous spectrum for Dirac operators with small analytic quasiperiodic potentials. 
The natural question is then if the analysis of \cite{VolbergYuditskii2014} can be used to study Deift's problem. Indeed, this is the purpose of the program proposed by \cite{DLVY}.

\subsection{Outline of the Paper}
In Section~\ref{background}, we recall some basic facts about almost-periodic functions, quasi-periodic systems, and spectral theory of dynamically defined Dirac operators. In Section~\ref{JLineq}, we prove the Jitomirskaya--Last inequality, and use it to control the maximal spectral measure with the norm of transfer matrix. We observe that at a spectral edge, there exists a subordinate solution. This observation will be important to clarify the behavior of the translation flow at the spectral edge.
In Section~\ref{canonicalSys}, we review the recent results of \cite{BLY1,BLY2} on canonical systems. In particular, the generalized Abel map is shown to be a homeomorhpism under the conditions that $\Omega=\bbC\setminus \sfE$ is Dirichlet regular, of Widom type, and obeys the Direct Cauchy Theorem (DCT). 
In Section~\ref{dubrovinFlow}, we introduce the rotation flow and Dubrovin flow and study their properties. In particular, we use the rotation flow to show the non-pausing property of the phase flow and the vector fields which describe the Dubrovin flow are Lipschitz under the Craig type conditions.
In Section~\ref{linearabel}, we analyze the generalized Abel map, in particular, the rotation coordinate and prove Theorem~\ref{thm:main3}.
In Section~\ref{arac}, we introduce the quantitative almost reducibility and quantitative Puig's argument for continuum systems and prove Theorem~\ref{thm:main2}. 
In Section~\ref{sec:mainproof}, we combine all the ingredients and give the proofs of Theorem~\ref{thm:main} and Theorem~\ref{thm:main1}.

\subsection*{Acknowledgements}
We would like to thank David Damanik and Peter Yuditskii for helpful discussions.
J.F.\ was supported in part by National Science Foundation grants DMS-2213196 and DMS-2513006. L.L. was supported by AMS-Simons Travel Grant 2024-2026. M.L.\ was supported in part by National Science Foundation grant DMS-2154563.Qi Zhou was supported by NSFC grant (12531006,12526201) and  Nankai Zhide Foundation. 

\subsection*{AI Statement}
The mathematical argument of this work was done in April 2025 and posted on arXiv by August 2, 2025 after several rounds of polishing. We declare no use of AI.

\section{Background on Ergodic Dirac Operators} 
\label{background}

Here, we recall some basic facts and background information about almost-periodic functions, quasi-periodic systems, and Dirac operators.

\subsection{Almost-Periodic Functions}

Let us briefly review a few definitions and aspects related to almost-periodic functions; 
we direct the reader to \cite[Section~5.13A]{Simon2011Szego} for additional details.

If $f:\bbR \to \bbC$ and $T \in \bbR$ are given, the \emph{translate} of $f$ by $T$ is the function $f_T:= f(\cdot-T)$.
A function $f: \bbR\to \bbC$ is called \emph{periodic} if $f \equiv f_T$ for some $T \neq0$.
We say that $T$ is an $\varepsilon$-almost-period of $f$ if $\|f - f_T\|_\infty < \varepsilon$.
A continuous function $f \in C(\bbR)$ is said to be almost-periodic if and only if one of the following equivalent conditions holds true (compare \cite[Theorem~5.13.A2]{Simon2011Szego}):
\begin{enumerate}[label={\rm(\alph*)}]
\item For every $\varepsilon>0$, the set of $\varepsilon$-almost-periods is relatively dense in $\bbR$ (i.e., for every $\varepsilon>0$, there exists $R>0$ such that every subinterval of $\bbR$ of length at least $R$ contains at least one $\varepsilon$-almost-period).
\item $\{f_T: T \in \bbR\}$ is precompact in the uniform topology
\item There is a continuous function $F:\bbT^\infty \to \bbR$ and $\omega \in \bbT^\infty$ such that $f(x) = F(x\omega)$ for all $x \in \bbR$ (here $\bbT^\infty$ denotes a product of a countably infinite collection of copies of $\bbT$ with the product topology).
\end{enumerate}

If $f$ is an almost-periodic function, then the limit
$$
\bohrmean(f)
=\lim\limits_{x\to\infty}\frac{1}{x} \int_{0}^{x} \! f(t) \, \rmd t
$$
exists and is called the {\it Bohr mean} of $f$. For more literature on different notions of almost periodicity and their relation with pure point diffraction, we refer to \cite{LSS, LSS2024ETDS, LSS2024CJM} and the reference therein.

\subsection{Quasi-Periodic Systems}
Given $A_0(\lambda)\in \essell(2,\bbR)$ for some $\lambda\in\bbR$, $F_0\in C^\omega( \bbT^d, \essell(2,\bbR))$ analytic and $ \omega\in\bbT^d$ rationally independent, we consider a quasi-periodic linear system of the form 
\begin{equation}\label{eq.linearSys}
X'(x) = (A_0(\lambda)+F_0(\theta + \omega x))X(x), \quad x \in \bbR,
\end{equation}
where $X=[X_1,X_2]^\top: \bbR \to \bbR^2$   and $\theta \in \bbT^d$.

\begin{definition}[Rotation number]\label{def:rotNum}
Consider a solution $X(\lambda, x, \theta)$  of \eqref{eq.linearSys} with $X(\lambda, 0, \theta)  \neq 0$.
The \emph{rotation number} of the system  is then given by
\begin{equation} \label{eq:rotnumberDef}  \rho (\lambda)=\lim\limits_{x\to\infty}\frac{\arg X(\lambda, x,\theta)}{x},
\end{equation}
which exists and is independent of $X(\lambda, 0,\theta)$ and $\theta$, as $ \omega\in\bbT^d$ is rationally independent; see \cite[Appendix~A]{eliasson} for further discussions.
Here in \eqref{eq:rotnumberDef}, $\arg X(\lambda, x,\theta)$ denotes a continuous choice of the argument of $X_1+iX_2$ normalized by $\arg X(\lambda, 0,\theta) \in [0,2\pi)$; see  \cite{JM82}.
\end{definition}

\begin{definition}[Fundamental Matrix]
Given $\theta \in \bbT^d$, choose fundamental solutions $X(\lambda, x, \theta)$ and $Y(\lambda, x, \theta)$ of the system \eqref{eq.linearSys} satisfying $X(\lambda, 0,\theta) = [1,0]^\top$ and $Y(\lambda, 0,\theta) = [0,1]^\top$.
The \emph{fundamental matrix} of the system is the function $\transfermat:\bbR \times \bbT^d \to \SL(2,\bbR)$ given by
    \[
        \transfermat(\lambda, x,\theta)
        = \begin{bmatrix}
            X_1(\lambda, x,\theta) & Y_1(\lambda, x,\theta) \\ 
            X_2(\lambda, x,\theta) & Y_2(\lambda, x,\theta)
        \end{bmatrix}.
    \]
\end{definition}

    The \emph{Lyapunov exponent} of the system \eqref{eq.linearSys} is given by
    \begin{eqnarray}
        \gamma (\lambda) = \lim_{x \to \infty} \frac{1}{x} \int_{\bbT^d} \! \log \|\transfermat (\lambda, x,\theta)\| \, \rmd\theta.
    \end{eqnarray}
We say that the system is \emph{uniformly hyperbolic} at $\lambda$ if there is a constant $c>0$ such that
\[
    \|\transfermat(\lambda, x,\theta)\| \geq ce^{c|x|} \quad \forall x \in \bbR, \ \theta \in \bbT^d.
\]
It is well-known that the system is uniformly hyperbolic if and only if it exhibits an \emph{exponential dichotomy}, that is, there are continuous maps $\stab, \unstab: \bbT^d \to \bbR\bbP^1$ and constants $C,c>0$ such that
\begin{align}
    \transfermat(\lambda, x,\theta)\bistab(\lambda, \theta) = \bistab(\lambda, \theta+x\omega) \quad \forall \theta \in \bbT^d, \ x \in \bbR,
\end{align}
and for any $v_\pm \in \bistab(\lambda, \theta)$, one has
\[
    \|\transfermat(\lambda, \pm x, \theta) v_\pm\| \leq C e^{-cx}\|v_\pm\| \quad \forall x \geq 0;
\]
see \cite{Yoc04} or \cite[Theorem~3.9.5]{DF2022ESO1} for more details.

While the general theory is established for $\mathrm{SL}(2,\bbR)$ valued matrices, the following simple relation is useful for interpreting results to their corresponding $\mathrm{SU}(1,1)$ version:
\[M^{-1}\mathrm{SU}(1,1)M=\mathrm{SL}(2,\bbR),~ M=\frac{1}{1+i}\begin{bmatrix}
    1&-i\\1&i
\end{bmatrix}.\]

\subsection{Relation to Spectra of Dirac operators}
Our work centers around Dirac operators $\Lambda_\varphi$ as in \eqref{eq:diracOperLambda} satisfying the uniform local integrability condition
\begin{equation} \label{eq:varphiUnifLocL2}
    \sup_{x \in \bbR} \int_x^{x+1} |\varphi(t)|^2 \, \rmd t <\infty.
\end{equation}
Of course, analytic quasi-periodic potentials are uniformly bounded, so \eqref{eq:varphiUnifLocL2} is trivially satisfied when $\varphi$ takes the form \eqref{eq:varphiQPform} with $\widetilde\varphi$ analytic.

The spectrum $\Sigma_\varphi : = \sigma(\Lambda_\varphi)$ can be characterized by a Schnol-type theorem; see \cite{EFGL} for a proof suited to the current setting.

\begin{theorem}[Schnol's Theorem]\label{thm:Schnol}
Assume $\varphi \in L^2_{\rm loc}(\bbR)$ satisfies \eqref{eq:varphiUnifLocL2}.
Given $\kappa>\frac{1}{2}$, let $S_\kappa(\varphi)$ be the set of $\lambda \in\bbC$ such that there exists a nontrivial solution $U$ of $\Lambda_\varphi U = \lambda U$ obeying $\Vert U(x)\Vert \leq C|x|^\kappa $ for a suitable constant $C>0$. 
Then 
\begin{enumerate}[label={\rm(\alph*)}]
\item $S_\kappa(\varphi)\subseteq \Sigma_\varphi$;
\item $S_\kappa(\varphi)$ supports the maximal spectral measure of $\Lambda_\varphi$;
\item $\Sigma_\varphi=\overline{S_\kappa(\varphi)}$.
\end{enumerate}
\end{theorem}

Since $\varphi$ is assumed to have the form \eqref{eq:varphiQPform}, solutions to the eigenvalue equation $\Lambda_\varphi U = \lambda U$ can be characterized by the quasiperiodic linear system
\begin{equation} \label{eq:DiracLinearSys}
   \left\{
\begin{aligned}
&  \frac{\rmd X}{\rmd x}
    = \begin{bmatrix}
        -i\lambda & -i \widetilde\varphi (\theta)\\
         i\overline{\widetilde\varphi(\theta)} & i\lambda 
    \end{bmatrix}X(x) \\
    & \frac{\rmd \theta}{\rmd x} = \omega.
    \end{aligned}
    \right.
\end{equation}
As a direct consequence of Theorem~\ref{thm:Schnol}, the following holds.
\begin{coro}\label{unre}
For any $\lambda\in\bbR$, the system \eqref{eq:DiracLinearSys} is uniformly hyperbolic if and only if $\lambda \notin \Sigma_\varphi$.
\end{coro}

\begin{proof}
Fix $\kappa>1/2$. 
If \eqref{eq:DiracLinearSys} is uniformly hyperbolic, then any nontrivial solution of $\Lambda_\varphi U = \lambda U$ most grow exponentially on at least one half-line, which forces $z \notin S_\kappa$. 
Since uniform hyperbolicity is an open condition, we see that $\lambda \notin \overline{S_\kappa}$, so $\lambda \notin\Sigma_\varphi$ by Theorem~\ref{thm:Schnol}.

Conversely, if the system is not uniformly hyperbolic, then it enjoys a uniformly bounded solution (e.g, by interpolating \cite[Theorem~1.2]{DFLY2016DCDS}) which in particular shows that $\lambda \in \Sigma_\varphi$.
\end{proof}

\subsection{Density of States and Gap Labeling}
Given $\lambda_1<\lambda_2$, let $N_L(\lambda_1,\lambda_2)$ denote the number of eigenvalues on the interval $(\lambda_1,\lambda_2)$ of the  boundary value problem for $U=(U_1,U_2)^\top$:
\begin{equation}
    \Lambda_\varphi U = \lambda U, \quad U_1(0) = U_1(L)=0.
\end{equation}
The density of states measure $\mathcal{N}$ is defined by
$$
\mathcal{N}(\lambda_1,\lambda_2)=\lim\limits_{L\to\infty}\frac{N_L(\lambda_1,\lambda_2)}{L}$$ 
whenever the limit exists.

Note as $\Lambda_\varphi U= \lambda U$ can be rewritten as quasiperiodic linear system \eqref{eq:DiracLinearSys},
the density of states is closely related to the rotation number of the system \eqref{eq:DiracLinearSys}, which we denote by $\rho(\lambda)$.
By \cite[Theorem 1]{Savin86}, $\rho$ is a nondecreasing function.
We also need the following result from \cite[Lemma~1]{Savin86}:
\begin{lemma}\label{lem:rotAndDensity}
If $\lambda_1<\lambda_2$, then $\rho(\lambda_2)-\rho(\lambda_1)=\pi\mathcal{N}(\lambda_1,\lambda_2)$. 
\end{lemma}

In light of Lemma~\ref{lem:rotAndDensity}, we select $\lambda_0$ such that $\rho(\lambda_0) =0$ and define
\begin{equation}
    \mathcal{N}(\lambda) =
    \begin{cases}
        \mathcal{N}(\lambda_0,\lambda) & \lambda \geq \lambda_0 \\
        -\mathcal{N}(\lambda,\lambda_0) & \lambda < \lambda_0,
    \end{cases}
\end{equation}
which then gives us the relation $ \rho(\lambda) = \pi \mathcal{N}(\lambda) $ for all $\lambda$.
The following result can be found in the literature; compare \cite{DF2023gap, Johnson1986JDE, JM82, Savin86, hadj2012}.

\begin{theorem}[Gap Labeling Theorem]\label{thm.gapLabel}
For any non-empty interval $I\subseteq\bbR\setminus\Sigma_\varphi$, there exists $k\in\bbZ^d$ such that 
$$
2\pi\mathcal{N}(\lambda)= 2\rho(\lambda)=\langle k,\omega\rangle, \quad \lambda \in I.
$$
\end{theorem}
In the current context, a closed subset $\sfE\subset\bbR$ is assumed to take the form $\sfE=\bbR\setminus\cup_{j\in\bbN}(a_j,b_j)$. We will use both notations $I_j=(a_j,b_j), j\in\bbN$ and $G_k, k\in\bbZ^d$ interchangeably. The natural correspondence  between these labels can be understood as $G_k=\rho^{-1}(\frac{\langle k,\omega\rangle}{2})=(a_j,b_j)$.

\subsection{Floquet Exponent and Green's Function}
Denote $\bbR^+ = [0,+\infty)$ and $\bbR^{-} = (-\infty,0]$, and let $\Psi^{\pm}(\lambda, x)=(\Psi^\pm_1,\Psi^\pm_2)^\top \in L^2(\bbR^{\pm},\bbC^2)$ be the Weyl solutions at $\pm\infty$ to the equation $L_{\varphi}\Psi=\lambda \Psi$ for $\lambda\in \bbC\setminus\Sigma_\varphi$. 
The diagonal elements of the Green's function is given by  
\begin{equation}\label{eq.GreenFunc}
G_{k,k}(\lambda,x)
= \Psi^{+}_{k}(\lambda,x)\Psi^{-}_{k}(\lambda,x)/W[\Psi^{+},\Psi^{-}], 
\quad k=1,2,
\end{equation}
where $W[\Psi^+,\Psi^-]=\Psi^-_{1}\Psi^{+}_{2}-\Psi^{-}_{2}\Psi^{+}_{1}$ denotes the Wronskian of $\Psi^+$ and $\Psi^-$.

Let us now consider $\varphi$ as in \eqref{eq:varphiQPform} with $\widetilde\varphi$ analytic (which in particular means that $\varphi$ is almost-periodic).
The {\it Floquet exponent} is a holomorphic function in the upper half-plane $\bbC^{+} := \{z\in\bbC:\Im z>0\}$ given by \begin{equation}\label{eq.floquet}
w(z)
=-\frac{1}{2}\bohrmean \left( (z-\Im\varphi )G^{-1}_{1,1}(\cdot ,z) \right)
\end{equation}
where $\bohrmean(f(x))$ denotes the Bohr mean of an almost periodic function.
This map is also known as the generalized Marchenko--Ostrovski or Schwarz–-Christoffel mapping for Schr\"odinger operators, which maps the upper half-plane to a comb domain.
It is known that as $z \to\infty$ with $\Im z>0$, the following holds
\begin{equation}\label{eq.asymptotic}
w(z)
=i z+\frac{1}{\pi}\int_{-\infty}^{\infty}\frac{\rho(t)-t}{t-z}\, \rmd t;
\end{equation}
see \cite[Theorem 2]{Savin86} and \cite[Corollary 4.7]{EGL}.
In particular, as $|z|\to\infty$, $w(z)\sim i z + O(1/|z|)$.
Moreover, $w$ is a holomorphic function on $\bbC^+$ satisfying $w(z) = -\gamma(z)+i\rho(z)$ for $\Im z>0$, where $\rho(z)$ (resp.\  $\gamma(z)$) denotes the rotation number (resp.\  Lyapunov exponent)  associated to the system \eqref{eq:DiracLinearSys}.

To establish our main results, we need a suitable linear lower bound on the Lyapunov exponent when the spectral parameter leaves the real axis. 
Similar results for Jacobi matrices and Schr\"odinger operators can be found in \cite{DS83}; the argument is standard, but we give it to keep the work more self-contained.

\begin{prop}\label{thm.complexLE} Assume that $\varphi$ is of the form \eqref{eq:varphiQPform} and Let $\gamma(z)$ be the Lyapunov exponent.
Then there exists $\delta_0 > 0$ such that for all $\lambda\in\bbR$ \begin{equation} 
\gamma(\lambda+i\delta)\geq \delta
\end{equation}
for all $0 < \delta < \delta_0$.
\end{prop}

\begin{proof}
Let $w$ denote the Floquet exponent as above. 
For $\epsilon,L>0$ let $H_{\epsilon}(z) = \Im z -\epsilon$ and 
$$
R(L,\epsilon)=\{z\in \bbC: \epsilon<\Im z<L 
\text{ and } -L<\Re z<L\}.$$  
For $L$ sufficiently large, we have $-\Re w(z)\geq H_{\epsilon}(z)$ on $\partial R(L,\epsilon)$ since $-\Re w(z) = \gamma(z)\geq 0$ and $w(z) = iz + O(1/|z|)$.
Since $\Re w(z)$ and $H_{\epsilon}(z)$ are harmonic functions in $R(L,\epsilon)$, we deduce
\[\gamma(z) \geq \Im z - \epsilon\]
for all $z \in R(L,\epsilon)$.
Taking the limit $\epsilon\to 0, L\to\infty$ then concludes the argument.
\end{proof}

\subsection{Martin Compactification and Martin Function} 
 Let $\sfE$ be an unbounded closed subset of $\bbR$ and $\Omega=\bbC\setminus \sfE$ is the associated Denjoy domain.
 Let $\Omega^\ast$ be the {\it Martin compactification} and $\partial^M\Omega=\Omega^\ast\setminus\Omega$ be the Martin boundary.
 Let $a\in \Omega$ be held fixed and $G(z,a)$ be the Green's function with a pole at $a$. 
 The Martin kernel $M(z,a)=\frac{G(z,a)}{G(z_*,a)}$ normalized at $z_*$ (by $M(z_*,a)=1$) extends continuously to the boundary a {\it Martin function} denoted by the same letter.  
 Denote $\partial^M_1\Omega\subset \partial^M\Omega$  the set of {\it minimal points}. Indeed, for $b\in\partial^M\Omega$, there exists an injection $\mathbf{j}:b\to M(z,b)$, $z$ running through all of $\Omega$. Let $z_*\in \Omega$ be held fixed and denote $HP(z_*)$ the convex cone of positive harmonic functions normalized by $u(z_*)=1$.  
 One can check that $\mathbf{j}$ maps $\partial^M\Omega$ into $HP(z_*)$ and that $\partial_1^M\Omega$ is the preimage of extreme points of $HP(z_*)$ under $\mathbf{j}.$ Interested readers may refer to \cite{Hasumi1983} for more detailed discussions.
 Every (positive) harmonic function $u$ on $\Omega$ has an integral representation 
 $$u(z)=\int_{\partial^M_1\Omega}M(z,b) \, \rmd\mu_u(b),$$
 where the measure $\mu_u$ is called the {\it canonical measure} of $u.$

 Let us introduce the set of Martin functions associated to $\infty$ as follows
 $$\mathcal{M}_\infty=\cap_{K>0}\overline{\{M(z,z_0):z_0\in\Omega,|z_0|>K\}}.$$

 Let $M_\infty$ be the {\it symmetric Martin function} at $\infty$, that is, a positive harmonic function on $\Omega$ satisfying $M_\infty(z)=M_\infty(\bar{z})$ and vanishes continuously on $\sfE$ but $\infty$ which is unique up to normalization. The convex cone of symmetric Martin function is always of dimension one. It can be shown that 
 \begin{equation}\label{eq:mPRWCond}
 \sum_{c_j:\nabla M_\infty(c_j)=0}M_\infty(c_j)<\infty
 \end{equation} implies the Widom condition \eqref{eq:WidomCond}; compare \cite{EVY19}.

 We say $\sfE\subset\bbR$ is of {\it Akhiezer-Levin} type or satisfies the AL-condition if the following holds:
 \begin{equation}\label{eq:ALcondition}
 \lim\limits_{y\to\infty}\frac{M_\infty(iy)}{y}>0.
 \end{equation}

In this case $\mathcal{M}_\infty$ contains two minimal independent elements $M_{\infty_\pm}(z)$ and moreover $$\mathcal{M}_\infty=\{\alpha_1M_{\infty_+}+\alpha_2M_{\infty_-}:\alpha_1,\alpha_2\geq 0,\alpha_1+\alpha_2=1\}.$$ 
The symmetric Martin function at the infinity can  simply be written as $$M_\infty=\frac{1}{2}(M_{\infty_+}+M_{\infty_-}).$$
We will  denote $M=M_\infty$ when it is clear from the context.

Martin functions of Denjoy domains naturally appear in spectral theory through the Marchenko--Ostrovskii map and its generalization to the almost periodic setting \cite{MarchenkoOstrovskii75,JM82,Lukic22,KheifetsYuditskii20,BDGL2025} and beyond \cite{EichingerLukic25,EGL}. We will use Martin functions to introduce the Abelian integrals of the second kind in Section \ref{linearabel}.

\section{Jitomirskaya--Last Inequality for Dirac Operators}
\label{JLineq}

In this section, we aim to establish the fundamental relationship between the spectral measure and the growth of the fundamental solution. To achieve this, we first derive a Jitomirskaya–Last inequality for Dirac operators on the half-line $\bbR_+ = [0,\infty)$ with $\varphi \in L^1_{\rm loc}(\bbR_+)$, that is
\begin{equation}
    \int_x^{x+1} |\varphi(t)| \, \rmd t < \infty, \quad \forall x \in \bbR.
\end{equation}

\subsection{The Inequality}
Given $\varphi \in L^1_{\rm loc}(\bbR_+)$ and $z \in \bbC_+:= \{ z \in \bbC : \Im z >0\}$, the \emph{Weyl solution}  is the unique (up to a constant multiple) solution of $\Lambda_\varphi U = zU$ in $L^2(\bbR_+,\bbC^2)$. 
We denote the Weyl solution at spectral parameter $z \in \bbC$ by $\Psi(z,x)$ and normalize it by $\Psi_2(z,0)=1$ so that
\[ \Psi(0,z) = \begin{bmatrix} s(z,0) \\ 1 \end{bmatrix}\]
for a suitable $s(z,0) \in \bbC$. It turns out that $s(z,0)$ is an analtic function $\bbC_+ \to \bbD$, so we call it the \emph{Schur function}. We will sometimes use $\Psi(x,x_0,z)$ and $s(z,x_0)$ to make the boundary condition at $x_0$ explicit. 

Given $z \in \bbC$, let $T(z,x):\bbR \to \bbC^{2 \times 2}$ denote the transfer matrices given by finding a solution of the matrix-valued initial-value problem:
\[\Lambda_\varphi T = zT, \quad T(0) = I.\]
Naturally, $T$ has columns given by $T(z,x) = [U(z,x) \,\vert \, {V}(z,x)]$ where ${U}(z,x)$ and ${V}(z,x)$ solve $\Lambda_\varphi U = zU$ with initial conditions
\begin{equation}
    \begin{bmatrix}
        U_1(z,0) & V_1(z,0) \\ 
        U_2(z,0) & V_2(z,0)
    \end{bmatrix}
    = \begin{bmatrix}
        1 & 0  \\ 0 & 1
    \end{bmatrix}.
\end{equation}
For $\lambda \in \bbR$, $T(\lambda,x) \in \SU(1,1)$, so we  may write
\[ U(\lambda,x) = \begin{bmatrix}
    A(\lambda,x) \\ B^*(\lambda,x)
\end{bmatrix}, 
\quad V(\lambda,x) =
\begin{bmatrix}
    B(\lambda,x) \\ A^*(\lambda,x)
\end{bmatrix}\] with $|A|^2 - |B|^2 \equiv 1$
in that case.

We will focus on subordinacy on the right half-line, so we will drop the signs and simply write $\Psi$ and $s$ instead of $\Psi^+$ and $s_+$; the other half-line is similar.
For a function $f:\bbR_+ \to \bbC$ and $L >0$, we denote the partial $L^2$ norm on the interval $[0,L]$ by
\begin{equation}
    \|f\|_L:=
    \left[ \int_0^L |f|^2 \, \rmd x \right]^{1/2}
\end{equation}

\begin{definition}
A solution $V$ to the equation $\Lambda_\varphi V=\lambda V$ is said to be {\it subordinate} at $\infty$ if $\frac{\Vert U\Vert_L}{\Vert V\Vert_L}\to\infty$ as $L\to\infty$ for any solution $U$ that is linearly independent with respect to $V.$ Similarly, a solution that is subordinate at $-\infty$ can be defined in this way.
\end{definition}
For each $L>0$ and $\xi \in \partial \bbD$, we define $\epsilon = \epsilon(\xi,L)>0$ by
\begin{equation} \label{eq:JLIDiracEpsilonchoice}
    \|A+\xi B\|_{L}\|A-\xi B\|_L
    = \frac{1}{2\epsilon}.
\end{equation}
The reader can verify that for each fixed $\xi$, the left-hand side of \eqref{eq:JLIDiracEpsilonchoice} is a strictly increasing and continuous function of $L$ that vanishes for $L=0$ and which tends to infinity as $L\to \infty$ so that $\epsilon(\xi,L)$ is well-defined.

\begin{theorem}[Jitomirskaya--Last Inequality] \label{t:JLDiracGenBC}
For absolute constants\footnote{The proof gives $c_\pm = 3 \pm \sqrt 8$.} $c_\pm >0$, one has
    \begin{equation}
    c_- \frac{\|A-\xi B\|_L}{\|A+\xi B\|_L} \leq 
        \left| \frac{1+\xi s(\lambda+i \epsilon(\xi,L)}{1 - \xi s(\lambda+i \epsilon(\xi, L)} \right | \leq c_+ \frac{\|A-\xi B\|_L}{\|A+\xi B\|_L}
    \end{equation}
    for any $\xi \in \partial \bbD$.
    In particular, $s(\lambda+i 0) = \xi^*$ if and only if the eigensolution $W$ with initial condition $W(0)=\binom 1\xi$
    is subordinate at $+\infty$.
\end{theorem}
\begin{lemma} \label{lem:weylSolNormgenbc}
For each $z \in \bbC \setminus \bbR$,
\begin{equation} \label{eq:weylSolNormgenbc}
\|\Psi(z,\cdot)\|_{L^2([0,\infty))}^2
= \frac{1-|s(z,0)|^2}{ 2 \, \Im z} .\end{equation}
\end{lemma}
\begin{proof}
    Multiply both sides of $\Lambda_\varphi \Psi = z \Psi$ by $\Psi^*$, take imaginary parts, integrate from $0$ to $\infty$, and use $\Psi(0) = (s(z), 1)^\top$ to obtain \eqref{eq:weylSolNormgenbc}.
\end{proof}

\begin{proof}[Proof of Theorem~\ref{t:JLDiracGenBC}]
    Consider $z=\lambda + i \epsilon$, and denote
\[
v(x) = T(\lambda,x)^{-1} \Psi(z,x).
\]
Using primes to denote $x$ derivatives, observe that

\begin{align}
\nonumber
    v'(x)
    & = T(\lambda,x)^{-1} T'(\lambda,x)T(\lambda,x)^{-1} \Psi(z,x) + T(\lambda,x)^{-1}\Psi'(z,x) \\
    \nonumber
    & = T^{-1} i j
    \left(\lambda-
    \begin{bmatrix}
        0 & \varphi \\ \overline{\varphi} & 0 
    \end{bmatrix}
    \right)TT^{-1}\Psi + T^{-1} i j \left( z -\begin{bmatrix}
        0 & \varphi \\ \overline{\varphi} & 0 
    \end{bmatrix} \right)\Psi \\
    \label{eq:JLI:v'VOPa}
    & = i  (\lambda - z)T(\lambda,x)^{-1}j T(z,x)\begin{bmatrix} s(z) \\ 1 \end{bmatrix},
\end{align} 
here we recall that $j=\begin{bmatrix}-1&0\\0&1\end{bmatrix}$.
Integrating \eqref{eq:JLI:v'VOPa} from $0$ to $x$ gives
\[
v(x) = \begin{bmatrix}
s(z) \\
1
\end{bmatrix}
- i \epsilon \int_0^x 
T(\lambda,y)^{-1} j T(z,y) \begin{bmatrix}
s(z) \\
1
\end{bmatrix} \, \rmd y,
\]
and multiplying by $T(\lambda,x)$ gives 
\begin{equation} \label{eq:PsiVOPgenbc1} 
\Psi(z,x) = T(\lambda,x)  \begin{bmatrix}
s(z) \\
1
\end{bmatrix}-  i \epsilon \int_0^x 
T(\lambda,x) T(\lambda,y)^{-1} j \Psi(z,y) \, \rmd y.
\end{equation}

Now, given $\xi \in \partial \bbD$, introduce
\begin{equation} \label{eq:JLI:Wxidef}
W(\lambda,x)=
W_\xi(\lambda,x)
=T(\lambda,x) \begin{bmatrix}
1 & 1 \\
\xi & -\xi
\end{bmatrix}.
\end{equation}
Since
\[
T(\lambda,x)= \frac {1}{2}
W_\xi(\lambda,x)  \begin{bmatrix}
1 & \xi^* \\
1 & -\xi^*
\end{bmatrix} ,
\]
\eqref{eq:PsiVOPgenbc1} becomes
\begin{equation}
    \label{eq:PsiVOPgenbc2}
\Psi(z,x) =\frac{1}{2}  W(\lambda,x)  \begin{bmatrix}
s(z) + \xi^* \\
s(z) - \xi^*
\end{bmatrix}- i \epsilon \int_0^x 
W(\lambda,x) W(\lambda,y)^{-1} j \Psi(z,y) \, \rmd y.
\end{equation}

Introducing $C = A+\xi B$ and $D = A-\xi B$, \eqref{eq:JLI:Wxidef} gives
\begin{equation} \label{eq:JLI:WxiCD} W(\lambda,x) = \begin{bmatrix} A(\lambda,x) & B(\lambda,x) \\ B^*(\lambda,x) & A^*(\lambda,x)\end{bmatrix}\begin{bmatrix}
    1 & 1 \\ \xi & -\xi
\end{bmatrix}
= \begin{bmatrix}
    C & D \\ \xi  C^* & -\xi  D^*
\end{bmatrix} \end{equation}
Suppress  $\lambda$ from the notation and observe
\begin{align*}
    W(x)W(y)^{-1} j
    & = 
    \begin{bmatrix} C(x) & D(x) \\ \xi C^*(x) & -\xi D^*(x) \end{bmatrix}
    \left(-\frac{1}{2\xi} \right)
    \begin{bmatrix} -\xi D^*(y) & -D(y) \\- \xi C^*(y) & C(y) \end{bmatrix}
    \begin{bmatrix} -1 & 0 \\ 0 & 1 \end{bmatrix}\\
    &=
\frac12    \begin{bmatrix} C(x) & D(x) \\ \xi C^*(x) & -\xi D^*(x) \end{bmatrix}
    \begin{bmatrix} -D^*(y) & \xi^* D(y) \\ -C^*(y) & -\xi^* C(y) \end{bmatrix}
    \\
    & = \frac12 \begin{bmatrix}
        -C(x)D^*(y) - D(x)C^*(y)
        & \xi^* (C(x)D(y) - D(x)C(y)) \\
        \xi(D^*(x)C^*(y) - C^*(x)D^*(y))
        & C^*(x)D(y) + D^*(x) C(y)
    \end{bmatrix}
\end{align*}
This enables us to estimate the norm as follows:
\begin{align*}
    \|W(x)W(y)^{-1} j\|
    & \leq |C(x)D(y)| + |C(y)D(x)|.
\end{align*}

For $x \in [0,L]$, it follows in turn that
\begin{align*}
    \left\|\int_0^x 
W(\lambda,x) W(\lambda,y)^{-1} j \Psi(z,y) \, \rmd y \right\|
& \leq \int_0^x \|
W(\lambda,x) W(\lambda,y)^{-1}j\|\| \Psi(z,y)\| \, \rmd y \\
& \leq \|\Psi\|_L \big( |C(x)|\|D\|_L + |D(x)|\|C\|_L \big).
\end{align*}

Rearranging \eqref{eq:PsiVOPgenbc2} and integrating from $0$ to $L$, we conclude:
\begin{equation} \label{eq:PsiVOPgenbcwithCD}
\left\| \frac12  W(\lambda,x)  \begin{bmatrix}
s(z) + \xi^* \\
s(z) - \xi^*
\end{bmatrix} 
\right \|_L
\le \|\Psi\|_L + 2 \epsilon \| C \|_L \|D \|_L \| \Psi \|_L.
\end{equation}
Choosing 
\[
\epsilon(L) = \frac 1{ 2 \| C \|_L \| D \|_L }, \quad z = \lambda + i \epsilon(L),
\]
we combine \eqref{eq:PsiVOPgenbcwithCD} Lemma~\ref{lem:weylSolNormgenbc} to get
\begin{equation} \label{eq:PsiVOPgenbc4}
\left\| \frac12  W(\lambda,x)  \begin{bmatrix}
s(z) + 1 \\
s(z) - 1
\end{bmatrix} 
\right\|_L^2
\le 4 \| \Psi \|_L^2 
\le 2 \frac{ 1-  \lvert s(z) \rvert^2}{\Im z}  
\le 4 \| C \|_L \| D \|_L (1-  \lvert s(z) \rvert^2).
\end{equation}

Denote
\[
m = i \frac{1+\xi s}{1-\xi s}, 
\]
and observe that
\[
\Im m = \frac{ 1 - | s|^2}{|1-\xi s|^2}, \quad 
\begin{bmatrix}
    s+\xi^* \\ s-\xi^*
\end{bmatrix}
=
\xi^* (1-\xi s)
\begin{bmatrix}
    -i m \\ -1
\end{bmatrix},
\]
so dividing \eqref{eq:PsiVOPgenbc4} by $|1-\xi s|^2$ gives
\begin{equation} \label{eq:DiracJSCDineqgenbc1}
\left\| \frac12  W(\lambda,x)  \begin{bmatrix}
-i  m\\
-1
\end{bmatrix} 
\right|_L^2
\le 4 \| C \|_L \| D \|_L \Im m(z) \le 4 \| C \|_L \| D \|_L | m(z) |
\end{equation}
On the other hand, using the form of $W$ once more, we note that
\begin{equation}\label{eq:DiracJSCDineqgenbc2}
\left\| \frac12  W(\lambda,x)  \begin{bmatrix}
-i  m\\
-1
\end{bmatrix} 
\right\|_L 
\ge \big\lvert  \| C \|_L |m(z)| - \| D \|_L \big\rvert.
\end{equation}
Combining \eqref{eq:DiracJSCDineqgenbc1} and \eqref{eq:DiracJSCDineqgenbc2} produces
\[
\left( \| C \|_L |m(z)| - \| D \|_L \right)^2 \le 4 \| C \|_L \| D \|_L | m(z) |
\]
Divide by $\|D\|_L^2$ to obtain a quadratic inequality for $\kappa = |m(z)| \| C \|_L /  \| D \|_L$:
\[
 \kappa^2  - 6 \kappa + 1 \le 0,
\]
which leaves us with
\[ 3 - \sqrt{8} \leq  \kappa \leq 3+\sqrt{8}\]
that is,
\begin{equation}
\left| \frac{1+ \xi s(\lambda+i \epsilon(\xi,L))}{1-\xi s(\lambda+i \epsilon(\xi,L))} \right | \sim \frac{\|A-\xi B\|L}{\|A+ \xi B\|_L}
\end{equation}
\end{proof}

\subsection{Quantitative Consequences}

Let us write 
\begin{equation} \label{eq:JLI:Fxidef}
    F_\xi(z) = \frac{1+\xi s(z)}{1-\xi s(z)}.
\end{equation}

\begin{lemma} \label{t:JLIquant1}
For each $\xi \in \partial \bbD$ and $\epsilon>0$,
    \[    |F_\xi(\lambda + i \epsilon)|
    \leq c_+ \sup_{0\le x \le \frac{1}{2\epsilon}} \|T(\lambda,x)\|^2
\]
\end{lemma}
\begin{proof}
Given $\xi$ and $\epsilon$, choose $L$ so that  $\epsilon = \epsilon(\xi,L)$ as in \eqref{eq:JLIDiracEpsilonchoice}. Together with Theorem~\ref{t:JLDiracGenBC}, this gives
\begin{equation}
2\epsilon c_- \|A-\xi B\|^2_L \leq
    |F_\xi(\lambda + i \epsilon)| \leq 2\epsilon c_+ \|A-\xi B\|^2_L.
\end{equation}
Note that
\[ \begin{bmatrix}
    A(x) - \xi B(x) \\ B^*(x) - \xi A^*(x)
\end{bmatrix}
= T(\lambda,x) \begin{bmatrix}
    1 \\ - \xi
\end{bmatrix},\]
so (using $|\xi|=1$)
\begin{align*}
    2|A(x) - \xi B(x)|^2
    & = |A(x) - \xi B(x)|^2 + |B^*(x) - \xi A^*(x)|^2 \\
    & \leq \|T(\lambda,x)\|^2 (1+|\xi|^2).
\end{align*}
Consequently,
\[ \|A-\xi B\|_L^2
= \int_0^L |A(x)-\xi B(x)|^2 \, \rmd x
 \leq \int_0^L \|T(\lambda,x)\|^2 \, \rmd x.\]
From \eqref{eq:JLI:Wxidef}, \eqref{eq:JLI:WxiCD}, and $\det T = 1$, we obtain
\begin{align*}
    C^*D + CD^*
    = -\frac{1}{\xi} \det(W) 
    = 2.
\end{align*}
Combining this with Cauchy--Schwarz and the choice of $L$,
\begin{align*}
L
 = \frac{1}{2} \int_0^L (CD^* + C^* D)  
 \leq \left[\int_0^L |C|^2\right]^{1/2} \left[\int_0^L |D|^2\right]^{1/2}
\leq \|C\|_L \|D\|_L
= \frac{1}{2\epsilon}.
\end{align*}

Combining these ingredients together yields
\begin{align*}
    |F_\xi(\lambda + i \epsilon)|
    \leq 2\epsilon c_+ \int_0^{\frac{1}{2\epsilon}} \|T(\lambda,x)\|^2 \, \rmd x
    \leq c_+ \sup_{0\le x \le \frac{1}{2\epsilon}} \|T(\lambda,x)\|^2
\end{align*}
as desired.
\end{proof}

We now aim to connect this back to the spectral measure. 
Recall that the Schur function $s$ of the Dirac operator is related to its Herglotz function $m$ by
\[
s(z) =- \frac{1+i m(z)}{1-i m(z)}
\]
 Let $m_\pm(z)$ be the Weyl function of the half-line operators $\Lambda_\varphi|_{\bbR_{\pm}}$. It is well known that they are Herglotz functions that map the upper half complex plane $\bbC_+$ to itself. 
 Let $\mu$ be the canonical maximal spectral measure for the Dirac operator, and define the Borel transform of $\mu$ as
    $$\mathcal{M}(z)=b + \int \left( \frac1{x-z} - \frac{x}{1+x^2} \right) \,\rmd \mu(x),$$
    It follows from formula $(2.78b)$ of  \cite{ClarkGesztesy} that \begin{equation}
    \mathcal{M}(z)=\frac{m_+m_--1}{m_++m_-}.
    \end{equation}
Once we have these, we have the following fundamental observation: 

\begin{theorem}\label{thm.acMeas}
Let $\mu$ be the maximal spectral measure of the full line Dirac operator $\Lambda_\varphi$. For all $\lambda\in\mathbb{R}$ and $\epsilon>0$, we have
\begin{equation}
  \frac{1}{2\epsilon}  \mu((\lambda-\epsilon,\lambda+\epsilon))
    \leq 2c_+ \sup_{0\le x \le \frac{1}{2\epsilon}} \|T(\lambda,x)\|^2.
\end{equation}
\end{theorem}

\begin{proof}
 
    Note that we have $$\mu((\lambda-\epsilon,\lambda+\epsilon))\leq 2\epsilon \Im \mathcal{M}(\lambda+i\epsilon)$$ and $$\Im \mathcal{M}=\frac{\Im m_+\Im m_-}{|m_++m_-|^2}\left(\frac{|1+m_+|^2}{\Im m_+}+\frac{|1+m_-|^2}{\Im m_-}\right).$$
    Since $\Im m_\pm(z)>0$ for $z\in\bbC_+,$ $\frac{\Im m_+\Im m_-}{|m_++m_-|^2}\leq \frac{1}{2}$. Then we have 
    $$\Im \mathcal{M}\leq \frac{1}{2}\left(\frac{|1+m_+|^2}{\Im m_+}+\frac{|1+m_-|^2}{\Im m_-}\right). $$
    It suffices to give an estimate of $\frac{|1+m|^2}{\Im m}$, where $m=m_\pm.$ We will prove the case $m=m_+$, the other case can be proved in a similar way. 

    By Lemma \ref{t:JLIquant1}, for any $\xi\in \partial\bbD$, denote $m=m_\xi=i\frac{1+\xi s_+}{1-\xi s_+}$, we have
    $$|m_\xi(\lambda+i\epsilon)|\leq 2c_+\epsilon\sup_{0\leq x\leq 2^{-1}\epsilon^{-1}}\Vert T(\lambda,x)\Vert^2.$$
    Assume that $\xi=-e^{i\zeta},\xi\neq 1$, since $s_+=\frac{m_+-i}{m_++i}$, 
    $$m_\xi(z)=i\frac{1+\xi\frac{m_1-i}{m_1+i}}{1-\xi\frac{m_1-i}{m_1+i}}=\frac{i(1+\xi)m_1-(1-\xi)}{(1-\xi)m_1+i(1+\xi)}=R_\beta\cdot m_1,$$
    where $\beta=\frac{\zeta-\pi}{2}.$ Since the function $f(z)=\frac{|1+z|^2}{\Im z},$ $z\in\bbC_+$ satisfies $f(R_\beta\cdot z)=f(z)$; compare \cite[Section 2.2]{Avila1}, we may choose $\xi$ to maxmize $\Im m_\xi$ such that $m_\xi$ is purely imaginary. Therefore, $\Im m_\xi\geq 1$ for such $\xi.$ It follows that $$f(m_\xi)\leq 2\Im m_\xi\leq 4c_+\epsilon\sup_{0\leq x\leq (2\epsilon)^{-1}}\Vert T(\lambda,x)\Vert^2.$$ 
    This completes the proof.
\end{proof}

\subsection{Gap Edges} 
We will conclude this section with a general observation about gap edges. 
The proof requires a general fact about Herglotz functions:

\begin{lemma}\label{lemmaHerglotzGapEdge}
If $f$ is a Herglotz function with a meromorphic continuation to $\bbC_+ \cup (a,b) \cup \bbC_-$ such that $\ol{f(z)}= f(\ol z)$, and if $f$ has finitely many poles in $(a,b)$, the limits $\lim_{\epsilon \downarrow 0} f(b-\epsilon)$ and $\lim_{\epsilon \downarrow 0}  f(b+i \epsilon)$ exist and are equal.
\end{lemma}

\begin{proof}
By altering $a$, we may assume that $f$ has an analytic continuation on $\bbC_+ \cup (a,b) \cup \bbC_-$ such that $\ol{f(z)}= f(\ol z)$. This is equivalent to saying that its Stieltjes measure puts no weight on $(a,b)$, and it implies that $f$ is strictly increasing on $(a,b)$. In particular, the limit $\lim_{\epsilon \downarrow 0} f(b-\epsilon)$ exists.

If $f$ satisfies the assumptions of the lemma, so does $-1/f$. In particular, by possibly passing to $-1/f$, we may assume without loss of generality that the limit $\lim_{\epsilon \downarrow 0} f(b-\epsilon)$ is finite.

In the Herglotz representation of $f$, the affine part and the part of the measure supported away from $b$ give contributions continuous at $b$, so they can be neglected (compare, e.g., \cite[Remark 7.36]{Lukic22}) and we may assume that $f$ is of the form 
\[
f(z) = \int \frac 1{\xi - z} \,\rmd \mu(\xi)
\]
with $\mu$ a finite measure on $[b,b+1]$. By monotone convergence,
\[
\lim_{\epsilon \downarrow 0} f(b-\epsilon)= \int \frac 1{\xi - b} \,\rmd \mu(\xi)
\]
and we have assumed this is finite. Thus,
\[
\left\lvert f(b+i \epsilon) - \int \frac 1{\xi-b} \,\rmd \mu(\xi) \right\rvert 
\le \int \left\lvert \frac 1{\xi-b-i \epsilon} -  \frac 1{\xi-b}  \right\rvert \,\rmd \mu(\xi) 
= \int \frac \epsilon{\lvert \xi - b-i \epsilon\rvert} \frac{\rmd \mu(\xi)}{\xi - b}.
\]
By dominated convergence with dominating function $1/(\xi-b)$, this converges to zero as $\epsilon\downarrow 0$.
\end{proof}

\begin{coro}
At any gap edge of the spectrum of a half line Dirac operator $\Lambda_\varphi$, i.e., a point $\lambda\in \sigma(\Lambda_\varphi)$ such that there exists $\epsilon > 0$ so that $(\lambda-\epsilon, \lambda) \cap \sigma(\Lambda_\varphi)=\emptyset$ or $(\lambda, \lambda +\epsilon) \cap \sigma(\Lambda_\varphi)=\emptyset$, there exists a subordinate solution. 
\end{coro}

\begin{proof}
By Lemma~\ref{lemmaHerglotzGapEdge}, the Herglotz function has a limit in $\bbR \cup \{\infty\}$. Denote it by $m(\lambda+i 0) = \cot(\alpha/2)$ for some $\alpha \in [0,2\pi)$. Then the Schur function has a limit $s(\lambda+i 0)=e^{-i \alpha}$, so the eigensolution with initial condition $\binom 1{e^{i \alpha}}$ is subordinate.
\end{proof}

\section{Inverse Spectral Theory of Reflectionless Operators} 
\label{canonicalSys}

In this section, we review aspects of the inverse spectral theory of reflectionless operators, as it pertains to Dirac operators. The highlight of this section is a new understanding of the extra component of the generalized Abel map, see Definition \ref{def:rotFlowAbelMap} and Theorem \ref{thm:identifyRotCord}.

\subsection{Dirac Operators in Arov Gauge}
Let $\mathscr{U}(z,x)$ be the fundamental solution matrix of the equation $\Lambda_{\varphi}X = zX$ satisfying the initial condition $\mathscr{U}(z,0)=I.$
Then $\mathscr{U}(z,\cdot)$ is a solution of the following canonical system in Arov gauge:
\begin{equation}\label{eq:ArovDirac}
\mathscr{U}(z,\ell)j
=j+\int_{0}^{\ell}\mathscr{U}(z,l)( i z\mathcal{P}(l)-\mathcal{Q}(l)) \, \rmd\mu(l), z\in\bbC,
\end{equation}
where $j=\begin{bmatrix}-1&0\\0&1\end{bmatrix}$, $\mathcal{P},\mathcal{Q}$ are $2\times2$ locally integrable matrix coefficients with respect to a Borel measure $\mu$ on $\bbR$ with $\mathcal{P}\geq 0, \mathcal{Q}=-\mathcal{Q}^*, \mathrm{trace}(j\mathcal{P})=\mathrm{trace}(j\mathcal{Q})=0$.

The integral equation takes account for singular $\mu$. The corresponding initial value problem is the following
\begin{equation}\label{eq:ArovDiff}
\partial_{\mu}\mathcal{U}(z,t)j=\mathcal{U}(z,t)(i z\mathcal{P}(t)-\mathcal{Q}(t)),~ \mathcal{U}(z,0)=I,
\end{equation}
where $\mathcal{U}$ is assumed to be absolutely continuous with respect to $\mu$ and $\partial_{\mu}\mathcal{U}$ is its Radon-Nikodym derivative with respect to $\mu.$ The classical Dirac operator with potential $\varphi$ corresponds to the canonical system \eqref{eq:ArovDiff} with $\mu(\ell)=\ell$, $\mathcal{P}=I$ and  
\[ \mathcal{Q} = \begin{bmatrix} 0 & -i \varphi\\ i \overline{\varphi} & 0\end{bmatrix}.\]
Interested readers may refer to \cite[Section 7]{eichinger2024necessarysufficientconditionsuniversality} for a discussion on preserving Weyl functions while using different conventions.
\begin{definition}
Let $j$ be a $2\times 2$ matrix with $j=j^*=j^{-1}$. An entire $2\times 2$ matrix function $A(z)$ is called $j$-inner if it obeys $j-A(z)jA(z)^{*}\geq 0$ for all $z\in\bbC^+$ and $j-A(z)jA(z)^*=0$ for $z\in\bbR.$
\end{definition}

\begin{definition}
A family of  matrix functions $A(z,\ell)$ parameterized by a real parameter $\ell$ is called $j$-monotonic if $A(z,\ell_1)^{-1}A(z,\ell_2)$ is $j$-inner whenever $\ell_1<\ell_2.$
\end{definition}

Define the associated {\it Weyl disks} by \begin{equation}\label{eq.WeylDisk}
D(z,\ell)=\{w\in \bbC:\left[w\quad 1\right]A(z,\ell)jA(z,\ell)^*\left[w\quad 1\right]^*\geq 0\}.
\end{equation}
The $j$-monotonicity implies the nesting property of Weyl disks:
$D(z,\ell_2)\subseteq D(z,\ell_1)$ for $\ell_2>\ell_1$;  compare \cite[(2.3), (2.4)]{EGL} and the surrounding discussion.
Note that  $D(z,\ell)$ is a subset of $\overline{\bbD}$ due to $A(z,0)=I.$ 

It turns out that $\cap_{\ell\in(0,\infty)}D(z,\ell)$ is either a disk or a single point, which is known as the limit circle/point dichotomy. If $\sup_x\int_x^{x+1}|\varphi(t)| \, \rmd t<\infty$, then we are at the limit point case \cite{EGL}. Define
$$s_+(z)=\cap_{\ell\in(0,\infty)}D(z,\ell)$$
to be the unique common point.
The function $s_+$ is known to be a Schur function, that is, an analytic map $\bbC^+\to\bbD.$
There is an analogous definition of $s_-$ for the left-half-line, which can be obtained by reflecting through the origin.

\subsection{Reflectionless Potentials}
Let $\sfE$ be the spectrum of the system \eqref{eq:ArovDirac} and $\Omega=\bbC\setminus \sfE$.
Since $\sfE \subseteq \bbR$ is closed, we can write its complement (in $\bbR$) as a countable disjoint union
\[
\bbC\setminus \sfE = \bigcup_j (a_j,b_j).
\]
\begin{definition}[Reflectionless]\label{def:reflectionless}
The pair of Schur functions $(s_+,s_-)$ is said to be a \emph{reflectionless pair} with spectrum $\sfE$ if both  extend to meromorphic single-valued functions on $\Omega$ with properties:
\begin{enumerate}[label={\rm(\alph*)}]
    \item symmetry property $\overline{s_{\pm}(\overline{z})}=1/s_{\pm}(z)$;
    \item reflectionless property $\overline{s_+(\xi+ i 0)}=s_{-}(\xi + i 0)$ for a.e. $\xi\in \sfE$;
    \item $1-s_+s_-\neq 0$ in $\bbR\setminus \sfE.$
\end{enumerate}
\end{definition}
Let $\mathcal{S}(\sfE)$ be the set of $s_+$ belonging to such pairs with the topology of locally uniform convergence of $\bbC$-valued maps on $\Omega$. 
Let $\mathcal{S}_{A}(\sfE)$ be the set of $s_+$ such that the corresponding $s_-$ in the pair satisfies the normalization condition $s_-(i)=0,$ which is a compact subset of $\mathcal{S}(\sfE)$, compare \cite{BLY2}.

\begin{definition}
The potential $\varphi\in L^1_{\rm loc}$ belongs to $\mathcal{R}(\sfE)$ if the spectrum of of its associated  Dirac operator $\Lambda_\varphi$ equals $\sfE$ and its associated Schur function $s_+$ belongs to $\mathcal{S}_{A}(\sfE)$.
In general, if $\sfE \subseteq \sigma(\Lambda_\varphi)$ and
\[ \overline{s_+(\xi+ i 0)}=s_{-}(\xi + i 0) \text{ for a.e.\  }\xi\in \sfE, \] 
we say that $\varphi$ is \emph{reflectionless on} $\sfE$.\footnote{According to \cite{PC09CMP}, a reflectionless measure whose support is a weakly homogeneous set $\sfE\subset \bbR$ is always absolutely continuous. In this work, $\sfE$ is assumed to be homogeneous, so we drop the condition that for every $\varphi\in\mathcal{R}(\sfE)$ $\Lambda_\varphi$ has purely a.c. spectrum; compare \cite{VolbergYuditskii2014, DamanikYuditskii2016Adv}.}
\end{definition} 
We endow $\mathcal{R}(\sfE)$ with the {\it strong resolvent topology.}
The resolvent function of $\Lambda_\varphi$  is related to the Schur functions by the following  \begin{equation}\label{eq:resolventFunc}
R(z) = i \frac{(1-s_+(z))(1-s_-(z))}{1-s_+(z)s_-(z)}.
\end{equation}
Note that it differs from the previously defined Green's function $G_{1,1}$ in \eqref{eq.GreenFunc} by a constant factor $-2$. 

The following notions are standard, e.g. \cite{Garnett2007Book,Hasumi1983}.
\begin{definition}
    Let $\Omega$ be a given domain, $\mathcal{N}(\Omega)$ be the Nevanlinna class, and $\mathcal{N}^+(\Omega)$ be the Smirnov class. 
    Denote the anti-linear involution $\sharp$ acting on functions by 
\begin{equation}\label{eq:involution}
f_\sharp(z)
=\overline{f(\bar{z})}.\end{equation}
We will also need the {\it pseudocontinuation} of a function. Let $f\in \mathcal{N}(\Omega)$, we say $f$ has a pseudocontinuation if there exists $f_*\in \mathcal{N}(\Omega)$ such that
\begin{equation}\label{eq:pseudoCont}
    f_*(z)=\overline{f(z)}\quad \text{for a.e. }z\in\partial\Omega.
\end{equation}
\end{definition}

\begin{lemma}[Bessonov--Luki\'c--Yuditskii \cite{BLY2}]
If $s_+\in\mathcal{S}(\sfE)$, then $R(z)$ is a Herglotz function, analytic on $\bbC\setminus \sfE$, with the symmetry $R_{\sharp}=R$ and satisfies $\arg R(\xi + i 0)=\frac{\pi}{2}$ for a.e.\ $\xi\in \sfE.$
\end{lemma}
\begin{remark}
Let $m_{\pm}(z)$ be the Weyl-Tichmarsh $m$-function.
Definition~\ref{def:reflectionless} converts to the Weyl-Tichmarsh form by the substitutions \begin{equation}\label{eq:mFunc}
m_\pm(z)
= i \frac{1+s_\pm(z)}{1-s_\pm(z)},~R(z)=\frac{-2}{m_+(z)+m_-(z)}.
\end{equation}
 Condition (2) of Definition~\ref{def:reflectionless} is then equivalent to \begin{equation}\label{eq:reflectionlessMfunc}
m_+(\xi+i0)=-\overline{m_-(\xi+i0)}.
\end{equation}
In particular, a zero of $R(z)$ corresponds to a pole of $m_{\pm}(z)$.
\end{remark}
Reflectionlessness in the absolutely continuous part of the spectrum is known to be true:

\begin{theorem}[Bessonov--Luki\'c--Yuditskii \cite{BLY1}]\label{thm:reflecAc}
If the potential $\varphi$ is uniformly almost periodic,  then it is reflectionless on its a.c.\ spectrum, that is, $\overline{s_+(\xi+i 0)} = s_-(\xi+ i 0)$ for a.e.\ $\xi \in \Sigma_{\varphi,{\rm ac}}$.
\end{theorem}

\subsection{The Abel Map}

Let us first introduce the explicit construction of $\mathcal{D}(\sfE)$ and the map from $\mathcal{S}_A(\sfE)$ to $\mathcal{D}(\sfE)$. 
\begin{definition}[Divisors]\label{def:divisor}
Let $\sfE$ be as above, define $\mathcal{D}(\sfE)$  by
\begin{equation}\label{eq:divisor}
\mathcal{D}(\sfE)
=\prod_{j\in\bbN}\mathcal{I}_j,
\end{equation}
where $\mathcal{I}_j$ is a circle corresponding to gluing two copies of the $j$th gap together at the endpoints, that is,
$$
\mathcal{I}_j
=([a_j,b_j] \times \{\pm\}) / \sim,$$
where the equivalence relation is given by $(a_j,+) \sim (a_j,-)$ and $(b_j,+) \sim (b_j,-)$.
\end{definition}
Let us describe an explicit construction of the map $\mathcal{B}:\mathcal{S}_A(\sfE)\to\mathcal{D}(\sfE)$.
Since $R$ is strictly increasing in each gap $(a_j,b_j)$ \cite[Lemma 3.9]{BLY2}, there exists at most one $\mu_j\in(a_j,b_j)$ such that $R(\mu_j)=0$. 
Suppose first that $R(\mu_j)=0$ for some $\mu_j \in (a_j,b_j)$.
Since $1-s_+(z)s_-(z)\neq 0$, this implies  either $s_+(\mu_j)=1$ or $s_-(\mu_j)=1$,
so in this case, $(\mu_j,\epsilon_j)$ is determined by $s_{\epsilon_j}(\mu_j)=1$. 
If $R$ never vanishes in $(a_j,b_j)$ then either $R(z)>0$ on $(a_j,b_j)$, in which case we put $\mu_j=a_j$, or $R(z)<0$ on $(a_j,b_j)$, in which case we put $\mu_j=b_j$.
The above construction then defines the map $\mathcal{S}_A(\sfE)\to \mathcal{D}(\sfE)$:
\begin{equation}\label{eq:traceFormula1}\mathcal{B}: s_+ \to \mathcal{B}(s_+)=D:=\{(\mu_j,\epsilon_j):j\in\bbN\}.
\end{equation}

The map $\mathcal{B}$ can also be understood through the product formula for $R$.
With $\mu_j$ as in the previous paragraph, let 
\begin{equation}\label{eq:divisor1}
\chi(\xi)=\left\{\begin{aligned}
&1/2,&& \exists j: \xi \in(a_j,\mu_j), \\
&-1/2,&& \exists j: \xi \in(\mu_j,b_j), \\
&0,&& \xi \in\sfE.
\end{aligned}
\right.
\end{equation}
In this case, the resolvent function has the Herglotz representation on $\bbC_+$:
\begin{equation*}
R(z) 
= i |R(i)|e^{\int \! \left(\frac{1}{\xi-z}-\frac{\xi}{1+\xi^2}\right)\chi(\xi) \, \rmd\xi}.
\end{equation*}
In case of finite gap length (i.e., when \eqref{eq:finiteGapCond} holds) the two parts of the integrand are separately integrable; combining this with the normalization $R(z) \to i$ as $z\to i\infty$, we conclude
\[
R(z) = i e^{\int\! \frac 1{\xi - z} \chi(\xi) \, \rmd\xi}
\]
which also gives a product representation:
\begin{equation}\label{eq:resolventProd}
R(z)
= i \prod_j\sqrt{\frac{(z-\mu_j)^2}{(z-a_j)(z-b_j)}}.
\end{equation}

In order to introduce the construction of the Abel map in \cite{BLY2}, we need some assumptions on the unbounded closed set $\sfE=\bbR\setminus \cup_{j}(a_j,b_j)$ and the associated Denjoy domain $\Omega=\bbC\setminus \sfE$. 
Notice that $\infty\in\partial\Omega$. 
Fix a spectral gap $(a_0,b_0)$ and pick $\xi^*\in(a_0,b_0)$ for the purpose of normalization.
 Let $\pi_1(\Omega)$ be the fundamental group of the domain $\Omega$ and $\bbT:=\bbR/\bbZ.$ Let $\pi_1(\Omega)^*$ be the group of characters, that is, the set of $\alpha:\pi_1(\Omega)\to\bbT$ satisfying $\alpha(\gamma_1\circ\gamma_2)=\alpha(\gamma_2)\alpha(\gamma_1)$ for $\gamma_1,\gamma_2\in\pi_1(\Omega)$.
\begin{definition}
The domain $\Omega$ is said to be {\it Dirichlet regular} if the potential-theoretic Green's function $G(z,z_0)$  with a logarithmic pole at $z_0$ extends continuously and vanishes everywhere on $\partial\Omega$, for some (and hence for all) $z_0 \in \Omega$.
\end{definition}

\begin{definition}\label{def:WidomDomain}
We say that $\Omega$ is  a {\it Parreau-Widom surface} (or just a {\it PWS}) if for some (hence for all) $z_0\in\Omega$, one has:
\begin{equation}\label{eq:WidomCond}
\sum_{c_j:\nabla G(c_j,z_0)=0}G(c_j,z_0)<\infty.
\end{equation}
 \end{definition}
Historically, this class was first introduced by M.\ Parreau in 1958 and perhaps independently by H.\ Widom in 1971 from different motivations. 
There are some other equivalent definitions, e.g.
$$\int_{0}^\infty B(\alpha,z_0)d\alpha<\infty,$$
where $B(\alpha,z_0)$ is the {\it first Betti number} of the connected domain $$\Omega(\alpha,z_0):=\{z\in\Omega:G(z,z_0)>\alpha\}$$
for each $\alpha>0.$
The fundamental result of Widom \cite{Widom1971Acta} says that the following are equivalent:
\begin{enumerate}
    \item $\Omega$ is a PWS
    \item $\mathcal{H}^\infty(\Omega,\alpha)\neq\{0\}$ for any $\alpha\in \pi_1(\Omega)^*$
    \item $\mathcal{H}^1(\Omega,\alpha)\neq\{0\}$ for any $\alpha\in \pi_1(\Omega)^*$
\end{enumerate}
where the character automorphic Hardy space $\mathcal{H}^p(\Omega,\alpha)=\{f\in \mathcal{H}^p(\Omega): f(\gamma(z))=e^{2\pi i\alpha(\gamma)}f(z)\}$.
Interested reader may refer to \cite[Theorem 2B of Chapter 5]{Hasumi1983} for detailed proof.

\begin{definition}\label{def:homogeneous}
 We say $\sfE \subseteq \bbR$ is {\it homogeneous} in the sense of Carleson if the following condition holds:
 \begin{equation}\label{eq:defHomogeneousSet}\inf_{\lambda\in\sfE}\inf_{0<h<\mathrm{diam}(\sfE)}\frac{|\sfE\cap(\lambda-h,\lambda+h)|}{2h}>0,
 \end{equation}
 where $|\cdot|$ denotes the Lebesgue measure on $\bbR$. 
 \end{definition}
Note that since the spectrum of the Dirac operator is unbounded, this definition is different from the one for bounded set such as the spectrum of Jacobi matrices:
$$\inf_{x\in\sfE}\inf_{0<h\leq 1}\frac{|\sfE\cap(x-h,x+h)|}{2h}>0.$$
Alternatively, for unbounded set $\sfE\subset\bbR$, one can also use the harmonic measure $\frac{1}{1+x^2}dx$ as seen from $i$.\footnote{This difference was pointed out to us by Peter Yuditskii.}

 Let $\sfE$ be a set of positive capacity such that $\Omega=\bbC\setminus \sfE$ is Dirichlet regular. Let $\widetilde{G}(z,z_0)$ be the harmonic conjugate of $G(z,z_0)$, and denote 
 $$
 \Phi_{z_0}(z)=e^{-G(z,z_0)-i\widetilde{G}(z,z_0)}
 $$
 the complex Green's function. Its lift to $\bbD$ is a Blaschke product with zeros $\{\Lambda^{-1}(z_0)\}$. 
 Note that $\widetilde{G}(z,z_0)$ is multi-valued and $|\Phi|$ is single-valued with 
 $$|\Phi_{z_0}(z)|=e^{-G(z,z_0)}.$$
 Denote also $\Phi(z)=\Phi(z,i )$, the complex Green function that has a simple zero at $i$.

We fix a point $\xi_* \in \bbC\setminus \sfE$ and the uniformization $\Lambda:\bbD\to\Omega$ such that $\Lambda(0)=\xi_*$. We let $\Gamma$ denote the corresponding Fuchsian group acting on the unit disk by M\"obius transformations; then $\Gamma \cong \pi_1(\Omega)$ and $\Lambda (z_1)=\Lambda (z_2)$ if and only if $z_2=\gamma(z_1)$ for some $\gamma \in \Gamma$.

A multi-valued function $f$ on $\Omega$ with single-valued modulus $|f|$ is lifted to a single-valued function $F$ on $\bbD$ such that 
 $F=f\circ\Lambda$. 
 There exists $\alpha \in \pi_1(\Omega)^*$ such that
 $$
 F \circ\gamma
 =e^{2\pi i \alpha(\gamma)}F
 \quad \forall \gamma \in \Gamma.
 $$
The function $f$ is said to be {\it character-automorphic} and $\alpha$ is its character.  
Denoting by $f\circ \gamma$ the analytic continuation of the multiplicative function $f$ along the element $\gamma$ of $\pi_1(\Omega)$, this corresponds to 
 $$f \circ \gamma 
 = e^{2\pi i \alpha(\gamma)}f \quad   \forall \gamma\in\pi_1(\Omega).$$
 In this case, we also call $f$ \emph{character automorphic}.

 \begin{definition}
 For every $\alpha\in\pi_1(\Omega)^*$, the character-automorphic Hardy space $\mathcal{H}^p(\Omega,\alpha) $ is the set of character-automorphic functions $f$ with character $\alpha$ whose lift $F = f\circ\Lambda$ belongs to $H^p(\bbD)$.
 \end{definition}
 
 For a Widom domain $\Omega$, $\mathcal{H}^2(\alpha)=\mathcal{H}^2(\Omega,\alpha)$ is a non-trivial reproducing kernel Hilbert space \cite{Widom1971Acta}. 
 For every $z_0\in\Omega$, there exists $k^\alpha_{z_0}(z)\in\mathcal{H}^2(\alpha)$ such that for any $f\in \mathcal{H}^2(\alpha)$,
 $$\langle f,k^\alpha_{z_0}\rangle=f(z_0).$$

 Namely, $k^\alpha_{z_0}(z)$ is the kernel of pointwise evaluation at $z_0$, which is a linear functional.

The \emph{Smirnov class} $\mathcal{N}_+(\Omega)$ consists of the analytic functions of the form $f/g$ where $f$ and $g$ are bounded and analytic and $g$ is outer. Let $M$ be the symmetric Martin function at the infinity of $\Omega$, $\Theta$ be an analytic function satisfying $\Im\Theta(z)=M(z)$. Denote $\{c_j\in (a_j,b_j):j\in\bbN\}$ the set critical points of $M$ and let $\mathcal{W}=\prod_j\Phi_{c_j}(z)$ be the Widom function with a character $\beta(\mathcal{W})$. For any $f\in \mathcal{H}^1(\Omega,\beta(\mathcal{W}))$, $\frac{f}{\mathcal{W}}d\Theta$ is a single valued differential in $\Omega$ and moreover $\frac{f\Theta'}{\mathcal{W}}\in\mathcal{N}_+(\Omega)$.
\begin{definition}[\cite{BLY2}]
We say that $\Omega$ satisfies the Direct Cauchy Theorem (DCT) if for every $f\in \mathcal{H}^{1}(\Omega,\beta(\mathcal{W}))$
\[
\oint_{\partial\Omega}\frac{f(z)}{\mathcal{W}(z)}d\Theta(z)=\oint_{\partial\Omega}\frac{f(z)\Theta'(z)}{\mathcal{W}(z)}dz=0.
\]
\end{definition}
There are various definitions of DCT, compare \cite{Hasumi1983, SY97, VolbergYuditskii2014, Yuditskii2011DCTFails,Yuditskii2020}. We are using the one that is consistent with our setting. 
There are examples of PWS that do not satisfy DCT: see \cite{Yuditskii2011DCTFails} and references therein.
  If $\sfE$ is homogeneous, then $\Omega = \bbC\setminus \sfE$ is known to be Dirichlet regular and satisfies DCT, but the converse is not always true; compare \cite{Yuditskii2011DCTFails, SY97}. 
 It is known that if $\Omega$ obeys DCT, then $k^\alpha_{z_0}(z_0)$ is a continuous function of $\alpha\in\pi_1(\Omega)^*$.
Define the $L^2$-normalized reproducing kernel at $z_0=i $ by $K^\alpha(z)=k^\alpha_{i }(z)/\sqrt{k^\alpha_{i }(i )}$.
Recall that $\sharp$ denotes the anti-linear involution \eqref{eq:involution} and that $\Phi(z)$ is the complex Green's function of $\Omega$ normalized at $i $. Let $\{c_j:j\in\bbN\}$ be the set of critical points of the Green's  function.
Given $\{(\mu_j,\epsilon_j):j\in\bbN\}\in \mathcal{D}(\sfE)$, the {\it canonical product} 
\begin{equation}\label{eq.prodRepRepKer}
\mathscr{K}_D(z)
=C\left\{\frac{\Phi(z)}{z-i } \frac{\Phi_{\sharp}(z)}{z+i } \prod_j \frac{z-\mu_j}{\sqrt{1+\mu_j^2}\Phi_{\mu_j}(z)} \frac{\sqrt{1+c_j^2} \Phi_{c_j}(z)}{z-c_j}\right\}^{\frac{1}{2}}\prod_j\Phi_{\mu_j}(z)^{\frac{1+\epsilon_j}{2}}.
\end{equation}
arises from the decomposition \begin{equation}\label{eq:ResolventDecomp}R^{\alpha,\tau}(z)=i\frac{(1-\tau s_+^\alpha)(1-\bar{\tau}s_-^\alpha)}{1-s_+^\alpha s_-^\alpha}=\frac{i}{2}\mathscr{K}\mathscr{K}_*(z-i)(z+i)\Theta',\end{equation}
where $R^{\alpha,\tau}(z)$ is the diagonal Green's function \eqref{eq:resolventFunc},  $\Theta(z)=\widetilde{M}(z)+iM(z)$ is the complex Martin function. It is well defined if $\Omega$ is of Parreau-Widom type and satisfies the {\it finite gap length condition}:
 \begin{equation}\label{eq:finiteGapCond}\sum_j(b_j-a_j)=\sum_j\gamma_j<\infty.\end{equation}
We will also need the $k$-th gap length condition ($k$-GLC)
\begin{equation}\label{eq:kGlc}
\sum_j|b_j^{k+1}-a^{k+1}_j|<\infty, \text{ for } k=0,1,\cdots
\end{equation}
In the current context, we only deal with space and time variables and thus only need $k$-GLC for $k=0,1$. 
However,  for higher AKNS hierarchies, the case $k\geq 2$ will be needed.
Since $\Omega$ is multi-connected, $\Theta$ is multi-valued obeying 
\begin{equation}\label{eq:charOfCompCMartin}
\Theta\circ\gamma(z)=\Theta(z)+2\pi\eta(\gamma), \gamma\in \Gamma=\pi_1(\Omega),
\end{equation}
where $\eta:\Gamma\to\bbR$ is an additive character.
It is shown \cite[Lemma 3.12]{BLY2} that $\mathscr{K}_D$ is orthogonal to $\Phi\Phi_{\sharp}\mathcal{H}^2(\alpha-2\beta_{\Phi})$ and therefore can be written as a linear combination $$\mathscr{K}^1_D=C_1K^\alpha-C_2K^\alpha_\sharp,$$
where by $\mathscr{K}_D^C$ we make the dependence of $C$ in \eqref{eq.prodRepRepKer} explicit. In particular, denote $\mathscr{K}_D=\mathscr{K}_D^{1}$. Note that $(\mathscr{K}_D)_\sharp$ is a multiple of $\mathscr{K}_D$, it follows that $|C_1|=|C_2|=1.$ Then it follows that 
\begin{equation}\label{eq.reasonOfRotation}
\mathscr{K}_D^{\overline{C_1}}=\overline{C_1}\mathscr{K}_D=K^\alpha-\tau K^\alpha_\sharp,
\end{equation}
where $\tau=\overline{C_1}C_2.$ Morally $\tau$ stands for a modulation of arguments between $K^\alpha$ and $K^\alpha_\sharp$. Now let $\alpha=\alpha(\mathscr{K}_D)$ be the character of this product.

Then above procedure gives the construction of the Abel map
$\mathcal{A}:\mathcal{D}\to \pi^*_1(\Omega)\times \bbT$
\begin{equation}
\label{eq:abelMap01}
D\mapsto (\alpha(\mathscr{K}_D),\tau).
\end{equation} 
To give the explicit formula of this map, let us consider a set of generators of $\pi_1(\Omega)$. Let $l_k$ be a simple loop that intersects $\bbR\setminus \sfE$ through $(a_0,b_0)$ at $\xi^*$ ``upward" and $(a_k,b_k)$ ``downward" as shown in Figure 1. 
Then $\{l_k\}$ is a generator of the fundamental group $\pi_1(\Omega)$.
\begin{figure}
\begin{tikzpicture}
\draw[dotted] (-3,0) -- (-2.1,0);
\draw[dotted] (2.1,0) -- (3,0);
\draw (-2,0) -- (-0.9,0);
\draw[dotted] (-0.7,0) -- (0.7,0);
\draw (0.9,0)--(2,0);
\draw (0,0) ellipse (1.2cm and 0.6cm);
\node at (-1.4,-0.7) {($a_0,b_0$)};
\node at (1.4,-0.7) {($a_k,b_k$)};
\node at (0,0.45) {$l_k$};
\node at (-1.2,0) {$\uparrow$};
\node at (1.2,0) {$\downarrow$};

\end{tikzpicture}
\caption{generators of $\pi_1(\Omega)$}
\end{figure}
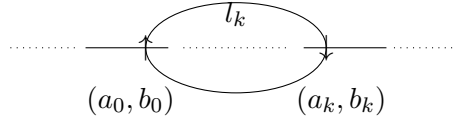

Let $\sfE_k$ be the subset of $\sfE$ that is inside $l_k$ and $\omega(z,F)$ is the harmonic measure of a subset $F\subseteq \sfE.$ Then denote 
\begin{equation}\label{eq:Abel1}
\mathcal{A}_c[D](l_k)=\sum_j\frac{\epsilon_j}{2}\int_{a_j}^{\mu_j}\omega(\rmd t,\sfE_k) \mod \bbZ.
\end{equation}

\begin{theorem}[Bessonov--Luki{\'c}--Yuditskii (2024) \cite{BLY2}]\label{thm:JacobiInver}
If $\Omega=\bbC\setminus \sfE$ is a regular Widom domain that obeys DCT, then
the Abel map $\mathcal{A}:\mathcal{D}(\sfE)\to\pi^*_1(\Omega)\times \bbT$ is a homeomorphism and can be given explicitly 
\begin{equation}\label{eq:Abel2}
\alpha(l_k)=\beta_\Phi(l_k)+\mathcal{A}_{c}[D](l_k)-\mathcal{A}_{c}[D_c](l_k),
\end{equation}
where $\beta_\Phi$ is the character of the complex Green's function $\Phi$ and $D_c=\{(c_j,-1):j\in\bbN\}$. 
The rotation coordinate is obtained by \begin{equation}\label{eq:defRotationFlow}\tau=-\bar{\tau}_*^2\frac{\mathscr{K}_{D_*}(i)}{\mathscr{K}_{D_*}(-i)},\end{equation} where $D_*=\{(\mu_j,-\epsilon_j):j\in\bbN\}$ and $\tau^*$ is a unimodular constant\footnote{The explicit formula for $\tau_*$ is given by the following:
$$\tau_*=\exp\left(-i\arg\left[\frac{\Theta'(i)}{\mathcal{W}(i)}\frac{i}{\Phi'(i)}\right]\right)$$} that is independent of $D$.
\end{theorem}

Let us show that indeed $\tau$ stands for the jump of argument of $\mathscr{K}_{D_*}$ from the lower sheet to upper sheet. Recall that $\Phi(z)=\Phi(z,i)$ and that $\Phi_\sharp(-i)>0$, $\mathcal{W}$ is the Widom function $\mathcal{W}(z)=\prod_j\Phi_{c_j}(z),$ where $\{c_j\}$ are critical points of the Green's function, and $\widetilde{\alpha}=\beta_{\Phi}+\beta_{\mathcal{W}}-\alpha$, where $\beta_\Phi,\beta_{\mathcal{W}}$ are the characters of $\Phi,\mathcal{W}$ respectively. Let $\Theta$ be the complex Martin's function with $\Im\Theta=M$, then $\mathcal{W}$ is the inner part of $\Theta'$, see \cite{Pommerenke1976}.
\begin{lemma}[\cite{BLY2} (Corollary~2.13)]\label{lem:lemma1}
Let $\tau_*$ be the same as in \eqref{eq:defRotationFlow}, then $K^{\widetilde{\alpha}},K_\sharp^{\widetilde{\alpha}}$ admit pseudocontinuations (see \eqref{eq:pseudoCont}) given by 
$$(K^{\widetilde{\alpha}})_*=\tau_*\frac{K^\alpha}{\Phi\mathcal{W}},~(K^{\widetilde{\alpha}}_\sharp)_*=\overline{\tau_*}\frac{K^\alpha_\sharp}{\Phi_\sharp\mathcal{W}}.$$
\end{lemma}

We also need the following

\begin{lemma}\label{lem:lemma2}
Let $\mathscr{K}_*$ be the pseudocontinuation of $\mathscr{K}$, with $\mathscr{K}$ given by the inner-outer factorization \eqref{eq.prodRepRepKer}, then the following is true
$$\Phi\Phi_\sharp \mathcal{W}\mathscr{K}_*=C_1\mathscr{K}_{D_*},$$
where $D_*=\{(\mu_j,-\epsilon_j):j\in\bbN\}$.
\end{lemma}
\begin{proof}
On $\partial\Omega$, $\mathscr{K}$ and its conjugate are related by the identity \cite[(3.15)]{BLY2},
$$\overline{\mathscr{K}}=\frac{\overline{C}}{C}\mathscr{K}\frac{1}{\Phi\Phi_\sharp\mathcal{W}}\prod_j\Phi_{x_j}^{-\epsilon_j}\quad a.e. \text{ on } \partial\Omega,$$
then the result follows directly by $C_1=\frac{\overline{C}}{C}$ and the product representation \eqref{eq.prodRepRepKer} with $D$ substituted by $D_*$.
\end{proof}

It follows from Lemma~\ref{lem:lemma1} and the decomposition \eqref{eq:ResolventDecomp} that 
\begin{equation}\label{eq:relationEq0}
\mathscr{K}=K^\alpha-\tau K^\alpha_\sharp
\end{equation}
and 
\begin{equation}\label{eq:relationEq}
\Phi\Phi_\sharp\mathcal{W}\mathscr{K}_*=\tau_*\Phi_\sharp K^{\widetilde{\alpha}}-\overline{\tau\tau_*}\Phi K^{\widetilde{\alpha}}_\sharp.
\end{equation}
By Lemma~\ref{lem:lemma2} and \eqref{eq:relationEq} for $z,\overline{z}$, we obtain
$$C_1\mathscr{K}_{D_*}(z)=\tau_*\Phi_\sharp(z)K^{\widetilde{\alpha}}(z)-\overline{\tau\tau_*}\Phi(z)K^{\widetilde{\alpha}}_\sharp(z),$$
$$\tau C_1\mathscr{K}_{D_*}(\overline{z})=\tau\tau_*\overline{\Phi(z)}K^{\widetilde{\alpha}}(\overline{z})-\overline{\tau_*}\Phi(\overline{z})\overline{K^{\widetilde{\alpha}}(z)},$$
where we use $(\cdots)_\sharp(z)=\overline{(\cdots)(\bar{z})}$ in the second equality. Note that $$\overline{C_1}(\mathscr{K}_{D_*})_\sharp(z)=-\tau C_1\mathscr{K}_{D_*}(z),$$ therefore, we have obtained an alternative definition of the rotation coordinate as follows:
\begin{lemma}\label{lem:newDefRotCor}
The rotation coordinate defined in \eqref{eq:defRotationFlow} equals
\begin{equation}\label{eq:newRepRot}\tau=-\frac{\overline{C_1}}{C_1}\frac{(\mathscr{K}_{D_*})_\sharp(z)}{\mathscr{K}_{D_*}(z)}.\end{equation}
Alternatively, an equivalent expression can also be easily obtained from \eqref{eq:relationEq0} 
\begin{equation}\label{eq:newRepRot2}
\tau=-\frac{\mathscr{K}_D(z)}{(\mathscr{K}_D)_\sharp(z)}.
\end{equation}
In particular, $\tau$ is independent of the choice of $z$.
\end{lemma}

Clearly, $\tau$ gives the ``character" of the reproducing kernel $\mathscr{K}_D$ in response to the anti-involution between sheets and is independent of the choice of $z$.

\begin{remark}
The representation \eqref{eq:defRotationFlow} is a consequence of the choice of normalization
$$\Phi(i)=\Phi_\sharp(-i)=0,\Phi_\sharp(i)>0,K^{\widetilde{\alpha}}(i)>0,K^{\widetilde{\alpha}}_\sharp(-i)>0$$
and the symmetry of $|\mathscr{K}_{D_*}|$. This rotation coordinate is new and is the novel part in the theory of NLS compared to existing literature of KdV.
\end{remark}

We will use the identification \eqref{eq:newRepRot2} to show (in Section~\ref{linearabel}) that the flow $\varphi(x_0,t_0)\mapsto \varphi(x_0+x,t_0+t)$ is linearized by the map $\mathcal{A}:\mathcal{D}(\sfE)\to\pi_1(\Omega)^*\times\bbT.$ To this end, we introduce a new definition for the extra component $\tau$ through Abelian integrals that is consistent to the definition of Abel maps on compact Riemann surface, e.g. \cite{GH03,BBEIM1994}. We will connect this new component  to the rotation coordinate $\tau$ defined in Lemma~\ref{lem:newDefRotCor} in Section~\ref{linearabel}. We will take one step further and show that the most fascinating linearization property of the map $\mathcal{A}$, including both components, holds in our situation (see Theorem~\ref{thm:approxFreqCharFlow}, Theorem~\ref{thm:linearzRotFlow}). Since our map $\mathcal{A}:\mathcal{D}(\sfE)\to\pi_1(\Omega)^*\times\bbT$ plays the same role as the Abel map in the case of compact Riemann surface and the generalized Abel map in the case of \cite{SY97}, we will slightly abuse the terminology and call $\mathcal{A}$ the {\it generalized Abel map}. The first component $\mathcal{A}_c$ is already well understood, the major challenge is to understand the new component \eqref{eq:newRepRot} or equivalently \eqref{eq:newRepRot2}.

Let $\sfE\subseteq \bbR$ satisfy the finite gap length condition \eqref{eq:finiteGapCond} and the associated Denjoy domain $\Omega$ is of Parreau-Widom type (see Definition~\ref{def:WidomDomain}).
\begin{definition}\label{def:rotFlowAbelMap}
For any $D=\{(\mu_j,\epsilon_j)\}\in \mathcal{D}(\sfE)$, define the formal sum
\begin{equation}\label{eq:rotFlowAbelMap}\mathcal{A}_r(D)=\sum_j-\frac{\epsilon_j}{2}\int_{a_j}^{\mu_j}\omega(\rmd\xi,\sfE_\infty) ~\mod\bbZ,\end{equation}
where $\omega(\lambda,F)$ is the harmonic measure of $\Omega,F\subseteq \sfE$ with respect to an interior point $\lambda\in\Omega$ and $\sfE_\infty=[\xi_*,+\infty)\cap \sfE$.
\end{definition}
This definition is indeed motivated by the following result, which is the main result of this section.
\begin{theorem}\label{thm:identifyRotCord}
    For $\sfE\subset\bbR$ satisfying the finite gap condition \eqref{eq:finiteGapCond}, $\Omega=\bbC\setminus\sfE$ a regular PWS, the following holds
    $$\tau(D)=C \exp(2\pi i \mathcal{A}_r(D)),~ |C|=1.$$
\end{theorem}
\begin{proof}
    Let us first show that Definition \ref{def:rotFlowAbelMap} is well defined under the assumptions. We are going to show the series \eqref{eq:rotFlowAbelMap} is majorized by an absolutely convergent series. 
    Let us first state what will go wrong with the idea of the Sodin-Yuditskii \cite{SY97}, as shown in the proof of Theorem \ref{thm:approxFreqCharFlow}. Let $\Pi_\infty^-$ be a path from $-\infty$ to $+\infty$ through the gap where $\xi_*$  locates that is invariant under the anti-involution. The orientation is decided such that $\sfE_\infty$ lies on the left side of $\Pi_\infty^-$. 
    The path $\Pi_\infty^+$ is constructed in the similar way with orientation reversed. We require that $\xi_*\notin \Pi_\infty^\pm$, as indicated in Figure 2.  Since $\infty\in\partial \Omega$, the choice of absolutely convergent series in the proof of Theorem~\ref{thm:approxFreqCharFlow} essentially fails. 
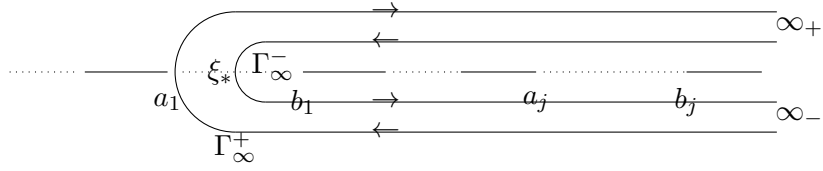
\begin{figure}
\begin{tikzpicture}
\draw[dotted] (-3,0) -- (-2.1,0);
\draw[dotted] (2.1,0) -- (3,0);
\draw (3,0)--(4,0);
\draw[dotted] (4.1,0) -- (6,0);
\draw (6,0)--(7,0);
\draw (-2,0) -- (-0.9,0);
\draw[dotted] (-0.7,0) -- (0.7,0);
\draw (0.9,0)--(2,0);
\draw (7.2,-0.4) -- (0.4,-0.4);
\draw (0.4,0.4) -- (7.2,0.4);
\node at (6,-0.4) {$b_j$};
\node at (4,-0.4) {$a_j$};
\node at (-0.9,-0.4) {$a_1$};
\node at (0.9,-0.4) {$b_1$};
\node at (-0.2,0) {$\xi_*$};
\node at (7.5,0.6) {$\infty_+$};
\node at (7.5,-0.6) {$\infty_-$};
\draw (0.4, -0.4) arc[start angle=270, end angle=90, radius=0.4];

\draw (7.2,-0.8) -- (-0.0,-0.8);
\draw (-0.0,0.8) -- (7.2,0.8);
\draw (-0.0, -0.8) arc[start angle=270, end angle=90, radius=0.8];
\node at (2,-0.4) {$\rightarrow$};
\node at (2,0.4) {$\leftarrow$};
\node at (2,-0.8) {$\leftarrow$};
\node at (2,0.8) {$\rightarrow$};
\node at (0,-1) {$\Gamma_\infty^+$};
\node at (0.5,0.1) {$\Gamma_\infty^-$};
\end{tikzpicture}
\caption{open loop}
\end{figure}
To remedy this,  recall  for any $\zeta\in \Omega$, we have the standard  relation
$$G(z,\zeta)=\int_{\partial\Omega}\log|\xi-\zeta|\omega(\rmd\xi,z)-\log|z-\zeta|.$$
Since Widom condition is independent of the choice of $\zeta\in\Omega$ \cite[Theorem 2]{Widom1971Acta}, we pick $\zeta=i\pi$ for example, let $z=\mu_j,a_j$ respectively and then take the difference, notice that $G(a_j,\zeta)=0$ since $a_j\in\partial\Omega$ and $\Omega$ is regular, we have 
$$G(\mu_j,\zeta)-G(a_j,\zeta)=\int_{\partial\Omega}\log\left|\xi-\zeta\right|(\omega(\rmd\xi,\mu_j)-\omega(\rmd\xi,a_j))-\log\left|\frac{\mu_j-a_j}{\zeta-a_j}\right|.$$
By symmetry, this is equivalent to 
$$G(\mu_j,\zeta)=\int_{\partial\Omega}\log\left|\xi-\zeta\right|(\omega(\mu_j,\rmd\xi)-\omega(a_j,\rmd\xi))-\log\left|1+\frac{\mu_j-a_j}{\zeta-a_j}\right|.$$
It follows from $\log|\zeta-\xi|\geq 1$ that 
$$G(c_j,\zeta)+\gamma_j\geq\left|\int_{a_j}^{\mu_j}\omega(\rmd t,\sfE_\infty)\right|.$$
Therefore, the generalized Abel map \eqref{eq:newAbelMap} 
$$D\mapsto\sum_j-\frac{\epsilon_j}{2}\int_{a_j}^{\mu_j}\omega(\rmd\xi,\sfE_\infty)$$
is majorized if the Widom condition and finite gap length condition hold.

    Next, let us show the equality $\tau(D)=C \exp(2\pi i\mathcal{A}_r(D))$. Note that the definition \eqref{eq:newRepRot2} is independent of the choice $z$, in particular, it is invariant when we send $z$ to the Martin boundary $\infty_+$ (resp. $\overline{z}\to\infty_-$). Note that the complex Green's function $\Phi_{\mu_j}(z)=\exp(-G(\mu_j,z)-i\widetilde{G}(\mu_j,z))$ and $$G(\mu_j,z)=\int_{\partial\Omega}\log|z-\xi|\omega(\rmd\xi,\mu_j)-\log|z-\mu_j|,$$
    since $\tau$ equals the total variation of $\mathscr{K}_D$ with respect to the involution $\overline{(\cdots)(\overline{z})}$, and that $G(\mu_j,z)$ is single-valued for each $j$, direct computation and sending $z\to\infty_+$ (resp. $\overline{z}\to\infty_-$) give the representation \eqref{eq:rotFlowAbelMap}.
\end{proof}

\begin{remark}
The construction of majorants should apply to $\mathcal{A}_c$ as well but requires more restrictive conditions. It seems that the generalized Abel map in this case needs both Widom and finite gap length conditions. We will see in later sections that the analysis of $\mathcal{A}_r$ requires more restrictive conditions on $\sfE$, especially compared to the KdV case.
\end{remark}

We will use the notation $\mathcal{A}=(\mathcal{A}_c,\mathcal{A}_r)$, where $\mathcal{A}_c$ denotes the first component of the Abel map $\mathcal{A}$ belonging to the character group $\pi_1(\Omega)^*$ and $\mathcal{A}_r$ is the second component introduced above belonging to the rotation coordinate. Note that the choice of $\xi_*$ in the definition is irrelevant, different choice clearly results in a translation in $\pi_1(\Omega)^*$. This definition will be studied in Section~\ref{linearabel}. Currently, it should be understood as a way to ``count" the total variation of $\mathscr{K}_D$ along a symmetric path from the infinity on the lower sheet to the infinity on the upper sheet.

\section{Rotation Flow and Dubrovin Vector Fields}
\label{dubrovinFlow}

In this section we consider the Dirac operator with a reflectionless time-dependent potential $\varphi(\cdot,t) \in \mathcal{R}(\sfE)$ and the behavior of its Dirichlet data; in particular, we clarify this behavior at the gap edges. The non-pausing of the translation flow will be explained by combining Lemma~\ref{lemma:GapEdgeSubordinate}. Explicit trace formulas and the corresponding Dubrovin system of differential equations are rigorously derived as well.

\subsection{Rotation Flow and Non-Pausing at Gap Edges}The argument is different than in the Schr\"odinger setting; in the KdV hierarchy, Dirichlet eigenvalues are monotone and have a nonpausing property at gap edges under the translation flow. This is not true in the Dirac/NLS setting, so it requires the introduction of an additional rotation flow
\[
\varphi(\beta,x) = e^{i\beta} \varphi(x), \qquad \beta \in \bbR.
\]
Note that $U$ is an eigenfunction of $\Lambda_\varphi$ if and only if $(\begin{smallmatrix} e^{i\beta} & 0 \\ 0 & 1 \end{smallmatrix}) U$ is an eigenfunction of $\Lambda_{e^{i\beta} \varphi}$ at the same energy $z\in \bbC$. This gives an immediate correspondence between subordinate solutions of the two operators, and between Weyl solutions of the two operators. In particular, the Schur functions depend on $\beta$ as
\[
s_\pm(\beta,z) = e^{\pm i\beta} s_\pm(z).
\]
Thus, if $\varphi\in \mathcal{R}(\sfE)$, then $e^{i\beta} \varphi \in \mathcal{R}(\sfE)$ for all $\beta \in\bbR$. As an action on Schur functions, this was considered in \cite{BLY2}, from where it follows that this flow keeps the character $\alpha$ fixed, while replacing the coordinate $\tau\in \bbT$ by $\tau e^{i\beta}$.
Similarly, $\varphi(\cdot -x) \in \mathcal{R}(\sfE)$ whenever $\varphi \in \mathcal{R}(\sfE)$, so we study the behavior of the Dirichlet data $\{(\mu_j(\beta,x),\epsilon_j(\beta,x))\}$ associated to $e^{i\beta}\varphi(\cdot - x)$ as functions of $\beta$ and $x$. 

We define the phase variables by
\begin{equation}\label{eq:argumentDivisor}
\begin{split}
\mu_j(e^{i\beta},x) & = a_j+(b_j-a_j)\sin^2(y_j(\beta,x)), \\
\epsilon_j & = \sgn \sin(2 y_j(\beta,x))
\end{split}
\end{equation}
so that $y_j:\bbR^2 \to \bbR$ is continuous and obeys $y_j(0,0) \in [0,\pi)$.

Let $G_k=(a_k,b_k)$ be the $k$-th gap in the spectrum, where $k \in \bbZ^d$ is dictated by Theorem~\ref{thm.gapLabel}. 
Recall $$\gamma_k=|G_k| \text{ and } \eta_{jk}=\mathrm{dist}(G_k,G_j).$$
Note that $\eta_{j0}=\mathrm{dist}((a_0,b_0),G_j)$.
Recall \eqref{eq:Cj} $C_j=\sup_{z\in G_j}\left(\prod_{G_k<G_j}\frac{a_k-z}{b_k-z}\prod_{G_j<G_k}\frac{b_k-z}{a_k-z}\right)^{\frac{1}{2}}.$
Clearly, we have $$C_j\leq \left(\prod_{k\neq j}\frac{\eta_{jk}+\gamma_k}{\eta_{jk}}\right)^{1/2}.$$

Equip $\mathcal{D}(\sfE)$ with the topology induced by the norm $$\Vert y\Vert=\sup_j\gamma_j^{\frac{1}{2}}\Vert y_j\Vert_{\bbT}.$$  Since $s_\pm$ are jointly continuous in $(\beta,x)$, so are the Dirichlet data.

The rotation flow allows us to characterize the gap edge case in terms of subordinate solutions, and prove a nonpausing property:

\begin{lemma}\label{lemma:GapEdgeSubordinate}
At a gap edge $E \in \{a_j,b_j\}$, if
\begin{equation} \label{eqn:gapConditionSubordinate}
\sum_{\substack{k\neq j\\ \eta_{j,k} \le 1}} \frac{\gamma_k}{\eta_{j,k}} < \infty,
\end{equation}
there exists an eigensolution $U(x,E)$ of $\Lambda_\varphi$ which is subordinate in both directions. Moreover, $\mu_j(\beta,x) = E$ if and only if
\[
U(x,E) \simeq \binom{1}{e^{i\beta}}.
\]
In particular, for fixed $x$, the set $\{\beta \in \bbR \mid \mu_j(\beta,x)=E\}$ is discrete.
\end{lemma}

\begin{proof}
We include the $x$-dependence of the diagonal Green's function element and denote it by $R(\beta,x;z)$. By the product formula \eqref{eq:resolventProd} for $R(\beta,x;z)$ ,
\[
\lim_{\substack{\lambda\to E \\ \lambda \in (a_j,b_j)}} \lvert R(\beta,x;\lambda) \rvert = \begin{cases}
0 & \mu_j(\beta,x) = E \\
\infty & \mu_j(\beta,x) \neq E
\end{cases}
\]
At first, let us assume that $\beta,x$ are such that $\mu_j(\beta,x) = E$. Using $R = - 1/(m_- + m_+)$, we conclude that $m_+ + m_-$ has limit $\infty$ as $\lambda \to E$ along the gap. By monotonicity of $m_\pm$, this implies that $m_+$ or $m_-$ has limit $\infty$. By Lemma~\ref{lemmaHerglotzGapEdge}, it follows that the eigensolution of $\Lambda_\varphi$ at energy $E$ with $U(x)=\binom 1{e^{i\beta}}$ is subordinate on at least one of the half-lines.

There is at most one subordinate solution $U_\pm$ on each half-line, so there are for every $x$ at most two values of $\beta \in [0,2\pi)$ such that $\mu_j(\beta,x) = E$. Choosing a value of $\beta$ such that $\mu_j(\beta,x) \in (a_j,b_j)$, we then conclude that $(m_+ + m_-)(\beta,x,\lambda) \to 0$ as $\lambda \to E$ along the gap. By monotonicity of $m_\pm$, it follows that $m_\pm$ have finite limits $\pm \cot(\alpha/2)$ for some $\alpha\in (0,2\pi)$. This implies $e^{i\beta} s_\pm$ have finite limits $\sfE^{\mp i \alpha}$, so the eigensolution of $\Lambda_\varphi$ with $U(x) = (\begin{smallmatrix}
1 \\  e^{i(\beta+\alpha)}    
\end{smallmatrix})$
is subordinate at both half-lines.
\end{proof}

\begin{lemma}\label{lem:diffEqPhase}
The function $y_j$ is monotone in $\beta$. If $\sfE$ has finite gap length and satisfies \eqref{eqn:gapConditionSubordinate}
the function $y_j$ is differentiable and its derivative satisfies
\begin{equation}\label{eqn:rotationflow}
\frac{\partial y_j}{\partial\beta} = -\frac 12 W_j(\mu)
\end{equation}
where
\[
W_j=W_j(\mu)=\prod_{\ell\neq j}\sqrt{\frac{(a_\ell-\mu_j)(b_\ell-\mu_j)}{(\mu_\ell-\mu_j)^2}}
\]
with the branch of square root taken so that $W_j > 0$.
\end{lemma}

\begin{proof}
Let us first consider points with $\mu_j(\beta) \in (a_j,b_j)$. By implicit differentiation, the definition $R(\beta,x,\mu_j(\beta,x))=0$ implies
\[
\frac{d\mu_j}{d\beta} = - \frac{\partial R / \partial \beta}{\partial R / \partial z} (\beta,x,\mu_j(\beta,x)).
\]
A direct calculation gives 
\[
\frac{\partial R}{\partial \beta}(e^{i\beta},\mu_j) = \frac{e^{i\beta}s_+ - e^{-i\beta} s_-}{1-s_+ s_-} = \epsilon_j.
\]
Differentiating the product representation \eqref{eq:resolventProd} and evaluating at $\mu_j$ gives
\[
\frac{\partial R}{\partial z}(e^{i\beta},\mu_j) =\frac 1{\sqrt{(\mu_j - a_j)(b_j - \mu_j)}}  \prod_{k\neq j} \sqrt{ \frac{(\mu_j - \mu_k)^2}{(\mu_j -a_k)(\mu_j - b_k)} },
\]
where we are taking positive square roots of positive numbers. Combining this with
\begin{align*}
\frac{d\mu_j}{dy_j} & = (b_j - a_j) \epsilon_j \sqrt{1- \cos^2(2y_j)} \\
& = 2 \epsilon_j \sqrt{(b_j-\mu_j)(\mu_j - a_j)}
\end{align*}
proves \eqref{eqn:rotationflow}.

This, together with continuity, shows monotonicity of $y_j$. 
Since $W_j$ is a continuous function of $\mu$, the right-hand side of \eqref{eqn:rotationflow} is globally a $C^1$ function. Since we already know differentiability and \eqref{eqn:rotationflow} away from a discrete set, by continuity, $y_j$ is differentiable everywhere and \eqref{eqn:rotationflow} holds.
\end{proof}


\subsection{Trace Formula and Dubrovin Vector Fields}
Suppose that the potential $\varphi(x,t)$ depends differentiably on a time variable $t$, and $\Lambda_t=\Lambda_{\varphi(\cdot,t)}$ is the corresponding family of Dirac operators. 
Then $\varphi(x,t)$ is a solution of the NLS \begin{equation}\label{eq:NLS2}
i\varphi_t=-\partial^2_{x}\varphi+2\varphi^2\overline{\varphi},~\varphi(x,0)=\varphi(x)
\end{equation}
if and only if $$\partial_t \Lambda_t=[B,\Lambda_t],$$
where $[B,L]=BL-LB$ is the commutator of $L$ with the operator 
$$B=\begin{bmatrix}2i\partial_x^2-i\varphi_1\varphi_2&\varphi_1'+2\varphi_1\partial_x\\\varphi_2'+2\varphi_2\partial_x&-2i\partial_x^2+i\varphi_1\varphi_2\end{bmatrix}, \varphi_1=\overline{\varphi_2}=\varphi.$$
It can be shown that $B$ is skew-adjoint, that is, $B^*=-B$. 
Let $V(t)$ be the flow 
$$
\partial_t V=BV, \quad V(0)=I,
$$
then $V(t)$ is a family of unitary operators.
Direct computations show that  $$\partial_t(V^*(t)\Lambda_tV(t))=V^*(\partial_t\Lambda-[B,\Lambda_t])V=0,$$
which implies that $\Lambda_t$ is an isospectral family. This is the classical Lax-pair formalism \cite{Lax1968,ZS1972}. A scheme of introducing the associated AKNS hierarchy \cite{AKNS} is put in Appendix~\ref{sec:AKNS}, interested readers may refer to \cite{GH03} for more details.

Define the scalar fields on $\mathcal{D}(\sfE)$:
\begin{equation}\label{eq:scalarField1}
Q_1(y)=\sum_{k}(a_k+b_k-2\mu_k).
\end{equation}
\begin{equation}\label{eq:scalarField2}
Q_2(y)=\sum_k(a_k^2+b_k^2-2\mu_k^2).
\end{equation}
Our next goal is to derive the Dubrovin vector fields that describe the dependence of the Dirichlet data $y(x,t)$ on $x,t$ respectively. We will first deal with the case $\mu_j\in (a_j,b_j)$ and then extend the result to the gap edges $\mu_j\in \{a_j,b_j\}$.

Let $\varphi\in\mathcal{R}(\sfE)$ be continuous and  $s_\pm(z)$ be the reflectionless pair of Schur functions corresponding to the potential $\varphi$ and $\{(\mu_j,\epsilon_j)\}$ be the corresponding Dirichlet data. 
The following result is known as the {\it trace formula}. Similar result for one dimensional Dirac operators was claimed by \cite{Zamonov1985} and used in \cite{ME97}.
Since we could not find an available proof, to be consistent with our context, we decide to give a proof.
\begin{lemma}\label{lem:traceFormula}
Assume that $\sum_j\gamma_j<\infty,$ for the divisor $D=\{(\mu_j,\epsilon_j)\}$ and the associated phase variable $y$ in \eqref{eq:argumentDivisor}, 
\begin{equation}\label{eq:traceEqRe}\Re\varphi=-\frac{1}{2}Q_1(y),\end{equation}
\begin{equation}\label{eq:traceEqIm}\partial_x \Im\varphi+(\Im\varphi)^2=\frac{1}{2}Q_2(y).
\end{equation}
\end{lemma}
\begin{proof}
Recalling the product representation \eqref{eq:resolventProd}, since $$\begin{aligned}R(z)&=i\mathrm{exp}\left(\frac{1}{2}\sum_j(2\log(1-\mu_jz^{-1})-\log(1-a_jz^{-1})-\log(1-b_jz^{-1}))\right)\\
&=i+\frac{i}{2}Q_1(y)z^{-1}+\frac{i}{4}(Q_2(y)+\frac{1}{2}Q_1^2(y))z^{-2}+o(|z|^{-2})\end{aligned}$$
as $|z|\to\infty, \arg z\in(\epsilon,\pi-\epsilon)$ for some $0<\epsilon<\pi/2,$
a comparison of coefficients of $z^{-1}$ term in the following asymptotic expansion of \cite[Theorem 4.10]{ClarkGesztesy}
$$M_{1,1}=\frac{i}{2}+[-\frac{i}{8}(B_{1,1}(x+0)+B_{1,1}(x-0))-(B_{2,2}(x+0)+B_{2,2}(x-0))]z^{-1}+o(|z|^{-1}),$$
with $~B_{1,1}=-B_{2,2}=\Re\varphi$
immediately gives the trace formula for $\Re\varphi.$ Note that their definition of the diagonal Green's function $M_{1,1}$ differs from $R(z)$ by a constant factor $2, i.e.  R(z)=2M_{1,1}$.

To obtain the trace formula for $\Im\varphi$, we need the higher order asymptotic expansion of the diagonal Green's function. 
Recall \cite[(2.76), (2.77),(4.62), (4.63)]{ClarkGesztesy} reduced to $m=1$ and $B_{1,1}=-B_{22}=\Re\varphi,B_{1,2}=B_{2,1}=\Im\varphi$:
$$\frac{1}{2}R(z)=M_{1,1}=\frac{1}{M_--M_+},$$
$$M_\pm=\pm i+\sum_{k=1}^{N}m_{\pm,k}z^{-k}+o(|z|^{-N}),$$
with $$\begin{aligned}
&m_{+,1}=-\Im\varphi+i\Re\varphi=i\varphi,\\
&m_{-,1}=-\Im\varphi-i\Re\varphi=-i\overline{\varphi},\\
&m_{+,2}=\frac{i}{2}(\partial_xm_{+,1}+m_{+,1}^2-2im_{+,1}\varphi),\\
&m_{-,2}=-\frac{i}{2}(\partial_xm_{-,1}+m_{-,1}^2+2im_{-,1}\overline{\varphi} ).
\end{aligned}
$$
As a consequence of these equations, we have the following identities by comparing the coefficient of $z^{-2}$ in the asymptotics of $R(z)$:
$$\frac{i}{4}(Q_2+\frac{1}{2}Q_1^2)=\frac{i}{4}(-(m_{-,1}-m_{+,1})^2-2i(m_{-,2}-m_{+,2})),$$
that is,
$$\frac{\partial\Im\varphi}{\partial x}+(\Im\varphi)^2=\frac{1}{2}Q_2.$$
\end{proof}

We derive the time dependence of the Schur functions by a treatment of \cite{LY20}:
\begin{lemma}\label{lem:RiccatiTime}
Denoting $p=\overline{q}=i\varphi$
the Schur functions $s_\pm$ satisfy the following Riccati-type equations
$$\begin{aligned}-i\partial_t s_+&=b-2as_+-cs_+^2,\\
-i\partial_t s_-&=c+2as_--bs_-^2,\end{aligned}$$
where $a=(\frac{1}{2}pq+z^2),b=(\frac{1}{2}\partial_xq-iqz), c=-(\frac{1}{2}\partial_xp+ipz).$
\end{lemma}
\begin{proof}
Let $\Psi_{\pm}(x,x_0,t,z)$ be the Weyl solutions of $\Lambda_{\varphi(\cdot,t)} \Psi=z\Psi$ that are square integrable on $\bbR_{\pm}^{x_0}$,
where $\bbR^{x_0}_+=[x_0,\infty)$ and $\bbR^{x_0}_-=(-\infty,x_0]$. The Schur functions are given as follows \cite{EFGL}
\begin{equation}\label{eq:schurFunc}
s_+(x,t,z)=\frac{\Psi_{+}^1}{\Psi_{+}^2},~s_-(x,t,z)=\frac{\Psi_-^2}{\Psi_-^1}.
\end{equation}
Our goal is to compute $\partial_t s_\pm$. Consider the following initial value problem and introduce some notations
\begin{equation}\label{eq:AKNS1}
\left\{\begin{aligned}
&\partial_x\Phi_{\pm}=U(z)\Phi_\pm\\
&\partial_t\Phi_{\pm}=\widetilde{V}_{2}(x,t)\Phi_{\pm}\\
&\Phi_{\pm}(x,x_0,0,z)=\Psi_{\pm}(x,x_0,0,z)
\end{aligned}
\right.
\end{equation}
where $U(z)=\begin{bmatrix}-iz&p\\q&iz\end{bmatrix}$ with $p=\overline{q}=i\varphi$ and by the computations in Appendix~\ref{sec:AKNS}
$$\widetilde{V}_2=i\begin{bmatrix}-\widetilde{G}_2(z)&\widetilde{F}_1(z)\\-\widetilde{H}_1(z)&\widetilde{G}_2(z)\end{bmatrix}=i\begin{bmatrix}-\frac{1}{2}pq-z^2&\frac{1}{2}\partial_x q-iqz\\-\frac{1}{2}\partial_x p-ipz&\frac{1}{2}pq+z^2\end{bmatrix}.$$
Note that the first equation in \eqref{eq:AKNS1} represents the Dirac equation $\Lambda_\varphi\Phi_{\pm}=z\Phi_\pm$. By \eqref{eq:zeroCur}, the defocusing NLS is equivalent to the case $p=\overline{q}$ and $n=1,$ namely the zero-curvature equation
$$U_t-\widetilde{V}_{2,x}+[U,\widetilde{V}_2]=0.$$
Solutions to \eqref{eq:AKNS1} are thoroughly discussed in \cite{BBEIM1994, GH03} as Baker-Akhiezer functions for the finite gap case. More general cases can be treated similarly as in \cite[Remark 2.1]{LY20}.
Let $\Phi_\pm(x,t)$ be  solutions to \eqref{eq:AKNS1}, define $g(x,t)=\Vert\Phi_{\pm}(x,x_0,t,z)\Vert_{\bbC^2}^2=|\Phi_{\pm}^1|^2+|\Phi_{\pm}^2|^2.$
Then $$\partial_t g(x,t)=2\Re\langle\partial_t\Phi_{\pm},\Phi_{\pm}\rangle=2\Re\langle\widetilde{V}_{n+1}\Phi_{\pm},\Phi_{\pm}\rangle\leq 2\Vert\widetilde{V}_{n+1}\Vert g(x,t).$$
Assume that $\varphi,\partial_x\varphi\in L^{\infty}(\bbR)$, the boundedness of $\varphi$  then implies 
$$C:=\sup_{t\in[0,T],x\in\bbR}\Vert \widetilde{V}_{2}\Vert<\infty.$$
By the Gronwall's inequality, we obtain
$$g(x,0)e^{-2Ct}\leq g(x,t)\leq g(x,0)e^{2Ct},$$
for $t\in[0,T]$.
Therefore, 
$$e^{-2Ct}\Vert\Phi_{\pm}(\cdot,x_0,0,z)\Vert^2\leq\int^{\pm\infty}_{x_0}\Vert\Phi_{\pm}(x,x_0,t,z)\Vert^2 \, \rmd x
\leq e^{2Ct}\Vert\Phi_{\pm}(x_0,x_0,0,z)\Vert^2.$$
Since $\Phi_{\pm}=\Psi_{\pm}$ when $t=0$, it follows that $\Phi_{\pm}(\cdot,x_0,t,z)$ are Weyl solutions for $t\in[0,T]$. 
According to the second differential equation of \eqref{eq:AKNS1}, and that since we are in the limit point case, 
$$-i\partial_t s_+=(\frac{1}{2}\partial_x q-iqz)-2(\frac{1}{2}pq+z^2)s_+-(-(\frac{1}{2}\partial_x p+ipz))s_+^2;$$
$$-i\partial_t s_-=-(\frac{1}{2}\partial_x p+ipz)+2(\frac{1}{2}pq+z^2)s_--(\frac{1}{2}\partial_xq-iqz)s_-^2.$$
\end{proof}

We also need the Riccati equation in space variable:
\begin{lemma}\label{lem:xRiccati}
The Schur functions $s_{\pm}$ satisfy the following Riccati type equations in space variable:
$$-i\partial_x s_{+}=\overline{\varphi(x)}s_{+}^2-2zs_++\varphi.$$
$$-i\partial_x s_-=-(\varphi s_-^2-2zs_-+\overline{\varphi}).$$
\end{lemma}
\begin{proof}
Let $\psi_\pm$ be the Weyl solutions of $\Lambda_\psi=z\psi$, rewritten as
$$\begin{bmatrix}\partial_x\psi_\pm^1\\\partial_x\psi_\pm^2\end{bmatrix}=\begin{bmatrix}-iz&i\varphi\\-i\overline{\varphi}&iz\end{bmatrix}
\begin{bmatrix}\psi_\pm^1\\\psi_\pm^2\end{bmatrix}.$$
The lemma follows directly from the definition  $s_+=\frac{\psi_+^1}{\psi_+^2},s_-=\frac{\psi_-^2}{\psi_-^1}$.
\end{proof}

In order to properly define the vector fields, we assume that the finite gap length condition \eqref{eq:finiteGapCond} holds. This will be eventually superseded by stronger Craig-type conditions that we will introduce. 

The following lemma deal with the case $\mu_j\in (a_j,b_j)$ as we promised.
\begin{lemma}\label{lem:DubrovinEq}
Assume \eqref{eq:finiteGapCond} holds and $\varphi,\partial_x\varphi\in L^\infty(\bbR)$, then for $\mu_j(x,t)\in (a_j,b_j)$, the flow $y_j(x,t)$ defined in \eqref{eq:argumentDivisor} satisfies the {\it Dubrovin type} equations
\begin{equation}\label{eq:spaceDiff}
\frac{\partial y_j}{\partial x}= \Psi_j(y)
\end{equation}
where $\Psi_j:\mathcal{D}(\sfE) \to \bbR$ is given by
\[
\Psi_j = \left\{Q_1(y)/2+\mu_j\right\}W_j
\]
 and
\begin{equation}\label{eq:timeDiff}
\frac{\partial y_j(x,t)}{\partial t}=\left(\frac{Q_2}{4}+\frac{1}{8}Q_1^2-\frac{1}{2}\mu_jQ_1+\mu_j^2\right)W_j
\end{equation}
where $W_j=W_j(\mu)=\prod_{k\neq j}\sqrt{\frac{(a_k-\mu_j)(b_k-\mu_j)}{(\mu_k-\mu_j)^2}}$.
\end{lemma}
\begin{proof}
According to Lemma~\ref{lem:xRiccati} and that $R(\mu_j)=0,s_{\epsilon_j}=1$, 
$$\partial_x R(z)|_{z=\mu_j(x)}=\epsilon_j(2\Re\varphi-2\mu_j).$$
Recall the normalized product representation $$R(z)=i\sqrt{\prod_\ell\frac{(\mu_\ell-z)^2}{(a_\ell-z)(b_\ell-z)}},$$
thus 
\begin{equation}\label{eq:diffInZ}\partial_z R(z)|_{z=\mu_j}=\frac{i}{\sqrt{(a_j-\mu_j)(b_j-\mu_j)}}\sqrt{\prod_{\ell\neq j}\frac{(\mu_\ell-\mu_j)^2}{(a_\ell-\mu_j)(b_\ell-\mu_j)}}.\end{equation}
For $\mu_j\in(a_j,b_j)$, by the implicit function theorem,
$$\frac{\partial \mu_j}{\partial x}=-\frac{\partial_x R|_{z=\mu_j}}{\partial_z R|_{z=\mu_j}}=2i\epsilon_j(\Re\varphi-\mu_j)\sqrt{(a_j-\mu_j)(b_j-\mu_j)}\sqrt{\prod_{\ell\neq j}\frac{(a_\ell-\mu_j)(b_\ell-\mu_j)}{(\mu_\ell-\mu_j)^2}}.$$
Note that $$\partial_x\mu_j=(b_j-a_j)\sin(2y_j)\partial_x y_j=2\epsilon_j\sqrt{(b_j-\mu_j)(\mu_j-a_j)}\partial_x y_j,$$
it follows that 
$$\partial_x y_j=-(\Re\varphi-\mu_j)\sqrt{\prod_{\ell\neq j}\frac{(a_\ell-\mu_j)(b_\ell-\mu_j)}{(\mu_\ell-\mu_j)^2}}=(-\Re\varphi+\mu_j)W_j.$$
This proves \eqref{eq:spaceDiff}.

Recall that $R=i\frac{(1-s_+)(1-s_-)}{1-s_+s_-}$ and that $s_{\epsilon_j}(\mu_j)=1.$ As a consequence of Lemma~\ref{lem:RiccatiTime}, we obtain $$\partial_t R(z)|_{z=\mu_j}=-i\partial_ts_{\epsilon_j}=\epsilon_j(b-c)-2\epsilon_j a,$$
where $a,b,c$ are given by Lemma~\ref{lem:RiccatiTime}. Note that $p=i\varphi,q=-i\overline{\varphi}$,
$$b-c=-\partial_x(\Im\varphi)-2z\Re\varphi,~2a=\Re\varphi^2+\Im\varphi^2+2z^2.$$
Let $z=\mu_j$ and apply the trace formulas in Lemma~\ref{lem:traceFormula}, we have the following
\begin{equation}\label{eq:timeDerivativeResolvent}
\begin{aligned}
\partial_t R(z)|_{z=\mu_j}&=\epsilon_j\{-\partial_x(\Im\varphi)-(\Im\varphi)^2-(\Re\varphi)^2-2z\Re\varphi-2z^2\}\\
&=\epsilon_j\{-\frac{1}{2}Q_2-\frac{1}{4}Q_1^2+\mu_jQ_1-2\mu_j^2\}.
\end{aligned}
\end{equation}

For $\mu_j\in(a_j,b_j)$, similarly, 
$$\frac{\partial\mu_j}{\partial t}=-\frac{\partial_t R|_{z=\mu_j}}{\partial_z R|_{z=\mu_j}}=i\partial_t R(z)|_{z=\mu_j}\sqrt{(a_j-\mu_j)(b_j-\mu_j)}\sqrt{\prod_{\ell\neq j}\frac{(a_\ell-\mu_j)(b_\ell-\mu_j)}{(\mu_\ell-\mu_j)^2}}.$$
Since $$\partial_t\mu_j=(b_j-a_j)\sin(2y_j)\partial_ty_j=2\epsilon_j\sqrt{(b_j-\mu_j)(\mu_j-a_j)}\partial_ty_j,$$
we have 
$$\partial_ty_j=\epsilon_j(-\frac{1}{2}\partial_tR(z)|_{z=\mu_j})\sqrt{\prod_{\ell\neq j}\frac{(a_\ell-\mu_j)(b_\ell-\mu_j)}{(\mu_\ell-\mu_j)^2}}=\left(\frac{Q_2}{4}+\frac{1}{8}Q_1^2-\frac{1}{2}\mu_jQ_1+\mu_j^2\right)W_j.$$
This proves \eqref{eq:timeDiff}.
\end{proof}

We prove that this is also true at gap edges. At this point, we can adapt some arguments from \cite{BDGL}, with $\beta$ in place of the spatial coordinate $x$ and $x$ in place of $t$; in other words, with changes of $\tau$ in the role of a translation flow, and the Dirac translation flow in the role of the KdV flow:

\begin{lemma} \label{lem:dyj/dxPsijy}
Fix a gap $(a_j, b_j)$ of the spectrum $\sfE$ of $\varphi\in\mathcal{R}(\sfE)$. If $\sfE$ has finite gap length and satisfies \eqref{eqn:gapConditionSubordinate}, 
then $y_j$ is a $C^1$ function of $x \in \bbR$ and obeys \eqref{eq:spaceDiff} for all $x\in \bbR$.
\end{lemma}

\begin{proof}
As in \cite[Lemma 3.5]{BDGL}, we have:
\begin{enumerate}[label={\rm(\alph*)}]
\item For any $a< b$,
\[
\lvert y_j(\beta,b) - y_j(\beta,a) \rvert \le \int_a^b \lvert \Psi_j(y(\beta,t)) \rvert \, \rmd t
\]
\item For any $a< b$, if $\Psi_j(y(\beta,t)) \ge 0$ for $t\in (a,b)$, then
\[
0 \le y_j(\beta,b) - y_j(\beta,a) \le \int_a^b \Psi_j(y(\beta,t)) \, \rmd t.
\]
\end{enumerate}
Next, just as in \cite[Prop. 3.2]{BDGL}, we conclude
\[
y_j(\beta,b) - y_j(\beta,a) = \int_a^b \Psi_j(y(\beta,t)) \, \rmd t
\]
so by the fundamental theorem of calculus, the proof is completed.
\end{proof}


We need suitable Craig-type conditions \cite{Craig89} on the spectral data $\sfE=\bbR\setminus\left(\cup_{j\in\bbN}(a_j,b_j)\right)$ so that the vector field in \eqref{eq:spaceDiff} and \eqref{eq:timeDiff} are well defined for our purpose. Suppose that $\sfE$ satisfies the Craig-type conditions \eqref{eq:craigCond1}, \eqref{eq:craigCond2}, and \eqref{eq:craigCond3} for some $\delta > 0$.

Equip the tangent space of $\mathcal{D}(\sfE)$ with the norm $$\Vert y\Vert=\sup_{j}\gamma_j^{1/2}|y_j|,$$
where $y=(\cdots,y_j,\cdots)$. Let $\Psi$ be the vector field on $\mathcal{D}(\sfE)$ of the right side of \eqref{eq:spaceDiff} and $\Xi$ be the vector field given by the right side of \eqref{eq:timeDiff}. We are interested in the following initial value problem: 
\begin{equation}\label{eq:DubrovinFlow}\left\{\begin{aligned}
&\partial_x y=\Psi(y),~ \partial_t y=\Xi(y)\\
&y(0,0)=f\in\mathcal{D}(\sfE).
\end{aligned}
\right.
\end{equation}

\begin{prop}\label{prop:LipVecField}
    If $\sfE$ satisfies \eqref{eq:craigCond1}, \eqref{eq:craigCond2} and \eqref{eq:craigCond3}, then $\Psi$ and $\Xi$ are Lipschitz vector fields.
\end{prop}
\begin{proof}
Recall the trace formula $Q_1,Q_2$ defined in \eqref{eq:scalarField1} and \eqref{eq:scalarField2}. 
For the vector field $\Psi$, the $j$-th component is given by 
$$\Psi_j(y)=(Q_1/2+\mu_j)W_j.$$

Note that $$\begin{aligned}\frac{\partial W_j}{\partial y_j}=&\frac{\gamma_j\sin(2y_j)}{2}\sum_{k\neq j}\left(\prod_{k'\neq j,k}\sqrt{\frac{(a_{k'}-\mu_j)(b_{k'}-\mu_j)}{(\mu_{k'}-\mu_j)^2}}\right)\\
&\times\frac{(b_k-\mu_j)(\mu_k-a_k)+(a_k-\mu_j)(\mu_k-b_k)}{\sgn(\mu_j-\mu_k)(\mu_k-\mu_j)^2\sqrt{(a_k-\mu_j)(b_k-\mu_j)}}
\end{aligned}$$
and $$\left\vert\frac{\partial W_j}{\partial y_j}\right\vert\leq C_j^2\gamma_j\sum_k\frac{\gamma_k}{\eta_{jk}^2}.$$
For $k\neq j$,
$$\frac{\partial W_j}{\partial y_k}=\prod_{\ell\neq j}\sqrt{\frac{(a_\ell-\mu_j)(b_\ell-\mu_j)}{(\mu_\ell-\mu_j)^2}}\frac{\gamma_k\sin(2y_k)}{\mu_j-\mu_k}=W_j\frac{\gamma_k\sin(2y_k)}{\mu_j-\mu_k}$$
and 
$$\left|\frac{\partial W_j}{\partial y_k}\right|\leq \frac{C_j^2\gamma_k}{\eta_{jk}}.$$
For $k\neq j$, 
$$\frac{\partial \Psi_j}{\partial y_k}=\frac{1}{2}\gamma_kW_j\sin(2y_k)\left(-2+\frac{Q_1+2\mu_j}{|\mu_k-\mu_j|}\right).$$
For $k=j$,
$$\frac{\partial\Psi_j}{\partial y_j}=(Q_1/2+\mu_j)\partial_{y_j}W_j.$$
Consequently, $$\begin{aligned}&\left\vert\frac{\partial\Psi_j}{\partial y_k}\right\vert \leq C^2_j\gamma_k\left(1+\frac{\Vert Q_1\Vert_\infty+2\mu_j}{2\eta_{jk}}\right), k\neq j,\\
&\left\vert\frac{\partial \Psi_j}{\partial y_j}\right\vert\leq \gamma_jC_j^2(\Vert Q_1\Vert_\infty/2+\mu_j)\sum_{k\neq j}\frac{\gamma_k}{\eta^2_{jk}}.\end{aligned}$$
Then $$\begin{aligned}\Vert \Psi(y)-\Psi(\widetilde{y})\Vert&=\sup_j\gamma_j^{\frac{1}{2}}\Vert \Psi_j(y)-\Psi_j(\widetilde{y})\Vert\\
&\leq\sup_j\gamma_j^{\frac{1}{2}}\sum_{k}\left\Vert\frac{\partial\Psi_j}{\partial y_k}\right\Vert_{\infty}\Vert y_k-\widetilde{y}_k\Vert\\
&\leq \Vert y-\widetilde{y}\Vert\sup_j\sum_{k}\gamma_j^{\frac{1}{2}}\gamma_k^{-\frac{1}{2}}\left\Vert\frac{\partial\Psi_j}{\partial y_k}\right\Vert_{\infty}.
\end{aligned}$$
We split the sum on the right hand side of the last inequality into two groups $$\sum_{k}\gamma_j^{\frac{1}{2}}\gamma_k^{-\frac{1}{2}}\left\Vert\frac{\partial\Psi_j}{\partial y_k}\right\Vert_{\infty}=\sum_{k\neq j}\gamma_j^{\frac{1}{2}}\gamma_k^{-\frac{1}{2}}\left\Vert\frac{\partial\Psi_j}{\partial y_k}\right\Vert_{\infty}+\sum_{k=j}\gamma_j^{\frac{1}{2}}\gamma_k^{-\frac{1}{2}}\left\Vert\frac{\partial\Psi_j}{\partial y_k}\right\Vert_{\infty}.$$ For the first group, it is straightforward to verify that by \eqref{eq:craigCond2}
$$\sup_j\sum_{k\neq j}\gamma_j^{\frac{1}{2}}\gamma_k^{-\frac{1}{2}}\left\Vert\frac{\partial\Psi_j}{\partial y_k}\right\Vert_{\infty}\leq \sup_j\sum_{k\neq j}\gamma_j^{\frac{1}{2}}\gamma_k^{\frac{1}{2}}C_j^2 \left(1+\frac{\Vert Q_1\Vert_\infty+2(\gamma_0+\eta_{0j})}{2\eta_{jk}}\right)<\infty.$$
For the second group, by \eqref{eq:craigCond1} and \eqref{eq:craigCond2}, $$\begin{aligned}
\sup_j C_j^2(\Vert Q_1\Vert_\infty/2+\gamma_0+\eta_{0j})\sum_{k\neq j}\frac{\gamma_j\gamma_k}{\eta^2_{jk}}<\infty.
\end{aligned}$$
For the vector field $\Xi,$ according to  \eqref{eq:timeDiff} 
$$\Xi_j(y)=\left(\frac{Q_2}{4}+\frac{1}{8}Q_1^2-\frac{1}{2}\mu_jQ_1+\mu_j^2\right)W_j.$$
Then for $k\neq j$, 
$$\frac{\partial \Xi_j}{\partial y_k}=\gamma_k\sin(2y_k)(\mu_j-\mu_k-Q_1/2)W_j+\left(\frac{Q_2}{4}+\frac{1}{8}Q_1^2-\frac{1}{2}\mu_jQ_1+\mu_j^2\right)\partial_{y_k}W_j,$$
and for $k=j$,
$$\frac{\partial\Xi_j}{\partial y_j}=\gamma_j\sin(2y_j)(2\mu_j-Q_1/2)W_j+\left(\frac{Q_2}{4}+\frac{1}{8}Q_1^2-\frac{1}{2}\mu_jQ_1+\mu_j^2\right)\partial_{y_j}W_j.$$
Therefore, for $k\neq j$, we have the estimates
$$\begin{aligned}\left\vert\frac{\partial\Xi_j}{\partial y_k}\right\vert&\leq \gamma_k C_j^2\left((\eta_{jk}+\Vert Q_1/2\Vert_\infty)+(\frac{\Vert Q_2\Vert_\infty}{4}+\frac{\Vert Q_1\Vert_\infty(|\mu_j|+\Vert Q_1\Vert_\infty/4)}{2}+\mu_j^2)/\eta_{jk}\right).\end{aligned}$$
For $k=j,$ $$\left|\frac{\partial\Xi_j}{\partial y_j}\right|\leq C_j^2\gamma_j\left(
2|\mu_j|+\Vert Q_1/2\Vert_\infty
+\sum_{k}\frac{\gamma_k}{\eta_{jk}^2}(\frac{\Vert Q_2\Vert_\infty}{4}+\frac{\Vert Q_1\Vert_\infty(|\mu_j|+\Vert Q_1\Vert_\infty/4)}{2}+\mu_j^2)
\right).$$
It follows from above estimates that  $$\begin{aligned}
\Vert \Xi(y)-\Xi(\widetilde{y})\Vert&=\sup_j\gamma_j^{\frac{1}{2}}\Vert \Xi_j(y)-\Xi_j(\widetilde{y})\Vert\\
&\leq \sup_j\gamma_j^{\frac{1}{2}}\sum_k\left\Vert\frac{\partial\Xi_j}{\partial y_k}\right\Vert_\infty\Vert y_k-\widetilde{y}_k\Vert\\
&\leq\Vert y-\widetilde{y}\Vert\sup_j\sum_k\gamma_j^{\frac{1}{2}}\gamma_k^{-\frac{1}{2}}\left\Vert\frac{\partial\Xi_j}{\partial y_k}\right\Vert_\infty.
\end{aligned}$$
Similarly, we split the sum on the right hand side of  the last inequality into two groups: $$\sum_k\gamma_j^{\frac{1}{2}}\gamma_k^{-\frac{1}{2}}\left\Vert\frac{\partial\Xi_j}{\partial y_k}\right\Vert_\infty=\sum_{k\neq j}\gamma_j^{\frac{1}{2}}\gamma_k^{-\frac{1}{2}}\left\Vert\frac{\partial\Xi_j}{\partial y_k}\right\Vert_\infty+\sum_{k=j}\gamma_j^{\frac{1}{2}}\gamma_k^{-\frac{1}{2}}\left\Vert\frac{\partial\Xi_j}{\partial y_k}\right\Vert_\infty.$$ 
For the first group, combining \eqref{eq:craigCond1}, \eqref{eq:craigCond2} and \eqref{eq:craigCond3} gives 
$$\begin{aligned}&\sup_j\sum_{k\neq j}\gamma_j^{\frac{1}{2}}\gamma_k^{-\frac{1}{2}}\left\Vert\frac{\partial\Xi_j}{\partial y_k}\right\Vert_\infty\\
&\leq \sup_j\sum_{k\neq j}\gamma_j^{\frac{1}{2}}\gamma_k^{\frac{1}{2}}C_j^2\left((\eta_{jk}+\Vert Q_1/2\Vert_\infty)+\frac{1}{\eta_{jk}}(\frac{\Vert Q_2\Vert_\infty}{4}+\frac{\Vert Q_1\Vert_\infty(|\mu_j|+\Vert Q_1\Vert_\infty/4)}{2}+\mu_j^2)\right)\\&<\infty.\end{aligned}$$
For the second group, it is straightforward to verify that under the conditions \eqref{eq:craigCond1} and \eqref{eq:craigCond2}
$$
\sup_{j}C_j^2\gamma_j\left(
2|\mu_j| + \Vert Q_1/2\Vert_\infty + \sum_{k}\frac{\gamma_k}{\eta_{jk}^2}
\left(\frac{\Vert Q_2\Vert_\infty}{4} + \frac{\Vert Q_1\Vert_\infty(|\mu_j|+\Vert Q_1\Vert_\infty/4)}{2}+\mu_j^2\right)
\right)<\infty.$$
\end{proof}
We need to discuss the regularity of the function $y(x,t)$ defined in \eqref{eq:argumentDivisor}.
\begin{prop}\label{prop:regularityOfDirichletData}
Suppose $\varphi(x,t)$ obeys \eqref{eq:nls} and satisfies $\varphi,\partial_x\varphi\in L^\infty(\bbR\times [0,T])$ for some $T>0$. 
Then, the flow $y(x,t)$ satisfies the following properties:
\begin{enumerate}[label={\rm(\alph*)}]
\item $y(x,t)$ is jointly continuous in $(x,t)$;
\item $y(x,t)$ is differentiable in $t$ and obeys $\partial_t y=\Xi(y)$.
\end{enumerate}
\end{prop}
Since the proof of differentiability at a spectral edge follows verbatim the proof of \cite[Proposition 3.2]{BDGL} we focus on the proof of continuity. 
\begin{proof}
According to \eqref{eq:DubrovinFlow} and Proposition~\ref{prop:LipVecField}, it remains to clarify the case $\mu_j(x,t)\in \{a_j,b_j\}$.
Recalling \eqref{eq:resolventFunc}, we conclude that $R(x,t,z)$ is differentiable with respect to $t$. Suppose that $\mu_j(x_0,t_0)=b_j$, the other case can be dealt with similarly. Since $R(x_0,t_0,z)$ is strictly increasing in $(a_j,b_j)$, for any $\epsilon>0$, $R(x_0,t_0,\mu_j-\epsilon)<0.$ By continuity of $R$ in $t$, there exists $\delta>0$ such that $R(x_0,t,b_j-\epsilon)<0$ for every $|t-t_0|<\delta$. This implies that $\mu_j(x_0,t_0)-\epsilon\leq\mu_j(x_0,t)\leq \mu_j(x_0,t_0)$ for $|t-t_0|<\delta$. Therefore, $\mu_j(x,t)$ is continuous in $t.$

\end{proof}

\section{Linearization of the Translation and NLS Flow}
\label{linearabel}

We would like to emphasize that the linearization of the translation flow for NLS is more difficult than the linearization of KdV flow, due to the extra component of the Abel map. 
The following example illustrates the situation and highlights some of the key new features.

\begin{ex}\label{ex:ex1}
For $c>0$, the reflectionless Dirac operators with $\sfE = \bbR \setminus (-c,c)$ are constants $\varphi(x) = c e^{i\beta}$, $\beta\in\bbR$. The corresponding Schur functions are
\[
s_+(z) = e^{i\beta} \left( \frac zc - \sqrt{ \frac{z^2}{c^2} - 1} \right)
\]
with branch of square root on $\bbC \setminus \sfE$ such that $s_+(z) \to 0$ as $z\to \infty$.

The solutions of the NLS with constant initial data $\varphi \equiv e^{i\beta} c$ are
\[
\varphi(x,t) = e^{-2ic^2 t} e^{i\beta} c.
\]
This is solved by noting that $\varphi(x, t)$ remains independent of $x$ for each fixed $t$ and solving the simple ODE that results from dropping the $\partial_x^2$ term from NLS. 
Note that in this case the character group is trivial. The $t$-dependence is entirely captured by the change of $\tau$ with $t$, which we read off as $\tau(t) = e^{-2ic^2 t} e^{i\beta}$.
\end{ex}

However, we will see the extra freedom provided by the rotation coordinate allows us to correctly formulate the uniqueness and almost periodicity through the trace formula \eqref{eq:traceEqRe}. In fact, although our trace formula in Lemma~\ref{lem:traceFormula} shows the difference between real \eqref{eq:traceEqRe} and imaginary \eqref{eq:traceEqIm} parts as in some other literature, compare \cite{ME97,LevitanSargsjan}, a simple rotation which flips the real and imaginary parts indicates that they can actually be treated in the similar way, see Section~\ref{sec:mainproof}. 
For technical reason, we first show the linearization of the character component.
\subsection{Linearization of the Character Flow}
The following result should hold for AKNS hierarchies, but we will only use the case $k=0,1$ which correspond to space and time variables respectively. Since we could not find the exact form that works in our setting, we decide to give a proof. Original proof of approximation theorem for Jacobi and Sturm-Liouville operators  can be found in \cite{SY97,EVY19}. See also \cite[Section 4]{EVY19} for discussions of convergence of many interesting objects.
\begin{theorem}\label{thm:approxFreqCharFlow} Let $\sfE\subseteq\bbR$ be a closed set satisfying the $k$-GLC \eqref{eq:kGlc} whose associated domain $\Omega=\bbC\setminus\sfE$ is a regular PWS. Let $\Theta_k(\lambda)$ be the Abelian integrals of the second kind of order $k=0,1,2,\cdots$ such that \footnote{It is easy to check that the $k$-GLC and Widom condition \eqref{eq:WidomCond} implies the convergence of the integral.}
 $$\Im\Theta_k(\lambda)
=\Im\lambda^{k+1}+\int_{\bbR\setminus \sfE}G(\xi,\lambda)\, \rmd\xi^{k+1}.$$ Then the character component of the Abel map $\mathcal{A}_c:\mathcal{D}(\sfE)\to \pi_1(\Omega)^*$ conjugates the flow $\varphi(x_0,t_0)\mapsto \varphi(x_0+x,t_0+t)$ into a linear flow in $\pi_1(\Omega)^*$. Moreover, the $B$-periods of $\Theta_k, k=0,1$ give the full set of frequencies of this flow.
\end{theorem}
\begin{proof} 
Let $\sfE^{(N)}=\bbR\setminus\left(\cup_{1\leq j\leq N}(a_j,b_j)\right)$ and $\Omega^{(N)}=\bbC\setminus \sfE^{(N)}$. Let $\{l_j\}$ be a set of generators of $\pi_1(\Omega)$ and note that $\{l_j:1\leq j\leq N\}$ is a set of generators of $\pi_1(\Omega^{(N)})$. A character $\alpha^{(N)}\in\pi_1(\Omega^{(N)})^*$ can be naturally extended as a character in $\pi_1(\Omega)^*$ by:
\begin{equation}\label{eq:extendDivisor}
\alpha^{(N)}(l_j)=0, \quad j >N.
\end{equation}
This provides a natural embedding $\pi_1(\Omega^{(N)})^*\hookrightarrow\pi_1(\Omega)^*$. A divisor $D^{(N)}\in\mathcal{D}(\sfE^{(N)})$ can be extended as a divisor in $\mathcal{D}(\sfE)$ by simply taking $D=D^{(N)}\cup\{b_j:j>N\}$. Conversely, a divisor $D\in \mathcal{D}(\sfE)$ can be projected to a divisor $D^{(N)}\in \mathcal{D}(\sfE^{(N)})$ in this manner. It follows that 
\begin{equation}\label{eq:divisorCharConverge}
\alpha^{(N)}\to\alpha,\quad D^{(N)}\to D\quad \text{as}\quad  N\to\infty.
\end{equation}

Let $\mathcal{A}_c^{(N)}:\mathcal{D}(\sfE^{(N)})\to\pi_1(\Omega^{(N)})^*$ be the Abel map in the finite gap case. It can be extended as above a map $\mathcal{D}(\sfE)\to\pi_1(\Omega)^*$. Our first goal is to show that 
\begin{equation}\label{eq:abelMapConverge}
\mathcal{A}_c^{(N)}\to\mathcal{A}_c\quad\text{and} \quad w^{(N)}_{k,j}\to w_{k,j}\quad \text{as}\quad N\to\infty,
\end{equation}
where $w^{(N)}_{k,j}$ and $w_{k,j}$ are defined as follows.
Normalized Abelian integrals of second kind on $\Omega^{(N)}$ are defined as\footnote{These Abelian integrals are normalized so that the first $N$ $B$-periods coincide with those of $\Theta_k$'s, see \cite{GesztesyYuditskii2006} for more details.} $$\Im\Theta^{(N)}_k=\Im\lambda^{k+1}+\int_{\bbR\setminus \sfE^{(N)}}G^{(N)}(\xi,\lambda) \, \rmd\xi^{k+1},$$
where $G^{(N)}(\xi,\lambda)$ is the potential theory Green's function of the domain $\Omega^{(N)}$. It is known that $\Omega^{(N)}$'s are always regular and of Parreau-Widom type and obeys DCT. $G^{(N)}(\xi,\lambda)$ vanishes on $\sfE^{(N)}$ and concave in each component of $\bbR\setminus \sfE^{(N)}$. It follows that the harmonic conjugate $\widetilde{G}^{(N)}(\xi,\lambda)$ of $G^{(N)}$ is non-decreasing on each component of $\sfE^{(N)}$ and remains constant in each component of $\bbR\setminus \sfE^{(N)}$. Then $2w_{k,j}^{(N)},2w_{k,j}$ are indeed the total variations of $\Theta^{(N)}_k,\Theta_k$ along the closed curve $l_j$. 

To show the first convergence of \eqref{eq:abelMapConverge}, we need to show that for every $j$, 
$$
\sum_{1\leq i\leq N}\frac{\epsilon_i}{2}\int_{a_i}^{\mu_i}\omega(\rmd\xi,\sfE^{(N)}_j,\Omega^{(N)})\to\sum_{i}\frac{\epsilon_i}{2}
\int_{a_i}^{\mu_i}\omega(\rmd\xi,\sfE_j,\Omega)\quad \text{as}\quad N\to\infty.$$
It suffices to prove the convergence $\omega(\xi,\sfE_j^{(N)},\Omega^{(N)})\to\omega(\xi,\sfE_j,\Omega)$ for every $j$ and that the sequence $\sum_{1\leq i\leq N}\frac{\epsilon_i}{2}\int_{a_i}^{\mu_j}\omega(\rmd\xi,\sfE^{(N)}_j,\Omega^{(N)})$ is majorized by a convergent series.
Since $\sfE$ is homogeneous, it follows that $\Omega$ is regular in the sense of potential theory and of Parreau-Widom type, compare \cite{SY97}.
Consider the harmonic functions $\omega(\xi,\sfE^{(N)}_j,\Omega^{(N)})$ and $\omega(\lambda,\sfE_j,\Omega)$, since $\Omega^{(N)},\Omega$ are regular with no internal boundary points, it follows from \cite[Theorem 5.15]{Landof1972} that $\omega(\lambda,\sfE^{(N)}_j,\Omega^{(N)})$ converges to $\omega(\lambda,\sfE_j,\Omega)$ uniformly on compacts of $\Omega$.

Let $\xi_*\in (a_1,b_1)$ and $N>j.$ The choice of $(a_1,b_1)$ is nothing essential. Let $\Gamma_j^\pm$ be loops indicated by the following picture 
\begin{figure}
\begin{tikzpicture}
\draw[dotted] (-3,0) -- (-2.1,0);
\draw[dotted] (2.1,0) -- (3,0);
\draw (3,0)--(4,0);
\draw[dotted] (4.1,0) -- (6,0);
\draw (6,0)--(7,0);
\draw (-2,0) -- (-0.9,0);
\draw[dotted] (-0.7,0) -- (0.7,0);
\draw (0.9,0)--(2,0);
\draw (2.5,0) ellipse (2.2cm and 1cm);
\draw (2.5,0) ellipse (3cm and 1.5cm);
\node at (6,-0.4) {$b_j$};
\node at (4,-0.4) {$a_j$};
\node at (-0.9,-0.4) {$a_1$};
\node at (0.9,-0.4) {$b_1$};
\node at (5,1) {$\Gamma_j^+$};
\node at (3.5,0.2) {$\Gamma_j^-$};
\node at (0,0) {$\xi_*$};
\node at (2.5,1.5) {$\rightarrow$};
\node at (2.5,1) {$\leftarrow$}; 
\node at (2.4,0.4) {$\Pi_j^-$};
\node at (5.8,1) {$\Pi_j^+$};
\end{tikzpicture}
\caption{closed loop}
\end{figure}
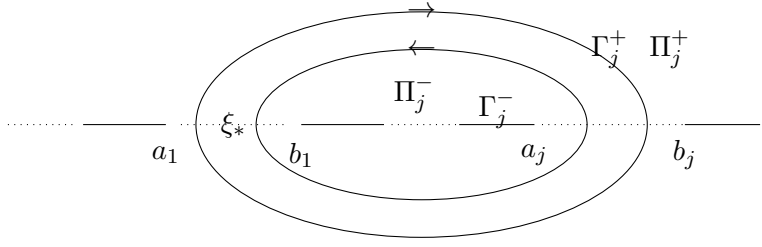
Let $\Pi_j^+$ be the connected component vicinity of $\infty$ that lies on the left side of $\Gamma_j^+$ and $\Pi_j^-$ be the vicinity of $\sfE_j^{(N)}$ that lies on the left side of $\Gamma_j^-$. Assume that $\Pi_j^+\cap\Pi_j^-=\Gamma_j^+\cap\Gamma_j^-=\emptyset$ and  $\xi_*\notin\Gamma_j^\pm$. 
Define $m_j^\pm=\min_{\lambda\in\Gamma_j^\pm}G(\xi_*,\lambda).$
By the Maximum Principle for harmonic functions on $\Pi_j^\pm\cap \Omega$,
$$\begin{aligned}
&m_j^-\omega(\lambda,E_j)\leq G(\xi_*,\lambda), \quad \lambda\in\Pi_j^+\\
&m_j^+(1-\omega(\lambda,E_j)\leq G(\xi_*,\lambda),\quad \lambda\in \Pi_j^-.
\end{aligned}$$
Indeed, these inequalities can be easily verified on the boundaries of $\Pi_j^\pm$ and $\Omega$ and thus follow from the maximum principle for harmonic functions. 
For each $\lambda\in(a_i,b_i)$ it follows that 
$$\max\{\omega(\lambda,E_j),1-\omega(\lambda,E_j)\}\leq \frac{G(\xi_*,\lambda)}{\min\{m_j^+,m_j^-\}}\leq \frac{h_i}{\min\{m_j^+,m_j^-\}},$$
since $h_i$ is the critical value of $G(\xi_*,\lambda)$ in $(a_i,b_i)$.  Therefore, the series \begin{equation}\label{eq:majorantSeries}\frac{1}{\min\{m_j^+,m_j^-\}}\sum_ih_i
\end{equation} is the majorant since $\Omega$ is of Parreau-Widom type and thus the conditon $\sum_i h_i<\infty$ holds. The linearization property of the Abel map follows from the constructions of normalized Abelian differentials and the Dubrovin system of differential equations. This can be made explicit in the finite gap case by \cite[Theorem F.10]{GH03}. The infinite gap case follows from above approximation process.

Our next goal is to prove the second convergence in \eqref{eq:abelMapConverge}, namely, the convergence of frequencies.
It is known that in the finite gap case, $\Theta^{(N)}_k$ coincides with the normalized Abelian integrals of second kind of order $k$.  Then the computation \cite[eq. (F.46)]{GH03} shows that the frequency vector $\partial_{t_k}\mathcal{A}_c(D(x,t))$ equals the $B$ periods of the normalized Abelian integral of the second of order $k$, where we interpret $t_0=x,t_1=t$. One may refer to \cite{GH03} for an argument through Riemann-Theta function, or \cite{SY1995} for an argument through conformal mapping. 

For the case $k=0$ we can understand $\Theta_k^{(N)},\Theta_k$ as conformal mappings from the upper half plane $\bbC_+$ to some comb domains $\Pi_N,\Pi$ and majorants of a symmetric class $K_\varphi$ \cite{LevinPart2}. Let $v^{(N)}_k,v_k$ be the imaginary parts of $\Theta^{(N)}_k,\Theta_k$ respectively. Then the convergence \begin{equation}\label{eq:convergeOfImMaj}
v_k^{(N)}\to v_k\quad \text{as}\quad N\to\infty
\end{equation}
indicates that variation of the real parts of $\Theta^{(N)}_k$ converges to the variation of the real part of $\Theta_k$. However, this argument fails on the case $k>0$ since $\Theta_k$ is not a conformal mapping from $\bbC_+$ onto such a comb domain. Interested reader may refer to \cite{SY1995} for an argument of convergence \eqref{eq:convergeOfImMaj} via \cite[Theorem 2.6]{LevinPart2} and \cite[Theorem 3.2]{LevinPart3} in the case $k=0$.

In the general case $k\geq 0$, to prove \eqref{eq:convergeOfImMaj}, by \cite[Theorem 5.15]{Landof1972} it suffices to prove 
\begin{equation}\label{eq:convergeOfGreenFunc}
\lim\limits_{N\to\infty}G^{(N)}(\xi,\lambda)\to G(\xi,\lambda).
\end{equation}
But this is a consequence of the convergence of the corresponding harmonic measures on compacts of a regular domain $\Omega$ without internal boundary points we have discussed and the condition $$\int_{\bbR\setminus \sfE}G(\xi,\lambda) \, \rmd\xi^{k+1}<\infty.$$ This completes the proof of the theorem.
\end{proof}
The following result follows as a direct consequence.
\begin{coro}\label{thm:characterFlow}
There exists $\eta,\eta^{(1)}\in\bbR^\infty$ such that the image of  of the translation flow 
$\varphi(x_0,t_0)\to\varphi(x_0+x,t_0+t)$ under the Abel map $\mathcal{A}_c$ given by Theorem~\ref{thm:JacobiInver} in the character group $\pi_1(\Omega)^*$ is a linear flow
$$\zeta+x\eta+t\eta^{(1)},$$
where $\zeta=\mathcal{A}_c\circ\mathcal{B}(\varphi)\in\pi_1(\Omega)^*$ is a character independent of $x,t$.
\end{coro}
\begin{proof}
It can be identified that $\eta_j$ equals the $B_j$ period of $\Theta_0$ and $\eta^{(1)}_j$ equals the $B_j$ period of $\Theta_1.$
\end{proof}
\subsection{Linearization of the Rotation Coordinate} Compared to the KdV and Schr\"odinger analogue, the Abel map in our literature generates an extra coordinate which does not belong to the character group $\pi_1(\Omega)^*.$ Indeed, one may check that $U_\tau=\mathrm{diag}(\tau^{1/2},\tau^{-1/2})$ induces a scaling in $L^2(\bbR,\bbC^2)$ which is equivalent to the isospectral transform $\varphi\mapsto\tau^{-1}\varphi$ in the space of potentials. We understand this phenomenon as the consequence of the scaling invariance of the NLS and the Dirac operator.
Although in the stationary case, $\tau$ can be chosen independent of the space variable $x$, compare \cite{BLY2}, its time dependence is not trivial in non-stationary case. In this case, $\tau$ is necessarily a function of $x,t$ which reveals the intrinsic rotating effect in the solution. The following theorem gives a detailed description of this effect.

\begin{theorem}\label{thm:linearzRotFlow}Assume that $\Omega=\bbC\setminus \sfE$ is a PWS and $\sfE$ satisfies the finite gap length condition \eqref{eq:finiteGapCond},
Assume the following conditions 
\begin{equation}\label{eq:gapMomentCond}\int_{\bbR\setminus \sfE}M(\xi) \rmd\xi^{k+1}<\infty, k=0,1,\end{equation}
hold,
then the second component $\mathcal{A}_r$ of the generalized Abel map defined by Definition~\ref{def:rotFlowAbelMap} conjugates the translation $\varphi(x_0,t_0)\mapsto\varphi(x_0+x,t_0+t)$ into a linear flow in $\bbT$. Moreover, the rotation coordinate can be specified as
\[\tau(x,t)=C\exp(-2\pi i(\vartheta_0x+\vartheta_1 t)),~|C|=1,\]
the frequencies are precisely 
\[\vartheta_k=2\int_{\bbR\setminus \sfE}\widetilde{M}(\xi,\infty_+)-(-\widetilde{M}(\xi,\infty_-)) \, \rmd\xi^{k+1},~k=0,1.
\]
\end{theorem}
\begin{proof}
Let $\sfE_N=\bbR\setminus \cup_{i=1}^{N+1}(a_i,b_i)$ and $\mathcal{R}_N=\{(z,w):w^2=\prod_{i=1}^{N+1}(z-a_i)(z-b_i)\}\cup\{\infty_+,\infty_-\}$. $\mathcal{R}_N$ is the two-sheet covering of $\Omega_N=\bbC\setminus \sfE_N$. Let $\{A_i,B_i\}_{i=1}^N$ be a homology basis of $\mathcal{R}_N$ constructed as follows. Pick $\xi_*\in (a_0,b_0)$ as the base point, let $A_j$ be the equivalent class of cycles circling around $(a_j,b_j)$ in the upper sheet and  let $B_j$ be the cycle starting at $\xi_*$ and goes through $(a_j,b_j)$ downward and return to $\xi_*$ in the lower sheet. Let $\{\omega_j^{(N)}\}_{j=1}^N$ be a system of Abelian integrals of the first kind normalized by 
\begin{equation}\label{eq:normalAbelFirst}\int_{A_j}\! \rmd\omega^{(N)}_k=\delta_{jk}.\end{equation}
Let $\omega^{(N)}(z_0,F_k)$ be the harmonic measure of $F_k$, the intersection of $\sfE$ with the (bounded) domain enclosed by the loop $B_k$. Then the normalization condition \eqref{eq:normalAbelFirst} indicates that 
$$\omega^{(N)}_k=\omega^{(N)}(\xi,F_k)+i\widetilde{\omega}^{(N)}(\xi,F_k),$$
where $\widetilde{\omega}^{(N)}(\xi,F_k)$ is the harmonic conjugate of $\omega^{(N)}(\xi,F_k).$
Let $M^{(N)}_{k}(\lambda)$ be the Abelian integral of second kind of $\Omega_N$ and let $M^{(N)}(\xi,\lambda)=\frac{G^{(N)}(\xi,\lambda)}{G^{(N)}(i,\lambda)}$ be the Martin kernel  such that 
$$M^{(N)}_{k}(\lambda)=\Im\lambda^{k+1}+\int_{\bbR\setminus \sfE_N}M^{(N)}(\xi,\lambda)\, \rmd\xi^{k+1},k=0,1.$$
Let $\Theta^{(N)}_{k}$ be the complex Martin function 
$$\Theta^{(N)}_k(\lambda)=iM^{(N)}_k(\lambda)-\widetilde{M}^{(N)}_k(\lambda).$$
Let $\Omega_{N,k,\pm}^{(2)}$ be the  Abelian integrals of the second kind on upper and lower sheets of $\mathcal{R}_N$ of order $k$ with the only pole at $\infty_\pm$ respectively  with asymptotics
$$\Omega^{(2)}_{N,k,\pm}(\lambda)=\lambda^{k+1}+O(1)+O(1/\lambda), \lambda\to\infty_\pm.$$
Note that in local coordinate $\zeta^2=\lambda$ as $\lambda\to\infty_\pm$, 
 $$d\Omega^{(1)}_{N,k,\pm}(\zeta)=d\Omega^{(2)}_{N,k,\pm}(\zeta)-(2k+2)\zeta^{2k+1}d\zeta$$
are Abelian differentials of the first kind.

For the rotation coordinate, in the finite gap case the frequencies are given by 
\begin{equation}\label{eq:finiteFrequency}
\vartheta^{(N)}_{k}
=\int_{\xi_*}^{+\infty} \! \rmd(\Omega^{(1)}_{N,k,+}-\Omega^{(1)}_{N,k,-})-\int_{\xi_*}^{-\infty} \! \rmd(\Omega^{(1)}_{N,k,+}-\Omega^{(1)}_{N,k,-}),
\end{equation}
where $\xi_*$ is an arbitrary point in a gap, compare \cite[eq 4.3.6]{BBEIM1994}. In these integrals, we require the path from $\xi_*$ to $\infty_+$ is related to the path from $\xi_*$ to $\infty_-$ by the anti-involution between sheets and reversed  directions. Different choice of $\xi_*$ clearly results in a coordinate in $\pi_1(\Omega_N)^*.$ Interpreting $\Theta^{(N)}_{k}$ as Abelian integrals of the second kind of order $k=0,1$, see \cite{VolbergYuditskii2016}, we can rewrite \eqref{eq:finiteFrequency} in terms of variations of the harmonic conjugate of Martin function
\begin{equation}\label{eq:GreenFrequency}
\vartheta^{(N)}_k=2\int_{\bbR\setminus \sfE_N}\widetilde{M}^{(N)}(\xi,\infty_+)-(-\widetilde{M}^{(N)}(\xi,\infty_-)) \, \rmd\xi^{k+1}.
\end{equation}
The rotation coordinate in this case is specified by 
\begin{equation}\label{eq:finiteGapRep}
\tau^{(N)}(x,t)=C^{(N)}\exp(-2\pi i(\vartheta_0^{(N)}x+i\vartheta^{(N)}_1t)),\quad |C^{(N)}|=1.
\end{equation}

Our next goal is to relate \eqref{eq:finiteGapRep} to the definition of rotation coordinate in the infinite gap case. In the infinite gap case, recall $\tau=-\frac{\mathscr{K}_D(z)}{(\mathscr{K}_D)_\sharp(z)}$, with $\mathscr{K}_D(z)$ given by \eqref{eq.prodRepRepKer}. By Theorem \ref{thm:identifyRotCord},
\begin{equation}\label{eq:newAbelMap}\tau=\lim\limits_{\lambda\to\infty_+}-\frac{\mathscr{K}_{D}(\lambda)}{(\mathscr{K}_{D})_{\sharp}(\lambda)}=C\exp\left(2\pi i\sum_j-\frac{\epsilon_j}{2}\int_{a_j}^{\mu_j}\omega(\rmd\xi,\sfE_\infty)\right),
\end{equation}
where $\sfE_\infty=(\xi_*,\infty)\cap \sfE.$  The convergence $\tau^{(N)}\to\tau$ follows from the same argument as in the proof of Theorem \ref{thm:approxFreqCharFlow}.
The frequencies are given by the generalized $B$-periods
\begin{equation}\label{eq:infiniteFrequency}\begin{aligned}
\vartheta_{k}&=\lim\limits_{N\to\infty}2\int_{\xi_*}^{+\infty} \! \rmd(\Omega^{(1)}_{N,k,+}-\Omega^{(1)}_{N,k,-})
=2\int_{\bbR\setminus \sfE}\widetilde{M}(\xi,\infty_+)-(-\widetilde{M}(\xi,\infty_-)) \, \rmd\xi^{k+1}\\
&=\lim\limits_{N\to\infty}\vartheta^{(N)}_k  \quad\text{ for } k=0,1.
\end{aligned}
\end{equation}

It can be shown by the approximation process of Theorem~\ref{thm:approxFreqCharFlow},  under the conditions that $\Omega$ is a PWS obeying DCT and \eqref{eq:gapMomentCond}, that \eqref{eq:GreenFrequency} converges to \eqref{eq:infiniteFrequency} and  to the correct limit. The rotation coordinate in the infinite gap case is then specified as \begin{equation}\label{eq:rotationFlowExp}\tau(x,t)=C\exp(-2\pi i(\vartheta_0x+i\vartheta_1t))
\end{equation}
where $C$ is a unimodular constant. 
This completes the proof.
\end{proof}

\begin{remark}
    It is important to guarantee the integrability of \eqref{eq:gapMomentCond}. While the modified Widom's condition \eqref{eq:mPRWCond} and the $k$-GLC \eqref{eq:kGlc} are sufficient, they are too strong and necessarily put conditions on the density of states; compare \cite[Theorem 1.1]{EVY19}. In our case that $\sfE$ is a  closed subset of $\bbR$ unbounded from above and below, we are in the Akhiezer-Levin setting, see \eqref{eq:ALcondition}. According to \cite[Lemma 3.2]{EGL}, the symmetric Martin function $M(\xi)$ appeared in \eqref{eq:gapMomentCond} grows at most linearly in each gap. We can easily translate the integrability for $k=0,1$ into the following conditions
    $$\sum_{j}(1+\eta_{0j})\gamma_j<\infty \text{ and } \sum_j(1+\eta_{0j})(|b_j^2-a_j^2|)<\infty.$$
    It turns out that our Craig type conditions are sufficient. In the case of higher AKNS hierarchies, one needs to introduce stronger conditions.
\end{remark}

\section{Quantitative Almost Reducibility and its Spectral Applications}
\label{arac}

In this section, we apply quantitative almost reducibility to prove the purely absolutely continuous spectrum and gap estimates for the Dirac operators.

\subsection{Quantitative Almost Reducibility}\label{qal}

For any $A\in C^{\omega} (\bbT^d, \su(1,1))$, consider the quasiperiodic linear systems represented by \begin{equation} \label{eq:generalQPsyst}
\left\{
\begin{aligned}
& \frac{{\rmd}f}{{\rmd}x}=A(\theta)f\\
&  \frac{{\rmd}\theta}{{\rmd}x}  = \omega
\end{aligned}
\right.
\end{equation} 
and denote this system as $(\omega,A)$. 
For $A_1, A_2\in C^\omega(\bbT^d, \su(1,1))$ and $W\in C^{\omega}(\bbT^d,$ ${\rm PSU}(1,1))$, we say that $(\omega,A_1)$ is {\it conjugated} to $(\omega,A_2)$ by $W$ if 
$$\partial_\omega W=A_1 W- W A_2,$$ 
where $ \partial_\omega W
:=\langle \omega, \nabla W\rangle$. 
We say that $(\omega,A_1)$ is {\it reducible} if it can be conjugated to a constant system $(\omega,A_2)$ with $A_2\in \su(1,1)$;  we say that $(\omega, A_1)$ is {\it almost reducible} if the closure of its analytical conjugacies contains a constant system $(\omega,C)$.

Before introducing the precise quantitative almost reducibility result, let us introduce the parameters inductively. 
For any sufficiently small $\epsilon_{0} > 0$ and any $h > 0$, let us define the following sequences:
\begin{equation}\label{eq:KAMqantities}
\epsilon_{j} = \epsilon_{0}^{2^{j}}, \quad h_{j} = \frac{h}{2^{j}}, \qquad N_{j} = \frac{4^{j+1} \ln \epsilon_{0}^{-1}}{h}.
\end{equation}
Within these parameters, we have the following result:

\begin{prop}[\cite{LYZZ},\cite{DLZ22}]\label{rd1}
Assume that $\kappa, \tau, h > 0$ and $\omega \in \DC_d(\kappa,\tau)$. Let $A_{0} \in \su(1,1)$, $F_{0} \in C^\omega_{h}(\bbT^{d}, \su(1,1))$ with
$$
\Vert F_{0} \Vert_{h} \leq \epsilon_{0}(d,\kappa,\tau,h) \leq D_{0}(1+\left\| A_{0}\right\|^{C_{0}}) \left( \frac{h}{2} \right)^{C_{0}},
$$
where $D_{0} = D_0(\kappa,\tau,d,h)$ and $C_{0}=C_0(d,\tau)$\footnote{The optimal choice of $C_0$ can be found in \cite{LYZZ}.}. Then for any $j \geq 1$, there exists $B_{j} \in C^\omega_{h_{j}}(\bbT^{d}, \mathrm{PSU}(1,1))$ such that
\begin{equation}\label{eq:conjugateEq}
\partial_{\omega}B_j=(A_0+F_0)B_j-B_j(A_j+F_j),
\end{equation}
where $\Vert F_{j} \Vert_{h_{j}}\leq \epsilon_{j}$ and $B_{j}$ satisfies
\begin{align}
\label{normOfB}
\left\| B_{j}\right\|_{0} & \leq \epsilon_{j-1}^{-\frac{1}{192}}, \\
\label{degreeOfB}\vert \deg{B_{j}}\vert & \leq 2N_{j-1}.
\end{align}
More precisely, we have

{\rm (a)}  if for some $m\in \bbZ^{d}$ with  $0 < \vert m \vert < N_{j-1}$, such that 
$$ \Vert 2 \rho( \omega, A_{j-1})- \langle m, \omega \rangle \Vert_{\bbR/\bbZ} < \epsilon_{j-1}^{\frac{1}{15}},$$ we have the following precise expression:
$$
A_{j} =  \begin{bmatrix} i  t_{j} & v_{j} \\ \bar{v}_{j} & -i  t_{j} \end{bmatrix},
$$
where $t_{j} \in \bbR$, $v_{j} \in \bbC$, and $\vert t_{j} \vert \leq \epsilon_{j-1}^{\frac{1}{16}}$, $\vert v_{j} \vert \leq \epsilon_{j-1}^{\frac{15}{16}}$.

{\rm (b)} Moreover, there always exist unitary matrices $U_{j}\in \SL(2,\bbC)$ such that
\begin{equation}\label{e.conj}
U_{j}(A_{j} + F_j(\cdot))U_{j}^{-1} = \begin{bmatrix} i e_{j} & c_{j} \\ 0 &  -i e_{j} \end{bmatrix} + F_{j}(x)
\end{equation}
where $\sfE_{j} \in \bbR \cup i \bbR$, with estimates $\Vert F_{j} \Vert_{h_{j}} \leq \epsilon_{j}$, and
\begin{equation}\label{Btimesc}
\Vert B_{j} \Vert_{0}^{2} \vert c_{j} \vert \leq 8 \Vert A_{0} \Vert.
\end{equation}

\end{prop}

\subsection{Purely Absolutely Continuous Spectrum}
As a solid application of our Jitomirskaya--Last type inequality and quantitative almost reducibility, we first prove the purely absolutely continuous spectrum for the Dirac operators.

\begin{theorem}\label{thm:ac}
Assume that $\omega\in \DC_d(\kappa,\tau)$ and $\|\widetilde{\varphi}\|_{h}<\epsilon_0$ for some $h>0$ and sufficiently small $\epsilon_0=\epsilon_0(d,\kappa,\tau,h)$. Then the spectrum of the associated Dirac operator $\Lambda_\varphi$ defined by \eqref{eq:diracOperLambda} is purely absolutely continuous.
\end{theorem}

Recall  the eigenvalue equation:
$
\Lambda_\varphi f=\lambda f
$
 can be rewritten as  
\begin{equation}\label{eq.diffEq}
\left\{
\begin{aligned}
&  f'
=\left(\begin{bmatrix}-i \lambda&0\\0& i  \lambda\end{bmatrix} + \begin{bmatrix}0 & i \varphi(\theta)\\-i \overline{\varphi(\theta)}&0\end{bmatrix}\right)f
=(A_{0}(\lambda)+F_{0}(\theta))f,\\
&\theta'=\omega,
\end{aligned}
\right.
\end{equation}
Thus if $\|\widetilde{\varphi}\|_{h}$ is small enough, such that 
$$
\| F_{0}\|_h \leq \epsilon_{0} \leq D_{0}(1+\left\| A_{0}\right\|^{C_{0}}) \left( \frac{h}{2} \right)^{C_{0}},
$$
then we can apply Proposition~\ref{rd1} to get almost reducibility of \eqref{eq.diffEq}, i.e. there exists $B_{j} \in C^\omega_{h_{j}}(\bbT^{d}, \mathrm{PSU}(1,1))$ such that
\begin{equation}\label{eq:conjugateEq1}
\partial_{\omega}B_j=(A_0+F_0)B_j-B_j(A_j+F_j),
\end{equation}
where $\Vert F_{j}(\theta) \Vert_{h_{j}}\leq \epsilon_{j}$. Moreover, for any $m\in\bbZ$ with $0<|m|<N_{j-1}$, define the resonant set \begin{align}
\label{e.lambdamjdef}
\Lambda_{m}(j) = \{ \lambda \in \Sigma_\varphi : & \Vert 2 \rho( \omega, A_{j-1})- \langle m, \omega \rangle \Vert_{\bbR/\bbZ} < \epsilon_{j-1}^{\frac{1}{15}}  \}.
\end{align}
and then the resonant set at the $j$-th interation step:
\begin{equation}\label{e.Kjdef}
K_{j}=\bigcup_{0<\vert m\vert\leq N_{j-1}}\Lambda_{m}(j)
\end{equation}
The quantitative estimates of Proposition~\ref{rd1}  immediately imply the following:

\begin{lemma}[\cite{DLZ22}]\label{lem:rotEst}
For $\lambda\in K_{j}$, there exists  $n_{j}\in\bbZ^{d}$ with $\vert n_{j}\vert \leq 2N_{j-1}$ such that
$$
\|2\rho(\omega,A_0+F_0)-\langle n_{j},\omega\rangle\|_{\bbR/\bbZ}\leq 2\epsilon_{j-1}^{\frac{1}{15}}.
$$
\end{lemma}

Let $U(\lambda,x,\theta)$ denote the fundamental solution of \eqref{eq.diffEq} that satisfies the condition $U(\lambda,0,\theta)=I$. Then, for any $\lambda\in K_{j}$, we obtain the following estimate for the behavior of $U(\lambda,x,\theta)$, serving as a direct corollary.

\begin{coro}\label{cor.transferGrowth}
For $\lambda \in K_{j}$,  we have
$$
\sup \limits_{0 \leq x \leq \epsilon_{j-1}^{-\frac{1}{16}}} \Vert U(\lambda,x,\theta) \Vert_{0} \leq C \epsilon_{j-1}^{-\frac{1}{96}}\qquad \forall \theta \in \bbT^d.
$$
\end{coro}

\begin{proof}
Given that $\lambda\in K_j$, there exists a function $B_j\in C^\omega(\bbT^d,\mathrm{PSU}(1,1))$ that conjugates the system \eqref{eq.diffEq} into $$g'=(A_j(\lambda)+F_j(\theta))g.$$ Let $\widetilde{U}_j(\lambda,x,\theta)$ denote the fundamental solution of this conjugated system. Consequently, we have $$ \widetilde{U}_j(\lambda,x,\theta) = e^{x A_j(\lambda)}\left(id + \int_0^x e^{-tA_j(\lambda) } F_j(\theta+t\omega) \widetilde{U}_j(\lambda,t,\theta) \,\rmd t \right),$$ where now we define $G_j( \lambda,x,\theta)=e^{-x A_j(\lambda)} \widetilde{U}_j(\lambda,x,\theta).$ This leads to $$G_j(\lambda,x,\theta) =id + \int_0^x e^{-tA_j(\lambda) } F_j(\theta+ t \omega) e^{tA_j(\lambda) } G_j(\lambda,t,\theta) \, \rmd t,$$ 
and we introduce $g_j (\lambda,x)= \|G_j(\lambda,x,\theta)\|_0$. 
Noting that $\Vert A_j\Vert_0\leq \epsilon_{j-1}^{\frac{1}{16}}$, we obtain
$$\|e^{-tA_j(\lambda) }\|\leq 2 \qquad \forall t\leq \epsilon_{j-1}^{-\frac{1}{16}}, $$ which in turn implies $$g_j (\lambda,x) \leq 1+ \int_0^x 4 g_j(\lambda,t) \,\rmd t $$ for $x\leq \epsilon_{j-1}^{-\frac{1}{16}}$. 
By applying Gronwall's inequality, we find that $g_j (\lambda,x) \leq e^{4 \epsilon_{j} x}$, which in turn implies $$\sup_{0\leq x\leq \epsilon_{j-1}^{-\frac{1}{16}}}\Vert \widetilde{U}_j(\lambda,x,\theta)\Vert_0\leq C.$$ This, in conjunction with \eqref{normOfB}, yields $$\sup_{0\leq x\leq \epsilon_{j-1}^{-\frac{1}{16}}}\Vert U(\lambda,x,\theta)\Vert_0\leq C\Vert B_j\Vert^2_0\leq C\epsilon_{j-1}^{-\frac{1}{96}}.$$
\end{proof}

Moreover, we obtain zero Lyapunov exponents throughout the spectrum:

\begin{theorem}\label{thm:zeroLyapunov}
Assume that $\omega\in \DC_d(\kappa,\tau)$ and $\|\widetilde{\varphi}\|_{h}<\epsilon_0$ for some $h>0$ and sufficiently small $\epsilon_0=\epsilon_0(d,\kappa,\tau,h)$. Then for every $\lambda\in\Sigma_\varphi,$ we have   $\gamma(\lambda)=0.$
\end{theorem}
\begin{proof}
We distinguish the proof into two cases:\\
\smallskip
 \textbf{Case 1: $\lambda\in K_j$ for finitely many $j.$}
In this scenario, there exists an analytic transformation  $B$ that conjugates the system \eqref{eq.diffEq} into $(\omega,A)$, where $A\in \su(1,1)$ is a constant matrix, indicating that \eqref{eq.diffEq} is reducible. Since $\lambda\in\Sigma_\varphi$, Corollary~\ref{unre} implies that $A$ is not hyperbolic. Consequently, either $A$ is elliptic, leading to boundedness of the system \eqref{eq.diffEq}, or $A$ is parabolic, resulting in linear growth of \eqref{eq.diffEq}.\\
\smallskip
 \textbf{Case 2: $\lambda\in K_j$ for infinitely many $j.$} 
 Before estimating the growth of its  fundamental solution, we have the following observation:
 \begin{claim}\label{claim1}
  Suppose that $\lambda\in\Sigma_\varphi,$ then for $j$ sufficiently large, there exists $B_j$  which conjugates the system $(\omega,A_0+ F_0(\cdot))$ into $(\omega,A_j+ \widetilde{F_j}(\cdot))$ 
  where $$A_j+\widetilde{F_j}(\cdot)=\begin{bmatrix}i e_j(\lambda)& c_j\\0&-i e_j(\lambda)\end{bmatrix}+\widetilde{F_j}(\cdot)$$
  with estimates  
  \begin{equation*}
e_j(\lambda)\in\bbR, \quad\Vert \widetilde{F_j}\Vert_0<\epsilon_j^{\frac{1}{4}},\quad \Vert B_j\Vert_{0}^{2} \vert c_{j} \vert \leq 8 \Vert A_{0} \Vert.
  \end{equation*}
 \end{claim}
 \begin{proof}
Let $B_j(\cdot)$, as given by Proposition~\ref{rd1}, transform the system $(\omega,A_0+F_0(\cdot))$ into $(\omega, A_j+F_j(\cdot))$. We denote by $U_j(x,\theta,\lambda)$ the corresponding fundamental solution matrix, as usual.
We focus on the case where $\det(A_j)<0$. In this scenario, $A_j$ possesses two real eigenvalues, which we denote as $\pm e_j(\lambda)=\pm\sqrt{-\det(A_j)}$. Assuming $\sfE_j>\epsilon_j^{\frac{1}{4}}$, there exists a matrix $Q_j$ such that  $$Q_jA_jQ_j^{-1}=\begin{bmatrix}e_j&0\\0&-e_j\end{bmatrix}$$ with $\Vert Q_j\Vert^2\leq \Vert A_j\Vert\epsilon_j^{-\frac{1}{4}}$
(consult for example \cite{NXQ2024}). Therefore, $Q_j$ further conjugates the system $(\omega, A_j+F_j(\cdot))$ into
$$g'=\left(Q_j(A_j+F_j (\cdot))Q_j^{-1}\right)g=\left(\begin{bmatrix}e_j&0\\0&-e_j\end{bmatrix}+ \overline{F_j}(\cdot)\right)g,$$ 
with $e_j>\epsilon_j^{\frac{1}{4}}$ and $\|\overline{F_j}\|\leq \Vert Q_j F_jQ_j^{-1}\Vert< \epsilon_j^{\frac{3}{4}}$.

By \cite[Lemma 3.1]{HY2012}, there exists $\bar{B}_j\in C^{\omega}(\bbT^d, PSU(1,1))$ that further conjugates the system into \begin{equation}\label{con1}
g'
= \begin{bmatrix}
e_j+ f_j(\theta) & 0 \\ 
0 & -e_j-f_j(\theta) 
\end{bmatrix} g,\end{equation}
where $\|f_j\|_0 \leq 2 \epsilon_j^{\frac{3}{4}}$.
This implies that the system is uniformly hyperbolic, which contradicts the assumption that $\lambda\in \Sigma_\varphi$ by Corollary~\ref{unre}. Therefore, for the case $\det(A_j)<0$, we incorporate everything into the perturbative term denoted by $\widetilde{F_j}(\cdot)$, with a slightly worse estimate $\Vert \widetilde{F}_j\Vert_0<\epsilon_j^{\frac{1}{4}}$. We rewrite the system \eqref{con1} as $(\omega,\widetilde{F}_j(\theta) )$.

Let $\widetilde{U}_j(\lambda,x,\theta)$ denote the fundamental solution of $(\omega,A_j+ \widetilde{F_j}(\cdot))$. By the variation of constants formula, we have $$ \widetilde{U}_j(\lambda,x,\theta) = e^{x A_j}\left(id + \int_0^x e^{-tA_j } \widetilde{F}_j(\theta+t\omega) \widetilde{U}_j(\lambda,t,\theta)\, \rmd t \right).$$
Applying the same argument as Corollary~\ref{cor.transferGrowth}, we obtain $$\sup_{0\leq x\leq \epsilon_{j-1}^{-\frac{1}{16}}}\Vert \widetilde{U}_j(\lambda,x,\theta)\Vert_0\leq 2(1+|c_j|x).$$ This leads to $$\sup_{0\leq x\leq \epsilon_{j-1}^{-\frac{1}{16}}}\Vert U(\lambda,x,\theta)\Vert_0\leq 2 \Vert B_j\Vert_0^2 (1+|c_j|x)\leq C (\epsilon_{j-1}^{-\frac{1}{96}}+|x|).$$
Let $I_j=(\epsilon_{j-1}^{-\frac{1}{96}},\epsilon_{j-1}^{-\frac{1}{16}})$. Since $\epsilon_{j+1}=\epsilon_j^2$ and $\{I_j:j\in\bbN_+\}$ covers $\bbR_+$, we have $ \Vert U(\lambda,x,\theta)\Vert_0\leq C|x| $ for $x\in I_j$.
The result follows from this estimate.
\end{proof}

With the claim proved, we are done.
\end{proof}

We also have the standard  $\frac{1}{2}$-H\"older regularity for the integrated density  states measure.
\begin{theorem}[H\"older continuity]\label{thm.holderCon}Assume that $\omega\in \DC_d(\kappa,\tau)$ and $\|\widetilde{\varphi}\|_{h}<\epsilon_0$ for some $h>0$ and sufficiently small $\epsilon_0=\epsilon_0(d,\kappa,\tau,h)$.
There exists a numerical constant $C>0$ with 
$$\mathcal{N}(\lambda+\epsilon)-\mathcal{N}(\lambda-\epsilon)\leq C \epsilon^{1/2}.$$
\end{theorem}

\begin{proof}
Let us first show the following claim:

\begin{claim}
For any $\epsilon \in \left( \epsilon_j^{\frac{1}{4}}, \epsilon_{j-1}^{\frac{1}{48}} \right)$ sufficiently small, there exists $W \in C^\omega\left(\bbT^d, \mathrm{PSU}(1,1)\right)$ with $\Vert W \Vert_0 \leq C\epsilon^{-\frac{1}{4}}$ and $Q \in C^\omega\left(\bbT^d, \su(1,1)\right)$ such that
$$\partial_\omega W=(A_0+F_0)W-WQ,$$
with estimate $$ \left\Vert Q-\begin{bmatrix}i e_j&0\\0&-i e_j\end{bmatrix}\right\Vert_0\leq C\epsilon^{\frac{1}{2}}.$$
\end{claim}
\begin{proof}
By Claim~\ref{claim1}, there exists $B_j \in C^\omega\left(\bbT^d, \mathrm{PSU}(1,1)\right)$ such that 
$$\partial_\omega B_j=(A_0+F_0)B_j-B_j(A_j+\widetilde{F_j})$$
where $A_j+\widetilde{F_j}(\cdot)=\begin{bmatrix}i e_j& c_j\\0&-i e_j\end{bmatrix}+\widetilde{F_j}(\cdot)$  with estimates  
  \begin{equation}\label{bcj}
e_j(\lambda)\in\bbR, \quad\Vert \widetilde{F_j}\Vert_0<\epsilon_j^{\frac{1}{4}},\quad \Vert B_{j} \Vert_{0}^{2} \vert c_{j} \vert \leq 8 \Vert A_{0} \Vert.
  \end{equation}
  Let $d = \Vert B_j \Vert_0\epsilon^{\frac{1}{4}}$ and 
  define $D=\begin{bmatrix}d^{-1}&0\\0&d \end{bmatrix}$. 
  Then $W= B_j D$  satisfies:
$$
\partial_\omega W
=(A_0+F_0)W-W D^{-1}
\left(\begin{bmatrix}i e_j& c_j\\0&-i e_j\end{bmatrix} + \widetilde{F_j}\right)D.$$
Then by \eqref{bcj}, $Q$ defined by  
\begin{eqnarray*}
Q=D^{-1}\left(\begin{bmatrix}i e_j&c_j\\0&-i e_j\end{bmatrix}+\widetilde{F_j}\right)D
= \begin{bmatrix}i e_j& \Vert B_j\Vert_0^2c_j \epsilon^{\frac{1}{2}} \\0&-i e_j\end{bmatrix}+ D^{-1}\widetilde{F_j}D
\end{eqnarray*}
satisfies the desired conditions.     
\end{proof}
  
For $\lambda \in \Sigma_\varphi$ and any $\epsilon > 0$ sufficiently small, the system $(\omega, A_0(\lambda + i\epsilon) + F_0(\theta))$ can be considered as an $\epsilon$-perturbation of $(\omega, A_0(\lambda) + F_0(\theta))$. It follows that $$
\partial_\omega W = (A_0(\lambda + i\epsilon) + F_0)W - W\widetilde{Q},
$$ where $$
\widetilde{Q} = Q + \epsilon W^{-1}\begin{bmatrix} 1 & 0 \\ 0 & -1 \end{bmatrix}W,
$$ with the estimate 
$$
\left\Vert \widetilde{Q} - \begin{bmatrix} i e_j & 0 \\ 0 & -i e_j \end{bmatrix} \right\Vert_0 \leq C\epsilon^{\frac{1}{2}}.
$$ 
Let $\widetilde{\mathcal{Q}}(\lambda + i\epsilon,x, \theta)$ be the fundamental solution of $(\omega, \widetilde{Q}(\cdot))$. It follows that 
\begin{equation}\label{eq:complexExponent}
\gamma(\lambda + i\epsilon) = \lim\limits_{L \to \infty} \frac{1}{L} \ln \Vert \widetilde{\mathcal{Q}}(\lambda + i\epsilon,L, \theta) \Vert_0 \leq C\epsilon^{\frac{1}{2}}.
\end{equation}

By Theorem~\ref{thm:zeroLyapunov},  we have $\gamma(\lambda)=0$ for $\lambda\in\Sigma_\varphi$.
Applying the Thouless formula \cite[Lemma 3.2]{ES96} yields
\begin{equation}\label{eq.Thouless}
\gamma(z)
=c_0\Re z+c_1+\int_{-1}^{1}\log|\ell-z| \, \rmd \mathcal{N}(\ell)
+\int_{\bbR\setminus(-1,1)}\left(\log\left|1-\tfrac{z}{\ell}\right|+\tfrac{\Re z}{\ell}\right) \, \rmd \mathcal{N}(\ell),
\end{equation}
which implies that 
\begin{eqnarray*}
\gamma(\lambda+i\epsilon)-\gamma(\lambda)
&=&\int_{\bbR}\frac{1}{2}\log\left( 1+\frac{\epsilon^2}{(\ell-\lambda)^2} \right)\,\rmd \mathcal{N}(\ell)\\
&\geq& \frac{1}{2}(\mathcal{N}(\lambda+\epsilon)-\mathcal{N}(\lambda-\epsilon)).
\end{eqnarray*}
By \eqref{eq:complexExponent}, we deduce $\mathcal{N}(\lambda+\epsilon)-\mathcal{N}(\lambda-\epsilon)\leq c \epsilon^{1/2}$. 
Moreover, $\mathcal{N}$ is locally constant in the complement of $\Sigma_\varphi$, which means precisely that $\mathcal{N}$ is $\frac{1}{2}$-H\"{o}lder continuous.
\end{proof}

As a consequence of Proposition~\ref{thm.holderCon} and Theorem~\ref{thm.complexLE}, we obtain a non-degenerate property of the density of states measure $\mathcal{N}(\lambda)$, which is crucial for establishing purely absolutely continuous spectrum. An analogous result for Schr\"odinger operators was first developed by Avila in \cite{Avila1}, and for CMV matrices, it can be found in \cite{DLZ22}. 
Since the argument is standard, we leave it in the appendix.

\begin{lemma}\label{lem:lowerBoundDos}
There exist a constant $c>0$ such that for any $\lambda\in\Sigma_\varphi$ and sufficiently small $\epsilon>0$, 
$$\mathcal{N}(\lambda+\epsilon)-\mathcal{N}(\lambda-\epsilon)\geq c\epsilon^{\frac{3}{2}}.$$
\end{lemma}

Having assembled these ingredients, we can now finish the proof of absolutely continuous spectrum.

\begin{proof}[Proof of Theorem~\ref{thm:ac}] Let $\mathcal{B}$ be the set of spectral parameters $\lambda\in \Sigma_\varphi$ for which $$\sup_{x}\Vert U(\lambda,x,\theta)\Vert<\infty.$$
According to Theorem~\ref{thm.acMeas}, $\mu|_{\mathcal{B}}$ is absolutely continuous. We define $\mathcal{R}$ as the set of spectral parameters for which the system $(\omega, A_0+F_0(\cdot))$ is reducible. As per a well-known result by Eliasson \cite{eliasson}, the set $\mathcal{R}\setminus\mathcal{B}$ comprises spectral parameters where the system $(\omega, A_0(\lambda)+F_0)$ is conjugate to a parabolic system, corresponding to $2\rho(\lambda)=\langle k,\omega\rangle$ for some $k\in\bbZ^d$, hence this set is at most countable. This implies, in particular, that there are no eigenvalues in $\mathcal{R}$, thus $\mu(\mathcal{R}\setminus\mathcal{B})=0$. To complete the argument, it suffices to demonstrate that $\mu(\Sigma_\varphi\setminus\mathcal{R})=0$.

Let $\epsilon_j,r_j,N_j$ be defined by \eqref{eq:KAMqantities} and $J_{j}(\lambda)$ be an open $2^{\frac{2}{3}} \epsilon_{j-1}^{\frac{2}{45}}$-neighborhood of $\lambda \in K_{j}$  (recall that $K_j$ was defined in \eqref{e.Kjdef}). 
By Theorem~\ref{thm.acMeas} and Corollary~\ref{cor.transferGrowth} we have
\begin{align*}
\mu(J_{j}(\lambda))  \leq \sup_{0\leq s\leq C\epsilon_{j-1}^{-\frac{2}{45}}}\Vert U(\lambda,s,\theta)\Vert^{2}_{0}\vert J_{j}(\lambda)\vert 
 \leq \sup_{0\leq s\leq C\epsilon_{j-1}^{-\frac{1}{16}}}\|U(\lambda,s,\theta)\|^{2}_{0}\vert J_{j}(\lambda)\vert 
 \leq C\epsilon_{j-1}^{\frac{1}{48}},
\end{align*}
where $|\cdot|$ denotes Lebesgue measure. Take a finite subcover such that $\overline{K_{j}}\subseteq \bigcup_{l=0}^{r}J_{i}(\lambda_\ell)$. Refining this subcover if necessary, we may assume that every $\lambda\in\bbR$ is contained in at most 2 different $J_{j}(\lambda_\ell)$. By Lemma~\ref{lem:lowerBoundDos} and Lemma~\ref{lem:rotAndDensity},
$$
\vert \mathcal{N} (J_{j}(\lambda) \vert = \mathcal{N}(\lambda+2^{-\frac{1}{3}}\epsilon_{j-1}^{\frac{2}{45}})- \mathcal{N}
(\lambda-2^{-\frac{1}{3}}\epsilon_{j-1}^{\frac{2}{45}}) \geq 2 c \epsilon_{j-1}^{\frac{1}{15}}.
$$
Meanwhile,  by Lemma~\ref{lem:rotAndDensity} and  Lemma~\ref{lem:rotEst}, if $\lambda\in K_{j}$, we have
$$
\left\| 2\mathcal{N}(\lambda)-\langle n_{j},\omega\rangle\right\|_{\bbR/\bbZ}\leq 2\epsilon_{j-1}^{\frac{1}{15}}
$$
for some $\vert n_{j}\vert <2N_{j-1}$. This implies that $2\mathcal{N}(K_{j})$ can be covered by $2N_{j-1}$ open intervals $T_{s}$ of length $2\epsilon_{j-1}^{\frac{1}{15}}$. Since $\vert T_{s}\vert \leq \frac{1}{c}\vert 2\mathcal{N}(J_{j}(\lambda))\vert$ for any $s$, $\lambda_\ell\in K_{j}$, there are at most $2([\frac{1}{c}]+1)+4$ open intervals $J_{j}(\lambda_\ell)$ such that $2\mathcal{N}(J_{j}(\lambda_\ell))$ intersects $T_{s}$. We conclude that there are at most $2(2([\frac{1}{c}]+1)+4)N_{j-1}$ open intervals $J_{j}(\lambda_{\ell})$ to cover $K_{j}$. Then
\begin{equation}\label{boundOfMeas}
\mu(K_{j})\leq \sum_{j=0}^{r}\mu(J_{j}(\lambda_\ell)) \leq CN_{j-1}\epsilon_{j-1}^{\frac{1}{48}} \leq C\epsilon_{j-1}^{\frac{7}{384}}.
\end{equation}
Since $\epsilon_{j}=\epsilon_{0}^{2^{j}}$ and $\epsilon_{0}$ is small, \eqref{boundOfMeas} implies $
\sum\limits_{j} \mu(\overline{K_{j}}) < \infty.$
As $\Sigma_{\varphi} \backslash \mathcal{R} \subseteq \limsup K_{j}$, then Borel-Cantelli Lemma gives the desired result.

\end{proof}

\subsection{Gap Estimates}
In this subsection, we give the upper bounds of spectral gaps of $\Sigma_\varphi$ in terms of their labels provided by the Gap Labelling Theorem.

Let $\epsilon_0, h>0$ be  positive numbers and assume that $\Vert F_0\Vert_{h}=\sup_{|\Im z|<h}|F_0(z)|<\epsilon_0$ sufficiently small. Assume that $\lambda$ locates at the boundary of possible  spectral gaps, i.e. $2\rho(\lambda)= \langle k,\omega\rangle$, then by well-known result of Eliasson \cite{eliasson},  the system $(\omega,A_0+F_0(\cdot))$ is reducible to  parabolic constant coefficients. More precisely, there exists $B(\theta)\in C^\omega(\bbT^d,\mathrm{PSU}(1,1))$ such that 
\begin{equation}\label{eq.reducible}
\partial_{\omega}B=(A_0(\lambda)+F_0(\theta))B-BA
\end{equation}
where 
\[A=\frac{1}{2}\begin{bmatrix}i \zeta & -i \zeta\\ i \zeta & -i \zeta\end{bmatrix}\] and $\zeta$ is  a real constant. In our application, we will assume that $\lambda$ is the left endpoint of the open spectral gap and thus assume $\zeta>0.$ The case that $\lambda$ is the  a right endpoint can be handled similarly.

In the following, we will apply the classical Moser-P\"oschel argument \cite{MoserPoschel1984,Puig2006}. The above transformation $B$ turns the system $(\omega, A_0(\lambda+\Delta)+F_0(\cdot))$ into a new system $(\omega, A+\Delta P(\cdot))$.
To obtain the explicit expression of $P$,
$$\begin{aligned}
\partial_{\omega}B&=(A_0(\lambda+\Delta)+F_0(\theta))B-B(A+\Delta P(\theta))\\
&=(A_0(\lambda)+F_0(\theta))B-BA.
\end{aligned}$$
As a consequence, 
\begin{equation}
P(\theta)=B^{-1}(\theta)\begin{bmatrix}-i 
 & 0\\ 0 & i \end{bmatrix}B(\theta).
\end{equation}
Let $B(\theta)=\begin{bmatrix}b_{11}(\theta)&b_{12}(\theta)\\\overline{b_{12}(\theta)}&\overline{b_{11}(\theta)}\end{bmatrix}\in SU(1,1)$, where $|b_{11}(\theta)|^2-|b_{12}(\theta)|^2=1.$ Then by direct computations, we have
\begin{equation}\label{eq.PerturbationExp}P(\theta)=\begin{bmatrix}
-i (|b_{11}|^2+|b_{12}|^2) & -2i \overline{b_{11}}b_{12}\\
2 i  b_{11}\overline{b_{12}} & i (|b_{11}|^2+|b_{12}|^2)
\end{bmatrix}.\end{equation}

Obviously, for any $0<r\leq h$,  $\Vert P\Vert_{r}\leq 2\Vert B\Vert_h^2.$
Denote $[\cdot]$ the average of a function over $\bbT^d$ with respect to the Lebesgue measure. 
Recall $\kappa,\tau$ are the constants of the frequency $\omega$ in \eqref{eq.diophantine}, define
\begin{equation}\label{eq.uConstant}D_{\tau}=2^{3\tau+8}\Gamma(3\tau+1)\end{equation}
where $\Gamma(\cdot)$ denotes the $\Gamma$ function. We have the following averaging lemma, as the proof is quite standard, we again leave it in the Appendix for completeness: 
\begin{lemma}\label{lem.conjugationg} Assume that $\omega \in \DC_d(\kappa,\tau)$ and $|\Delta|<D_\tau^{-1}\kappa^3R^{3\tau+1}\Vert B\Vert_{R}^{-2}$ for some $R>0$, then
there exists $X\in C^\omega(\bbT^d,\mathrm{PSU}(1,1))$ such that 
$$\partial_{\omega}X=(A+\Delta P(\theta))X-X(\widetilde{A}+\Delta^2\widetilde{P}(\theta))$$
where \begin{equation}\label{eq.newConstants}
\begin{aligned}
\widetilde{A}&=A+\Delta[P]=\frac{1}{2}\begin{bmatrix}i \zeta&-i \zeta\\ i \zeta&-i \zeta\end{bmatrix}+\Delta\begin{bmatrix}
- i [|b_{11}|^2+|b_{12}|^2]&-2 i [\overline{b_{11}}b_{12}]\\
2 i [b_{11}\overline{b_{12}}]& i [|b_{11}|^2+|b_{12}|^2]
\end{bmatrix}\\
&=\frac{1}{2}\begin{bmatrix} i (\zeta-2\Delta[|b_{11}|^2+|b_{12}|^2])&- i (\zeta+4\Delta[\overline{b_{11}}b_{12}])\\
 i (\zeta+4\Delta[b_{11}\overline{b_{12}}])&- i (\zeta-2\Delta[|b_{11}|^2+|b_{12}|^2])
\end{bmatrix}.
\end{aligned}
\end{equation}

with the following estimates for $0<r<R$
\begin{equation}\label{eq.estimateZ}
\Vert X-id\Vert_{r}\leq 2D_\tau\kappa^{-3}R^{-(3\tau+1)}\Delta\Vert B\Vert^2_{R}
\end{equation}
and 
\begin{equation}\label{eq.estimateP}
\Vert \widetilde{P}\Vert_{\frac{R}{2}}\leq D^2_\tau\kappa^{-6}R^{-2(3\tau+1)}\Vert B\Vert^4_{R}.
\end{equation}
\end{lemma}

\bigskip

With Lemma~\ref{lem.conjugationg} in hand,
denote $d(\Delta)$ the determinant of $\widetilde{A}$ in \eqref{eq.newConstants}, then  direct computation gives 
\begin{equation}\label{eq.determinant}
d(\Delta)=
\Delta([|b_{11}|^2+|b_{12}|^2]^2-4|[b_{11}\overline{b_{12}}]|^2)\left(
\Delta-\zeta\frac{[|b_{11}|^2+|b_{12}|^2]+([b_{11}\overline{b_{12}}+\overline{b_{11}}b_{12}])}{[|b_{11}|^2+|b_{12}|^2]^2-4|[b_{11}\overline{b_{12}}]|^2}
\right).
\end{equation}
Clearly, by Cauchy-Schwartz inequality,  $$|b_{11}|^2+|b_{12}|^2+(b_{11}\overline{b_{12}}+\overline{b_{11}}b_{12})\leq \Vert B\Vert^2_{R}.$$
It is crucial to observe that the denominator $[|b_{11}|^2+|b_{12}|^2]^2-4|[b_{11}\overline{b_{12}}]|^2$ is  bounded from below.
Indeed, by the H\"older inequality, $$|[b_{11}\overline{b_{12}}]|^2\leq[|b_{11}|^2][|b_{12}|^2]$$
and since $B\in SU(1,1)$, $$[|b_{11}|^2+|b_{12}|^2]^2-4|[b_{11}\overline{b_{12}}]|^2\geq [|b_{11}|^2-|b_{12}|^2]^2=1.$$

Therefore, we have the following estimates:
\begin{equation}\label{eq.keyEstimate}
\frac{[|b_{11}|^2+|b_{12}|^2]+([b_{11}\overline{b_{12}}+\overline{b_{11}}b_{12}])}{[|b_{11}|^2+|b_{12}|^2]^2-4|[b_{11}\overline{b_{12}}]|^2}\leq \Vert B\Vert^2_{R}.
\end{equation}

Let us briefly describe the idea to obtain upper bounds of the spectral gaps, which was first developed in \cite{LYZZ}.  Assume that $\Vert B\Vert^2_{R}\leq \frac{1}{16}\zeta^{-\nu}$ and we will have $$\zeta\frac{[|b_{11}|^2+|b_{12}|^2]+([b_{11}\overline{b_{12}}+\overline{b_{11}}b_{12}])}{[|b_{11}|^2+|b_{12}|^2]^2-4|[b_{11}\overline{b_{12}}]|^2}\leq \frac{1}{16}\zeta^{1-\nu}.$$
So for  $\Delta\sim \frac{1}{16}\zeta^{1-\nu}$, $\lambda+\Delta$ will belong to the spectrum since $d(\Delta)>0$, which indicates the corresponding gap length is bounded above by $\frac{1}{16}\zeta^{1-\nu}$.
\begin{theorem}[Upper bounds of gaps]\label{Thm.upperBound}
Let $\omega\in \DC_d(\kappa,\tau)$ , $\nu\in(0,\frac{1}{4})$, assume that $\lambda$ is the left endpoint of a spectral gap $G$, and 
there exists $B\in C^\omega(\bbT^d,\mathrm{PSU}(1,1))$ such that
\[\partial_{\omega}B=(A_0+F_0(\theta))B-B\frac{1}{2}\begin{bmatrix} i \zeta&- i \zeta\\ i \zeta&- i \zeta\end{bmatrix}.\]
If for some $R>0$ \begin{equation}\label{eq.gapCondition}\Vert B\Vert_R^4\zeta^{\nu}<2^{-6}D_{\tau}^{-2}\kappa^{6}R^{2(3\tau+1)},\end{equation}
then \begin{equation}\label{eq.gapUpperB}|G|\leq \zeta^{1-\nu}.\end{equation}
\end{theorem}
\begin{proof}
Let $\Delta=\zeta^{1-\nu}$, applying Lemma~\ref{lem.conjugationg}, the system $(\omega,A_0(\lambda+\Delta)+F_0(\theta))$ is conjugated to $(\omega, \widetilde{A}+\Delta^2\widetilde{P}(\theta))$.
By \eqref{eq.determinant} and \eqref{eq.keyEstimate}, one has $d(\Delta)\geq \Delta(\Delta-\zeta\Vert B\Vert_{R}^2)$. Furthermore, by \eqref{eq.gapCondition}, one has $d(\Delta)\geq \frac{1}{4}\Delta^2>0$. Then, in view of \cite[Lemma 8.1]{HY2012}, there exists $U\in SU(1,1)$ with $\Vert U\Vert^2\leq 8\frac{\zeta+\Delta\Vert B\Vert_{R}^2}{\sqrt{d(\Delta)}}$ such that \begin{equation}U^{-1}\widetilde{A}U=\begin{bmatrix} i \sqrt{d(\Delta)}&0\\0&- i \sqrt{d(\Delta)}\end{bmatrix}.
\end{equation} Moreover, by \eqref{eq.estimateP} and \eqref{eq.gapCondition}, one has the estimate $$\Vert U^{-1}\Delta^2\widetilde{P}U\Vert_{\frac{R}{2}}\leq 8\Delta^2\frac{(\zeta+\Delta\Vert B\Vert_{R}^2)}{\sqrt{d(\Delta)}}D_{\tau}^2\kappa^{-6}R^{-2(3\tau+1)}\Vert B\Vert_{R}^4\leq \frac{1}{4}\Delta.$$ Furthermore, it follows immediately from the definition of rotation number that $$\left|\rho(\omega,U^{-1}(\widetilde{A}+\Delta^2\widetilde{P}(\theta))U)-\sqrt{d(\Delta)}\right| \leq \Vert U^{-1}\Delta^2\widetilde{P}U\Vert_0 \leq \frac{1}{4}\Delta,$$ which in turn implies that $$\rho(\omega,\widetilde{A}+\Delta^2\widetilde{P}(\theta))>\sqrt{d(\Delta)}- \frac{1}{4}\Delta> \frac{1}{4}\Delta>0,$$ therefore, the system $(\omega, \widetilde{A}+\Delta^2\widetilde{P}(\theta))$ is not uniformly hyperbolic and $\lambda+\zeta^{1-\nu}\in\Sigma_\varphi$. Since $\lambda$ is the left endpoint of $G$, $\lambda+\Delta$ crosses the right endpoint of $G$ and enters the spectrum, as a result, $|G|<\zeta^{1-\nu}$ and \eqref{eq.gapUpperB} follows.
\end{proof}

We need the following result from \cite [Theorem 3.1]{LYZZ}, parsed as a $\su(1,1)$ analogue.
\begin{lemma}
Let $\omega\in \DC_d(\kappa,\tau)$, $A_0\in \su(1,1)$, $h>0$. For any $r\in(0,h)$, there exists $\epsilon_*=\epsilon_*(d,\kappa,\tau,|A_0|,h,r)$ such that for any $F_0(\theta)\in C^\omega_{h}(\bbT^d,\su(1,1))$, if $2\rho(\omega,A_0+F_0)=\langle k,\omega\rangle \mod\bbZ $ and $\Vert F_0\Vert_{h}=\epsilon_0<\epsilon_*$, then  there exists $W\in C^\omega_{r}(\bbT^d,\mathrm{PSU}(1,1))$ such that 
\begin{equation}\label{eq.transformation}
\partial_{\omega}W=(A_0+F_0(\theta))W-W\frac{1}{2}\begin{bmatrix} i \zeta&- i \zeta\\ i \zeta&- i \zeta\end{bmatrix}
\end{equation}
with the following estimates
\begin{equation}\label{eq.mainEst1}
|\zeta|<\epsilon_0^{\frac{3}{4}}e^{-2\pi r|k|}
\end{equation}
and for any $r'\in (0,r]$
\begin{equation}\label{eq.mainEst2}
\Vert W\Vert_{r'}\leq De^{\frac{3\pi r'}{2}|k|},
\end{equation}
where $D=D(d,\kappa,\tau,|A_0|,h)$ is a constant.
\end{lemma}

Once we have these preparations, we obtain the following finer estimates on  the spectral gaps. To be precise, for any $k,k'\in\bbZ^d$ with $k\neq k'$,
let $G_k$ $G_{k'}$ be  spectral gaps with canonical labeling $k$, $k'$, then we have the following: 

\begin{coro}\label{gap-sp}Assume that $\omega\in \DC_d(\kappa,\tau)$ and $\|\widetilde{\varphi}\|_{h}<\epsilon_0$ for some $h>0$ and sufficiently small $\epsilon_0=\epsilon_0(d,\kappa,\tau,|A_0|,h)$.
Then the following  estimates holds for any $r\in (0,h)$
\begin{equation}\label{eq.gapBd}
|G_{k}|\leq \epsilon_0^{\frac{1}{2}}e^{-2\pi r |k|}.
\end{equation}
\begin{equation}\label{eq.gapDist}
\mathrm{dist}(G_k,G_{k'})\geq\frac{4^\tau\kappa^2}{c^2|k'-k|^{2\tau}}.
\end{equation}
\begin{equation}\label{lem:gapPositions}
\mathrm{dist}( \overline{G_k}, \frac{\langle k,\omega\rangle}{2} )\leq \epsilon_0.
\end{equation}

\end{coro}
\begin{proof}
 In order to apply above results, we only need to pick $\nu=\frac{\widetilde{r}-r}{4\widetilde{r}}$ and $R=\epsilon_0^{\frac{\widetilde{r}-r}{13(3\tau+1)\widetilde{r}}}$, where $\widetilde{r}=\frac{h+r}{2}$. Then for sufficiently small $\epsilon_0$, one can verify that 
$\zeta^\nu\Vert B\Vert^4_{R}<2^{-6}D_{\tau}^{-2}\kappa^6 R^{2(3\tau+1)}$
and therefore \eqref{eq.gapBd} follows from Theorem~\ref{Thm.upperBound}.

On the other hand,  denote $G_k=(\lambda^-_k,\lambda_k^+), G_{k'}=(\lambda_{k'}^-,\lambda_{k'}^+)$, and assume $\lambda_k^+<\lambda_{k'}^{-}$. As $\omega\in \DC_d(\kappa,\tau)$, then it follows from Theorem~\ref{thm.holderCon} that
\begin{equation*}
\mathrm{dist}(G_k,G_{k'})=\lambda_{k'}^--\lambda_{k}^+\geq c^{-2}|\rho(\lambda_{k'}^-)-\rho(\lambda_k^+)|^2\geq\frac{4^\tau\kappa^2}{c^2|k'-k|^{2\tau}}.
\end{equation*}

For any $\lambda \in \overline{G_k}$, where the rotation number of the system $(\omega, A_0(\lambda) + F_0(\theta))$ satisfies $\rho(\omega, A_0(\lambda) + F_0(\theta)) = \frac{\langle k, \omega \rangle}{2}$, it follows directly from the definition of the rotation number that \[ |\rho(\omega, A_0(\lambda) + F_0(\theta)) - \rho(\omega, A_0(\lambda))| \leq \|F_0\|_0 \leq \epsilon_0,\] which is equivalent to \[|\lambda - \frac{\langle k, \omega \rangle}{2}| \leq \epsilon_0.\] We thus complete the proof.
\end{proof}

We need the following result in the Appendix of \cite{BDGL}.
\begin{lemma}
    If $\sfE\subset\bbR$ satisfies the Craig type conditions \eqref{eq:craigCond1}-\eqref{eq:craigCond3}, then $\sfE$ is a Carleson homogeneous set in the sense of Definition \eqref{def:homogeneous}.
\end{lemma}
We would like to point out that the original proof in \cite{BDGL} works for the definition given by $$
\inf_{x\in\sfE}\inf_{0<h\leq 1}\frac{|\sfE\cap(x-h,x+h)|}{2h}>0.$$
It also works for Definition \ref{def:homogeneous} by removing the requirement ``$0<t<\frac{1}{2},\gamma_j<1$'' in their argument.

As a consequence, we can prove the following as the final result of this section:
\begin{theorem}[Homogeneous spectrum]\label{thm:homoSpectrum}
Assume that $\omega\in \DC_d(\kappa,\tau)$ and $\|\widetilde{\varphi}\|_{h}<\epsilon_0$ for some $h>0$ and sufficiently small $\epsilon_0=\epsilon_0(d,\kappa,\tau,h)$.
Then the spectrum of the Dirac operator $\Lambda_\varphi$ is homogeneous in the sense of Carleson. 
\end{theorem}
\begin{proof}
We only need to show that the spectrum satisfies the Craig type conditions.
Recall that $C_j$ is defined by \eqref{eq:Cj}. 
By \cite[Lemma 6.1]{BDGL}, $C_j$ grows at most subexponentially. According to \eqref{lem:gapPositions} of  Corollary~\ref{gap-sp},
$\eta_{0k}\leq C|k|$ for some positive constant $C$ that depends on $\Vert\varphi\Vert_\infty$ and $\omega$. We also have $\gamma_k<\epsilon_0^{\frac{1}{2}}e^{-2 \pi r|k|}$ and $\eta_{jk}\geq 4^\tau\kappa^2c^{-2}|k-j|^{-2\tau}$ by \eqref{eq.gapBd} and \eqref{eq.gapDist} respectively. Therefore, \eqref{eq:craigCond1} follows immediately. Note that $\eta^{-1}_{jk}\leq \kappa^{-2}c^2\max\{|k|^{2\tau},|j|^{2\tau}\}$, therefore, $\sum_kC_k\gamma_k\eta_{jk}^{-1}<\infty.$
Moreover, $\gamma_j(1+\eta_{0,j}^2)|j|^{2\tau}\to 0$ as $j\to\infty$, \eqref{eq:craigCond2} follows. \eqref{eq:craigCond3} holds for any $\delta<\frac{1}{2}$. 

\end{proof}

\section{Proofs of Main Results}
\label{sec:mainproof}

In this section, we combine the results developed in previous sections and prove the main theorems, that is, Theorems~\ref{thm:main} and~\ref{thm:main1}. We complete the proof in two separate steps. In the first step we prove the existence and uniqueness. In the second step, we prove the almost periodicity. 

The following result is an analogue of \cite[Proposition 4.3]{BDGL}.
\begin{prop}\label{prop:ExistenceUniqueness}
    Suppose $\sfE$ satisfies conditions \eqref{eq:craigCond1}-\eqref{eq:craigCond3}. For any $f\in \mathcal{D}(\sfE)$, there exists a unique function $y:\bbR^2\to \mathcal{D}(\sfE)$ such that $y(0,0)=f$ and
    $$\partial_xy=\Psi(y),\quad \partial_t y=\Xi(y).$$
    Let $\varphi(x,t)$ be determined by $$\Re\varphi(x,t)=-\frac{1}{2}Q_1(y(x,t)),\quad \Im\varphi(x,t)=-\frac{1}{2}Q_1(y(x,t;i))$$
    where $y(x,t;i)$ is the Dirichlet data of $-i\mathcal{B}^{-1}(f)$. Then $\varphi(x,t)$ solves the NLS. For each $t\in\bbR$, $\varphi(x,t)\in\mathcal{R}(\sfE)$ and $\mathcal{B}(\varphi(x,0))=f.$
\end{prop}
\begin{proof}
    The proof follows the same argument as that of \cite[Proposition 4.3]{BDGL} with cited results \cite[Theorem 1.48, Lemma 1.16]{GH03} replaced by \cite[Theorem 3.37, Lemma 3.8]{GH03}. Let us give a brief sketch to highlight the different places. Let $\sfE_{N},\Omega_N,\mathcal{R}_N$ be chosen as in the proof of Theorem \ref{thm:approxFreqCharFlow} and Theorem \ref{thm:linearzRotFlow}. Let $\varphi^N(x,t)$ be the finite gap solution associated to the Dirichlet data $y^N\in \mathcal{D}(\sfE_N)$ and vector fields $\Psi^N,\Xi^N$. Introduce $\mathbf{p}:\mathcal{D}(\sfE)\to\mathcal{D}(\sfE_N)$ by $\mathbf{p}(y)(j)=y_j$,  $|j|\leq N.$ Let $\tilde{\Psi}^N, \tilde{\Xi}^N$ be vector fields on $\mathcal{D}(\sfE)$ defined by 
    $$
        \tilde{\Psi}^N_j=\left\{
        \begin{aligned}
        &\Psi^N_j(\mathbf{p}(y))&|j|\leq N,\\
        &\Psi_j(\mathbf{p}(y))&|j|>N.
        \end{aligned}\right.
    $$

    $$
    \tilde{\Xi}^N_j=\left\{
    \begin{aligned}
    &\Xi_j^N(\mathbf{p}(y))&|j|\leq N,\\
    &\Xi_j(\mathbf{p}(y))&|j|>N.
    \end{aligned}
    \right.
    $$
    Let $\tilde{y}^N(x,t)$ be determined in the following order. Define $\tilde{y}^N(0,0)=f.$ Then for $x=0$ solve 
    $$\partial_t\tilde{y}^N=\tilde{\Xi}^N(\tilde{y}^N),~\forall t\in\bbR.$$
    Then for each $t$ solve 
    $$\partial_x\tilde{y}^N=\tilde{\Psi}^N(\tilde{y}^N),~\forall x\in\bbR.$$
    Let $\varphi^N$ be obtained by applying the trace formula to $\tilde{y}^N.$ Then the spectrum of $\Lambda_{\varphi^N}$ equals $\sfE_N$ as a known fact of the finite gap construction; compare \cite{BBEIM1994,GH03}.
    It can be shown that $\tilde{y}^N\to y$ uniformly on compacts of $(x,t)\in\bbR^2.$ It follows that $\varphi^N(x,0)\to \varphi(x,0)$ and $\varphi^N(x,t)\to\varphi(x,t)$ uniformly on compacts. Each application of trace formula in \cite{BDGL} should be replaced by applications of trace formula \eqref{eq:traceEqRe} separately for $\mathcal{B}(\varphi)$ and $\mathcal{B}(-i\varphi)$.
    Recall that \cite[Lemma 3.8]{GH03} says if $\varphi^N$ solves the stationary AKNS hierarchy, then its Dirichlet data obeys the corresponding Dubrovin system of differential equations. Theorem 3.37 of \cite{GH03} says if $\varphi^N$ arises from the trace formula, then it solves the corresponding AKNS hierarchy. Therefore, we should replace \cite[Theorem 1.48]{GH03},\cite[Lemma 1.16]{GH03} in \cite{BDGL} by \cite[Theorem 3.33]{GH03}, \cite[Lemma 3.8]{GH03} respectively.

    Since $\varphi^N\to\varphi$ uniformly on compacts, $\Lambda_{\varphi^N}\to\Lambda_\varphi$ in the sense of strong resolvent. This implies the uniform convergence of the diagonal Green's function on compacts. Since $\sfE=\cap_N\sfE_N$, by Lemma 5.2 of \cite{Craig89}, which is a discussion for general Herglotz functions, $\varphi(x,0)\in\mathcal{R}(\sfE)$ and for each $t$, $\varphi(x,t)\in\mathcal{R}(\sfE)$.

    The claim $\mathcal{B}(\varphi)=f$ follows from the analysis of the behavior of the diagonal Green's function on each gap $(a_j,b_j), |j|\leq N$ and the uniform convergence on compacts as $N\to\infty.$
    Uniqueness is a consequence of the Lipschitz continuity of the vector fields $\Psi,\Xi.$
\end{proof}

We have established existence and uniqueness. Let us now prove almost periodicity by showing the following result.

\begin{prop}\label{prop:uniformHomeo}
    Suppose that  $\sfE\subset\bbR$ is a closed set satisfying the Craig type conditions \eqref{eq:craigCond1}-\eqref{eq:craigCond3}. Then the map $\mathcal{M}=\mathcal{B}^{-1}\circ \mathcal{A}^{-1}:\pi_1(\Omega)^*\times\bbT\to \mathcal{R}(\sfE)$ is a homeomorphism if $\mathcal{R}(\sfE)$ is equipped with the uniform topology.
\end{prop}
\begin{proof}
    First note that the Craig type conditions implies that $\sfE$ is homogeneous, the finite gap length condition, $k$-th GLC for $k=1,2$ and that $\Omega=\bbC\setminus\sfE$ is a regular PWS. 
    We need to show that for any $(\alpha^n,\tau^n)$ that converges to $(\alpha,\tau)$ with respect to the metric induced by $$\Vert(\alpha,\tau)\Vert=\sum_{j\geq 1}\frac{1}{2^j}|\alpha_j|_\bbT+|\tau|,$$
    $\varphi_n(x,t)=\mathcal{B}^{-1}\circ\mathcal{A}^{-1}((\alpha^n(x,t),\tau^n(x,t)))\to \varphi(x,t)=\mathcal{B}^{-1}\circ\mathcal{A}^{-1}((\alpha(x,t),\tau(x,t)))$ with respect to the metric 
    $$\Vert\varphi\Vert_0=\sup_{(x,t)\in\bbR^2}|\varphi(x,t)|.$$
    Let $$\begin{aligned}&(\alpha^n(x,t),\tau^n(x,t))=(\alpha^n+\eta x+\eta^{(1)}t,\tau^n+\vartheta_0 x+\vartheta_1t)\\
    &(\alpha(x,t),\tau(x,t))=(\alpha +\eta x+\eta^{(1)}t,\tau+\vartheta_0 x+\vartheta_1 t)
    \end{aligned}$$
    be the linear flows of of Theorem \ref{thm:approxFreqCharFlow} and Theorem \ref{thm:linearzRotFlow}. It follows that 
    $$\Vert (\alpha^n,\tau^n)-(\alpha,\tau)\Vert=\Vert (\alpha^n(x,t),\tau^n(x,t))-(\alpha(x,t),\tau(x,t))\Vert \quad \text{ for all } x,t\in\bbR.$$ 
    Let $y^n(x,t)=\mathcal{A}^{-1}(\alpha^n(x,t),\tau^n(x,t)), y(x,t)=\mathcal{A}^{-1}(\alpha(x,t),\tau(x,t))$. Since $\pi_1(\Omega)^*\times\bbT$ is a compact metric space and $\mathcal{A}^{-1}$ is continuous, for any $\epsilon>0$ there exists $\delta>0$ such that for sufficiently large $N>0$, for any $n>N$
    $$\Vert (\alpha^n,\tau^n)-(\alpha,\tau)\Vert<\delta$$
    implies $$\Vert y^n-y\Vert<\epsilon, \text{ namely } \sup_{j}\gamma_j^{1/2}|y^n_j-y_j|_\bbT<\epsilon.$$
    It follows from the uniform continuity of $\mathcal{A}^{-1}$ that $$\sup_{(x,t)\in\bbR^2}\Vert y^n(x,t)-y(x,t)\Vert<\epsilon .$$
    By the finite gap length condition, trace formula \eqref{eq:traceEqRe} is uniformly and absolutely convergent. For any $\epsilon'>0$ pick $M>0,\epsilon>0$ such that $\sum_{|j|>M}4\pi\gamma_j<\frac{1}{2}\epsilon'$ and $2\pi M\epsilon<\epsilon'/2$.
    Since $$|2\mu^n_j-2\mu_j|\leq 2\pi\gamma_j|y^n_j-y_j|_\bbT$$
     for any  $\mu_j^n,\mu_j\in (a_j,b_j)$,
    it follows that $$\sup_{(x,t)\in\bbR^2}|\varphi^n(x,t)-\varphi(x,t)|<\epsilon'.$$
\end{proof}
\begin{proof}[Proof of Theorem \ref{thm:main}]
     For any given $\varphi(x,0)=\varphi(x)$, let $f=\mathcal{B}(\varphi)\in\mathcal{D}(\sfE)$. Let $\varphi(x,t)$ be given by Proposition \ref{prop:ExistenceUniqueness}. This establishes existence and uniqueness.  Since $\mathcal{M}=\mathcal{B}^{-1}\circ\mathcal{A}^{-1}:\pi_1(\Omega)^*\times\bbT\to\mathcal{R}(\sfE)$ is a continuous map with $\mathcal{R}(\sfE)$ equipped with uniform topology, it follows that $\varphi(x,t)$ is uniformly almost periodic.
\end{proof}

Theorem \ref{thm:main1} follows from Theorem \ref{thm:main} by verifying the Craig type conditions. This is shown in Theorem \ref{thm:homoSpectrum}.

\begin{appendix}

\section{Proofs of Preparatory Lemmas}

\subsection{Proof of Lemma~\ref{lem:lowerBoundDos}}

The proof mainly follows the proof of \cite[Lemma 3.11]{Avila1}, we include the proof for completeness.
Let $\delta=c\epsilon^{\frac{3}{2}}$,
since $\gamma(\lambda)=0$ by Theorem~\ref{thm:zeroLyapunov}, we have $\gamma(\lambda+i\delta)=\gamma(\lambda+i\delta)-\gamma(\lambda)$. By the Thouless formula \eqref{eq.Thouless},
$$\gamma(\lambda+i\delta)
=\int_{\bbR}\log\left\vert\frac{\ell-(\lambda+i\delta)}{\ell-\lambda}\right\vert \, \rmd \mathcal{N}(\ell)
=\int_{\bbR}\frac{1}{2}\log\left(1+\frac{\delta^2}{(\ell-\lambda)^2}\right) \,\rmd \mathcal{N}(\ell).$$
We split the integral into four parts: $I_{1}=\int_{\left|\ell-\lambda\right|\ge 1}$, $ I_{2}=\int_{\epsilon\le\left|\ell-\lambda\right|<1 }$, $I_{3}=\int_{\epsilon^{4}\le\left|\ell-\lambda\right|<\epsilon}$ and $I_{4}=\int_{\left|\ell-\lambda\right|<\epsilon^{4}}$. We clearly have $I_1 <c^2\epsilon^3$.

For sufficiently small $\epsilon>0$, by Proposition~\ref{thm.holderCon} we have $I_1 <\frac{2c^2\epsilon^3}{\delta_0^2},$
and
\begin{equation*}
\begin{split}
I_{4}=\sum_{k \geq 4} \int_{\epsilon^{k}>\left|\ell-\lambda\right|\ge\epsilon^{k+1}}  \frac{1}{2}\ln (1+\frac{\delta^{2}}{\left|\ell-\lambda\right|^{2}})\, \rmd \mathcal{N}(\ell) 
\leq \frac{1}{2}\sum_{k \geq 4} \epsilon^{\frac{k}{2}} \ln (1+c^{2} \epsilon^{1-2 k})\leq \epsilon^{\frac{7}{4}}.
\end{split}
\end{equation*}
We also have the estimate
$$
\begin{aligned} I_{2}  \leq \sum_{k=0}^{m} \int_{e^{-k-1}\le|\ell-\lambda|<e^{-k}} \frac{1}{2}\ln (1+\frac{\delta^{2}}{\left|\ell-\lambda\right|^{2}}) \, \rmd \mathcal{N}(\ell)\le\sum_{k=0}^{m}\frac{1}{2}e^{-\frac{k}{2}}\delta^2e^{2k+2}\leq Cc^2\delta,\end{aligned}
$$
with $m=[-\ln \epsilon]$. 
It follows that
\begin{align*}
I_3\ge \gamma(\lambda+i\delta)-C\delta.
\end{align*}

Proposition~\ref{thm.complexLE} implies that $\gamma(\lambda+i\delta) \geq\delta/10$ for sufficiently small $\delta>0$. Since the constant $c$ above is consistent with our choice of $\delta$, we can shrink it such that $I_3\ge\frac{1}{20}\delta$. Since $I_{3} \leq C(\mathcal{N}(\lambda+\epsilon)-\mathcal{N}(\lambda-\epsilon)) \ln \epsilon^{-1},$
the result follows.

\subsection{Proof of Lemma~\ref{lem.conjugationg}}
Let $Y=\sum_{n\in\bbZ}\hat{Y}(n)e^{ i \langle n,x\rangle}\in\su(1,1),$ where $\hat{Y}(n)=\begin{bmatrix}\hat{Y}_{11}(n)&\hat{Y}_{12}(n)\\\hat{Y}_{21}(n)&-\hat{Y}_{11}(n)\end{bmatrix}$ being determined by $$\partial_{\omega}Y=[A,Y]+\Delta(P(x)-[P]).$$
By comparing Fourier coefficients of both sides, one can solve 
\begin{equation}\label{eq.zCoefficients}
\left\{
\begin{aligned}
&\hat{Y}_{11}(n)=\frac{ i \zeta\Delta}{2\langle n,\omega\rangle^3}[(\langle n,\omega\rangle-\zeta)\hat{P}_{21}(n)+(\langle n,\omega \rangle+\zeta)\hat{P}_{12}(n)]+\frac{\langle n,\omega\rangle^2-\zeta^2}{ i \langle n,\omega\rangle^3}\hat{P}_{11}(n)\\
&\hat{Y}_{12}(n)=\frac{ i \Delta}{2\langle n,\omega\rangle^3}[-2\zeta(\langle n,\omega\rangle+\zeta)\hat{P}_{11}(n)+\zeta^2\hat{P}_{21}(n)+(2\langle n,\omega\rangle^2+\zeta^2+2\zeta\langle n,\omega\rangle)\hat{P}_{12}(n)]\\
&\hat{Y}_{21}(n)=\frac{ i \Delta}{2\langle n,\omega\rangle^3}[-2\zeta(\langle n,\omega\rangle-\zeta)\hat{P}_{11}(n)+\zeta^2\hat{P}_{12}(n)+(2\zeta\langle n,\omega\rangle-2\langle n,\omega\rangle^2-\zeta^2)\hat{P}_{21}(n)]
\end{aligned}
\right.
\end{equation}
Define $X=\mathrm{exp}(Y)$. As a consequence of \eqref{eq.zCoefficients},  for $\omega\in \DC(\kappa,\tau)$ we have 
$$\begin{aligned}\Vert Y\Vert_{\frac{R}{2}}&=\sum_{n\in\bbZ^d}|\hat{Y}(n)|e^{\frac{Rn}{2}}
\leq 20\Delta \Vert P\Vert_{R}\sum_{n\in\bbZ}\frac{e^{-\frac{nR}{2}}}{|\langle n,\omega\rangle|^3}\leq 20\Delta\kappa^{-3}\Vert P\Vert_{R}\sum_{n\in\bbZ}e^{-\frac{nR}{2}}|n|^{3\tau}\\
&\leq 40\kappa^{-3}\Delta\Vert P\Vert_{R}\int_{0}^{\infty}x^{3\tau}e^{-\frac{Rx}{2}} \, \rmd x.\end{aligned}$$
The integral in the equation has an upper bound $2^{3\tau+1}\Gamma(3\tau+1)R^{-(3\tau+1)}$. Therefore, we obtain 
$$\Vert Y\Vert_{\frac{R}{2}}\leq 40\kappa^{-3}R^{-(3\tau+1)}\Delta\Vert P\Vert_{R}2^{3\tau+1}\Gamma(3\tau+1)\leq \frac{1}{2}D_{\tau}\kappa^{-3}R^{-(3\tau+1)}\Delta \Vert P\Vert_{R},$$
and
$$\Vert X-id\Vert_{\frac{R}{2}}\leq 2\Vert Y\Vert_{\frac{R}{2}}\leq D_{\tau}\kappa^{-3}R^{-(3\tau+1)}\Delta\Vert P\Vert_{R}.$$
One can read $\Vert P\Vert_R\leq 2\Vert B\Vert_{R}^2$ off \eqref{eq.PerturbationExp}, and therefore $$\Vert X-id\Vert_{\frac{R}{2}}\leq 2\Vert Y\Vert_{\frac{R}{2}}\leq 2D_{\tau}\kappa^{-3}R^{-(3\tau+1)}\Delta\Vert B\Vert^2_{R}$$
and 
moreover, we can determine 
\begin{equation}
\widetilde{P}_{1}=\sum_{m+n\geq 2}\frac{(-Y)^n}{n!}A\frac{Y^m}{m!}+\sum_{m+n\geq 1}\frac{(-Y)^n}{n!}\Delta P(x)\frac{Y^m}{m!}.
\end{equation}
As a result of straightforward computations, $$\Vert \widetilde{P}_{1}\Vert_{\frac{R}{2}}\leq D_{\tau}^2\kappa^{-6}R^{-2(3\tau+1)}\Delta^2\Vert B\Vert^4_{R}. $$
Define $\widetilde{P}$ such that $\Delta^2 \widetilde{P}=\widetilde{P}_{1},$ then it follows that $$\Vert \widetilde{P}\Vert_{\frac{R}{2}}\leq D^2_\tau\kappa^{-6}R^{-2(3\tau+1)}\Vert B\Vert^4_{R}$$
and \eqref{eq.estimateP} follows.

\section{Non-Stationary AKNS Hierarchy} \label{sec:AKNS}
Let $U(z)=\begin{bmatrix}-iz&p\\q&iz\end{bmatrix}$ and
$\widetilde{V}_{n+1}$ is defined as follows:
$$\widetilde{V}_{n+1}=i\begin{bmatrix}-\widetilde{G}_{n+1}(z)&\widetilde{F}_{n}(z)\\-\widetilde{H}_{n}(z)&\widetilde{G}_{n+1}(z)\end{bmatrix},$$
\begin{equation}\label{eq:AKNSCoefficient}\begin{aligned}
&\widetilde{F}_{n}(z)=\sum_{s=0}^{n}\widetilde{f}_{n-s}z^{s},~\widetilde{f}_{0}=-iq,\\
&\widetilde{G}_{n+1}=\sum_{s=0}^{n+1}\widetilde{g}_{n+1-s}z^s,~\widetilde{G}_0=1,\\
&\widetilde{H}_n=\sum_{s=0}^n\widetilde{h}_{n-s}z^s,~\widetilde{H}_0=ip,
\end{aligned}\end{equation}
with the time dependent coefficients $\{\widetilde{f}_\ell,\widetilde{g}_\ell,\widetilde{h}_\ell\}_{\ell\in\bbN_0}$ recursively defined as
\begin{equation}\label{eq:AKNSCoefficient1}\begin{aligned}
&\widetilde{f}_0=-iq,~\widetilde{g}_0=1,~\widetilde{h}_0=ip,\\
&\widetilde{f}_{\ell+1}=(i/2)f_{\ell,x}-iq\widetilde{g}_{\ell+1},~\ell\in\bbN_0,\\
&\widetilde{g}_{\ell+1,x}=p\widetilde{f}_\ell+q\widetilde{h}_\ell,~\ell\in\bbN_0,\\
&\widetilde{h}_{\ell+1}=-(i/2)\widetilde{h}_{\ell,x}+ip\widetilde{g}_{\ell+1},~\ell\in\bbN_0;
\end{aligned}\end{equation}

The first a few terms can be computed explicitly as follows
\begin{equation}\label{eq:InitialTerms}
\begin{aligned}
&\widetilde{f}_0=-iq,~\widetilde{f}_1=\frac{1}{2}\partial_x q+c_1(-iq),\\
&\widetilde{f}_2=\frac{i}{4}\partial^2_x q-\frac{i}{2}pq^2+c_1(\frac{1}{2}\partial_x q)+c_2(-iq);\\
&\widetilde{g}_0=1,~\widetilde{g}_1=c_1,\\
&\widetilde{g}_2=\frac{1}{2}pq+c_2;\\
&\widetilde{h}_0=ip,~\widetilde{h}_1=\frac{1}{2}\partial_x p+c_1(ip),\\
&\widetilde{h}_2=-\frac{i}{4}\partial_x^2p+\frac{i}{2}p^2q+c_1(\frac{1}{2}\partial_x p)+c_2(ip).
\end{aligned}
\end{equation}
The constants $\{c_i:i\in\bbN_0\}$ with convention $c_0=1$ are integration constants, the case $c_i=0$ for all $i\in\bbN$ is called {\it homogeneous} AKNS hierarchy. $\widetilde{V}_{n+1}, U$ are differential expressions of order $n+1$ and $1$.
Interested reader may consult to Section 3.2 and Section 3.4 of \cite{GH03} for more detailed explanations.
The time dependent AKNS hierarchy now reads as the following zero-curvature equation
\begin{equation}\label{eq:zeroCur}
\begin{aligned}
U_{t_n}-\widetilde{V}_{n+1,x}+[U,\widetilde{V}_{n+1}]&=0\\
-V_{n+1,x}+[U,V_{n+1}]&=0,
\end{aligned}
\end{equation}
where $U,V_{n+1}$ are the corresponding stationary version of \eqref{eq:AKNSCoefficient}, \eqref{eq:AKNSCoefficient1}. In particular, \eqref{eq:nls} corresponds to the case $n=1$ and $c_1=c_2=0$.

\end{appendix}



\end{document}